\documentclass[a4paper,10pt]{article}
\usepackage[english]{babel}
\usepackage{lmodern}
\usepackage{latexsym}
\usepackage{a4}
\usepackage{amsthm,amssymb,amsfonts,amscd,amsmath,mathtools}
\usepackage{eucal}
\usepackage{tikz}
\DeclareMathAlphabet{\pazocal}{OMS}{zplm}{m}{n}
\usepackage{pgfplots}
\usepackage{graphicx}
\usepackage{stmaryrd}
\usepackage{enumerate}

\usepackage[backend=biber,isbn=false,giveninits=true,doi=false,sortcites]{biblatex}
\usepackage{csquotes}
\usepackage[bookmarks=true, bookmarksopen=true]{hyperref}
\usepackage{microtype}
\usetikzlibrary{arrows.meta}
\usetikzlibrary{decorations.pathreplacing,matrix,calc,arrows,shapes.geometric}
\pgfplotsset{compat=1.18}
\allowdisplaybreaks

\newcommand{\dlgraffa}{\{ \hspace{-0.1cm} \{}
\newcommand{\drgraffa}{\} \hspace{-0.1cm} \}}
\newcommand{\vertiii}[1]{{\left\vert\kern-0.25ex\left\vert\kern-0.25ex\left\vert #1 \right\vert\kern-0.25ex\right\vert\kern-0.25ex\right\vert}}

\newcommand{\IR}{{\mathbb R}}
\newcommand{\IC}{{\mathbb C}}
\newcommand{\IN}{{\mathbb N}}
\newcommand{\IZ}{{\mathbb Z}}
\newcommand{\bx}{{\mathbf x}}
\newcommand{\bd}{{\mathbf d}}
\newcommand{\bn}{{\mathbf n}}
\newcommand{\bE}{{\mathbf E}}

\newcommand{\ri}{{\mathrm i}}
\newcommand{\re}{{\mathrm e}}
\newcommand{\rd}{{\,\mathrm d}}
\newcommand{\inc}{{\mathrm{inc}}}

\newcommand{\scat}{{\mathrm{scat}}}
\newcommand{\Mh}{{\mathcal M_h}}
\newcommand{\cA}{{\mathcal A}}
\newcommand{\pa}{{\mathtt{a}}}
\newcommand{\pb}{{\mathtt{b}}}
\newcommand{\pd}{{\mathtt{d}}}
\newcommand{\dn}{{\partial_\bn}}
\newcommand{\dnK}{{\partial_{\bn_K}}}
\newcommand{\dxo}{{\partial_{x_1}}}
\newcommand{\dxt}{{\partial_{x_2}}}
\newcommand{\cu}{{\overline{u}}}
\newcommand{\cv}{{\overline{v}}}
\newcommand{\cw}{{\overline{w}}}
\newcommand{\GD}{{\Gamma_D}}
\newcommand{\GH}{{\Gamma_H}}
\newcommand{\GmH}{{\Gamma_{-H}}}
\newcommand{\GpmH}{{\Gamma_{\pm H}}}
\newcommand{\Gleft}{\Gamma_{\mathrm{left}}}
\newcommand{\Gright}{\Gamma_{\mathrm{right}}}
\newcommand{\Fh}{\pazocal{F}_h}
\newcommand{\FhI}{\pazocal{F}_h^I}

\graphicspath{{figs/}}

\hypersetup{pdftitle={Trefftz Discontinuous Galerkin methods for scattering by periodic structures},  
	pdfauthor={Monforte, Moiola},}

\definecolor{myblue}{rgb}{0,0,0.6}     
\hypersetup{colorlinks=true, linkcolor=myblue,  citecolor=myblue, filecolor=myblue,   urlcolor=myblue,  }

\title{Trefftz Discontinuous Galerkin methods for scattering by periodic structures} 
\author{Armando Maria Monforte\thanks{Department of Mathematics, University of Pavia, Italy
		(\href{mailto:armandomaria.monforte01@universitadipavia.it}{armandomaria.monforte01@universitadipavia.it}), ORCID: 0009-0000-7687-2217},
	Andrea Moiola\thanks{Department of Mathematics, University of Pavia, Italy 
		(\href{mailto:andrea.moiola@unipv.it}{andrea.moiola@unipv.it}), ORCID: 0000-0002-6251-4440}	
}
\date{\today}

\begin{document}
	\newtheorem{thm}{Theorem}[section]
	\newtheorem{prop}[thm]{Proposition}
	\newtheorem{lem}[thm]{Lemma}
	\newtheorem{corol}[thm]{Corollary}
	\newtheorem{defin}[thm]{Definition}
	
	\newtheorem{rem}[thm]{Remark}
	\newtheorem{example}[thm]{Example}

	\maketitle 
	
	\section*{Abstract}

	We propose a Trefftz discontinuous Galerkin (TDG) method for the approximation of plane wave scattering by periodic diffraction gratings, modelled by the two-dimensional Helmholtz equation.
	The periodic obstacle may include penetrable and impenetrable regions.
	The TDG method requires the approximation of the Dirichlet-to-Neumann (DtN) operator on the periodic cell faces, and relies on plane wave discrete spaces.
	For polygonal meshes, all linear-system entries can be computed analytically.
	Using a Rellich identity, we prove a new explicit stability estimate for the Helmholtz solution, which is robust in the small material jump limit.
	
	\bigskip\noindent
	{\bf Keywords:} \; Diffraction grating, Quasi-periodic, Helmholtz equation, Rellich identity, Discontinuous Galerkin, Trefftz method, Plane wave basis
	
	\bigskip\noindent
	{\bf Mathematics Subject Classification (2020):} \;
	65N30,
	35J05,  
	35Q60, 
	78A45, 
	78M10 
	
	\section{Introduction}
	The electromagnetic scattering by periodic structures has been an area of significant interest in computational electromagnetics for many decades \cite{BaoLi2022,BonnetBD1994}.
	The numerical simulation of such problems typically consists in the truncation of the computational domain to a bounded cell, and in the approximation of a time-harmonic boundary value problem (BVP) with appropriate boundary conditions.
	We thus propose the use of a numerical scheme that has been successfully employed for other time-harmonic problems: the Trefftz discontinuous Galerkin method (TDG) \cite{Hiptmair2016,Hiptmair2011}.
	The TDG approximates the unknown field, on a finite-element mesh, with a discrete space spanned by elementwise solutions of the PDE under consideration.
	When plane wave (complex exponentials) basis functions and polygonal meshes are used, it is possible to obtain efficient quadrature-free system assembly and very accurate solutions.
	In particular, convergence may enjoy faster rates than for piecewise-polynomial discrete spaces, see \cite{Moiola2011} and Remark~\ref{rem:PWapprox} below.
	
	We consider obstacles that are periodic in one direction (say $x_1$) and invariant under translation in another direction (say $x_3$).
	The scattering of an electromagnetic plane wave, with electric field parallel to the grating direction $x_3$, is described by the two-dimensional Helmholtz equation in the $(x_1,x_2)$ plane \cite[\S1.3]{BaoLi2022}.
	We allow both perfect electric conductor (PEC) scatterers, leading to Dirichlet impenetrable obstacles, and dielectric media, leading to piecewise-constant, possibly complex-valued, material parameters.
	The scattering problem can be formulated as a Helmholtz BVP posed in a bounded cell, with quasi-periodic boundary conditions on two sides, and a Dirichlet-to-Neumann (DtN) condition on the rest of the boundary.
	These conditions exactly replace the Sommerfeld radiation condition.
	
	The well-posedness and the stability analysis of this class of problems is delicate, since some configurations can support guided modes and non-unique solutions \cite{BonnetBD1994}.
	We prove a new stability bound on the solution under some assumptions: that the obstacle is non trapping as in \cite[Theorem~3.5]{BonnetBD1994}, and that the wavenumber is not a Rayleigh--Wood anomaly, i.e.\ the DtN operators are injective.
	The stability estimate \eqref{stimaH1} thus obtained is fully explicit in all parameters.
	This result is related to some previous ones, notably those in \cite{Lechleiter2010,Zhu2024,Civiletti2020}, but includes configurations not covered by these references as detailed in Remark~\ref{rem:StabilityCompare}.
	
	We describe in detail the application of the TDG to the quasi-periodic Helmholtz problem with DtN boundary conditions.
	We recall that the TDG is an extension and reformulation of the ultra weak variational formulation (UWVF) \cite{Cessenat1998,Huttunen2002}.
	To effectively take into account the discontinuous material coefficients, we follow the numerical flux definition proposed in \cite{Howarth2014}.
	The simple complex-exponential expression of the plane wave basis functions allows to compute all integrals arising in the linear system assembly analytically.
	This includes the computation of the Fourier coefficients needed for the implementation of a truncated DtN map.
	We briefly describe the approximation properties of plane waves, their numerical instabilities, and possible remedies based on evanescent plane waves in Remark~\ref{rem:PWapprox}.
	The DtN-TDG method proposed resembles that of \cite{Kapita2018}, which addresses the scattering by bounded obstacles on a computational domain truncated with a DtN map, and that of \cite{Monk2024}, which addresses acoustic waveguides.
	
	This paper is structured as follows.
	In \S\ref{secmod}, we describe the BVP of interest, recalling known definitions and results on quasi-periodic function spaces and DtN operators, including their truncation.
	In \S\ref{s:Stability}, we derive the simple Rellich identity \eqref{rellich}; we show well-posedness under non-trapping conditions in Theorem~\ref{wellposed} extending the proof in \cite{BonnetBD1994} to the Rayleigh--Wood anomaly case; we derive the explicit stability estimate \eqref{stimaH1}.
	In \S\ref{s:TDG}, we derive the DtN-TDG following \cite{Hiptmair2011,Howarth2014,Kapita2018}; we show that the method is coercive, well-posed, and its solution satisfies a quasi-optimality bound; and we describe in detail the matrix and load-vector assembly.
	Finally, in \S\ref{s:Numerics}, we report several numerical results involving smooth and singular solutions, ill-posed BVPs, penetrable and impenetrable obstacles.
	The method has been implemented in MATLAB and the code is freely available online.
	
	\section{Model problem} \label{secmod}
	
	We present the problem of the scattering of a plane wave by a grating, i.e.\ a periodic structure.
	This problem has been studied in depth with different approaches in many articles and books such as \cite{BonnetBD1994,Civiletti2020,Kirsch1993}.
	
	We consider linear optics with $\re^{-\ri \omega t}$ dependence on time $t$, where $\ri = \sqrt{-1}$ and $\omega$ is the wave angular frequency. 
	This assumption agrees with the convention in \cite{Civiletti2020,Huttunen2002,Howarth2014}, 
	while in \cite{Hiptmair2016,Hiptmair2011,Kapita2018} the time-dependence is implicitly assumed to be $\re^{\ri \omega t}$.
	A difference between these two notations is that, with the first one, a plane wave with expression $\re^{\ri k \bx \cdot \bd}$ propagates in the direction of the unit vector $\bd$, while with the second notation the wave propagates in the opposite direction $-\bd$. This difference in the convention leads to opposite signs in the radiation conditions and in the impedance boundary conditions. 
	
	We consider electric fields in the form $\mathbf E(x_1,x_2,x_3)=(0,0,E_3(x_1,x_2))$ solving time-harmonic Maxwell equations in materials that are invariant in the direction $x_3$.
	This setting is called ``transverse-electric'' (TE) in \cite[\S1.3]{BaoLi2022} and \cite{Zhu2024}, ``transverse-magnetic'' (TM) in \cite{BonnetBD1994}, ``electric mode'' in \cite{Lechleiter2010}, ``$s$-polarized'' in \cite{Civiletti2020}.
	Maxwell equations then reduce to the two-dimensional Helmholtz equation in the variables $x_1,x_2$ for the component $u=E_3$, \cite[eq.~(2.5)]{BonnetBD1994}.
	
	\subsection{Domain, material parameters and truncation}
	\begin{figure}
		\centering
		\begin{tikzpicture}[scale=0.87,transform shape]
			\draw[-{Stealth[slant=0]}] (1.7,6.1) -- (2.3,5.3);
			\draw[dashed] (1.5,5.4) -- (2.25,5.4);
			\draw[dashed] (0,0) -- (0,5.5);
			\draw[dashed] (4,0) -- (4,5.5);
			\draw (1.82,5.4) .. controls (1.85,5.6) and (1.9,5.65) .. (2.0, 5.7);
			\node[] at (1.7,5.65) {$\theta$};
			\fill[gray!40!white, draw=black, rotate around={45:(3,1.45)}] (3,1.95) rectangle (4,2.95);
			\draw (0,3.5) -- (4,3.5);
			\draw (0,4.2) -- (1,4.5) -- (3,3.9) -- (4,4.2);
			\draw (0,0.5) -- (0.67,0.5) -- (0.67,1) -- (1.33,1) -- (1.33,1.5) -- (2.67,1.5) -- (2.67,1) -- (3.33,1) -- (3.33,0.5) -- (4,0.5);
			\fill[gray!40!white, draw=black, rotate around={45:(-1,1.45)}] (-1,1.95) rectangle (0,2.95);
			\draw (-4,3.5) -- (0,3.5);
			\draw (-4,4.2) -- (-3,4.5) -- (-1,3.9) -- (0,4.2);
			\draw (-4,0.5) -- (-3.33,0.5) -- (-3.33,1) -- (-2.67,1) -- (-2.67,1.5) -- (-1.33,1.5) -- (-1.33,1) -- (-0.67,1) -- (-0.67,0.5) -- (0,0.5);
			\fill[gray!40!white, draw=black, rotate around={45:(7,1.45)}] (7,1.95) rectangle (8,2.95);
			\draw (4,3.5) -- (8,3.5);
			\draw (4,4.2) -- (5,4.5) -- (7,3.9) -- (8,4.2);
			\draw (4,0.5) -- (4.67,0.5) -- (4.67,1) -- (5.33,1) -- (5.33,1.5) -- (6.67,1.5) -- (6.67,1) -- (7.33,1) -- (7.33,0.5) -- (8,0.5);
			\node[] at (2.6,2.5) {$D$};
			\node[] at (6.6,2.5) {$D$};
			\node[] at (-1.4,2.5) {$D$};
			\node[] at (-3,5) {$\Omega_0$};
			\draw[decorate, decoration={brace,amplitude=10}] (4,-0.2)--(0,-0.2);
			\draw(2,-.8)node{$L$};
			\draw[decorate, decoration={brace,amplitude=10}] (-4.2,0)--(-4.2,5);
			\draw(-4.9,2.5)node{$2H$};
		\end{tikzpicture}  
		\caption{Geometry of the periodic scattering region $\Omega_0$ and the Dirichlet obstacle $D$.
			Continuous lines separate regions of $\Omega_0=\IR^2\setminus\overline D$ with constant permittivity $\varepsilon$.}
		\label{fig:scatt}
	\end{figure}
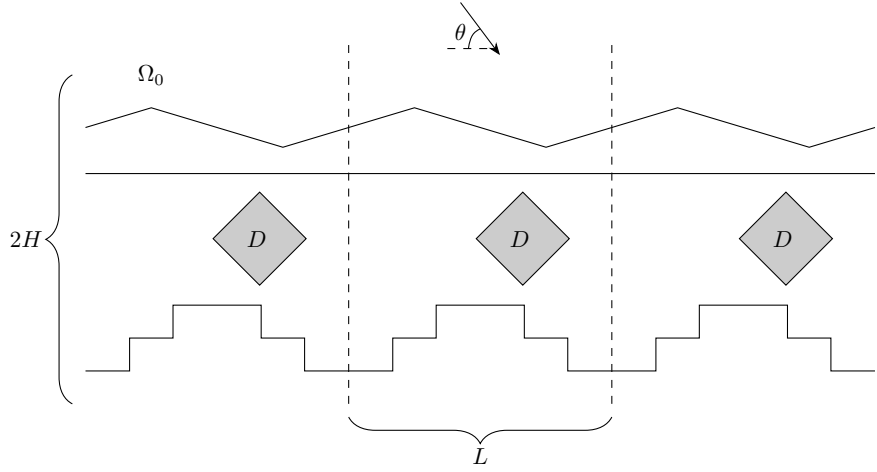

	Let $L$ and $H$ be positive parameters denoting the scatterer space period and height, respectively.
	The impenetrable (perfect electric conductor, PEC) obstacle is represented by an open Lipschitz set $D\subset\IR^2$, possibly empty, that is periodic in the $x_1$-direction with period $L>0$ (i.e.\ $(x_1,x_2)\in D\Leftrightarrow (x_1+L,x_2)\in D$) and bounded in the $x_2$ direction so that $\overline D\subset \{|x_2|<H\}$.
	The scattering region is $\Omega_0:=\IR^2\setminus \overline D$, and we assume it to be connected.
	
	Let $\varepsilon_0,\mu_0>0$ denote the permittivity and permeability of vacuum.
	The relative 
	permittivity of the medium is represented by the complex-valued, piecewise-constant function $\varepsilon\in L^\infty(\Omega_0)$ with $\Re(\varepsilon) > 0$ and $\Im(\varepsilon) \geq 0$.
	We assume that $\varepsilon$ is periodic with period $L$ in $x_1$, and that the inhomogeneity is bounded in $x_2$ i.e.\  $\varepsilon(\bx)=\varepsilon^+>0$ in $\{x_2>H\}$ and $\varepsilon(\bx)=\varepsilon^-$ in $\{x_2<-H\}$.
	We assume that the relative magnetic permeability is constant $\mu=1$ in $\Omega_0$.
	
	A possible scattering region is depicted in Figure \ref{fig:scatt}:
	the shaded region represents the impenetrable scatterer $D$;
	the continuous lines represent the interfaces, periodic in $x_1$, separating different materials, i.e.\ regions with constant $\varepsilon$.
	
	We use the parameters $L,H$ to define a truncated domain taking into account the periodicity in $x_1$ and the boundedness in $x_2$ of the scatterer:
	\begin{equation}\label{eq:Omega}
		\Omega:=\{\bx\in\Omega_0:\; 0<x_1<L,\; -H<x_2<H\}.
	\end{equation}
	We introduce also some notation for the parts of the boundary of $\Omega$ and the upper/lower regions:
	\begin{align*}
		\GpmH:=&\{\bx\in\IR^2: 0\le x_1\le L,\; x_2=\pm H \},\\
		\Gleft:=&\{\bx\in\IR^2:  x_1=0,\; -H\le x_2\le H \},\\
		\Gright:=&\{\bx\in\IR^2:  x_1=L,\; -H\le x_2\le H \},\\
		\GD:=& \partial D\cap\overline\Omega,\\
		\Omega_H^+:=&\{\bx\in\IR^2:   0\le x_1\le L,\; x_2> H \},\\
		\Omega_H^-:=&\{\bx\in\IR^2:   0\le x_1\le L,\; x_2<- H \},
	\end{align*}
	so that $\partial\Omega=\GH\cup\GmH\cup\Gleft\cup\Gright\cup\GD$ and $\varepsilon|_{\Omega_H^\pm}=\varepsilon^\pm$.
	We denote by $\bn$ the outward-pointing unit normal on $\partial\Omega$.
	
	In Figure \ref{fig:region} we see the domain $\Omega$ obtained from the truncation of the region $\Omega_0$ in Figure~\ref{fig:scatt}.
	A few other geometries in this framework are described and depicted in the numerical experiments of \S\ref{s:Numerics}
	
	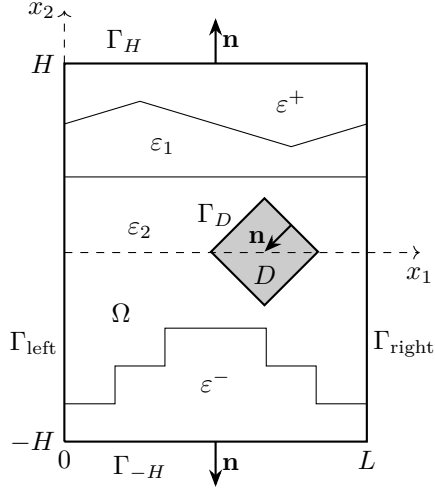
\begin{figure}
		\centering
		\begin{tikzpicture}
			\node[] at (4.7,2.2) {$x_1$};
			\node[] at (-0.3,5.7) {$x_2$};
			\node[] at (-0.4,1.3) {$\Gamma_{\mathrm{left}}$};
			\node[] at (4.5,1.3) {$\Gamma_{\mathrm{right}}$};
			\node[] at (0.8, 5.3) {$\Gamma_H$};
			\node[] at (1, -0.35) {$\Gamma_{-H}$};
			\node[] at (3,4.5) {$\varepsilon^+$};
			\node[] at (1.3,3.9) {$\varepsilon_1$};
			\node[] at (1,2.8) {$\varepsilon_2$};
			\node[] at (2,3) {$\Gamma_D$};
			\node[] at (2,0.8) {$\varepsilon^-$};
			\draw [dashed,->](0,5) -- (0,5.7); 
			\draw [-{Stealth[slant=0]}, thick](2,5) -- (2,5.6);
			\node[] at (2.2,5.3) {$\mathbf{n}$};
			\draw [-{Stealth[slant=0]}, thick](2,0) -- (2,-0.6);
			\node[] at (2.2,-0.3) {$\mathbf{n}$};
			\fill[gray!40!white, draw=black, thick, rotate around={45:(3,1.45)}] (3,1.95) rectangle (4,2.95);
			\draw[-{Stealth[slant=0]}, thick, rotate around={45:(3,1.45)}](4,2.45)--(3.5,2.45);
			\draw [dashed,->](0,2.5) -- (4.7,2.5); 
			\node[] at (2.55,2.7) {$\mathbf{n}$};
			\draw (0,3.5) -- (4,3.5);
			\draw (0,4.2) -- (1,4.5) -- (3,3.9) -- (4,4.2);
			\draw (0,0.5) -- (0.67,0.5) -- (0.67,1) -- (1.33,1) -- (1.33,1.5) -- (2.67,1.5) -- (2.67,1) -- (3.33,1) -- (3.33,0.5) -- (4,0.5);
			\draw[thick] (0,0) -- (4,0) -- (4,5) -- (0,5) -- (0,0);
			\node[] at (2.65,2.2) {$D$};
			\draw (0,0) node[left]{$-H$};
			\draw (0,5) node[left]{$H$};
			\draw (0,0) node[below]{$0$};
			\draw (4,0) node[below]{$L$};
			\draw (1,1.7) node[left]{$\Omega$};
		\end{tikzpicture}
		\caption{The truncated domain $\Omega=(0,L)\times(-H,H)\setminus \overline D$. The relative permittivity assumes the values $\varepsilon^+$, $\varepsilon^-$, $\varepsilon_1$ and $\varepsilon_2$ in the four regions delimited by the continuous lines.}
		\label{fig:region}
	\end{figure}
	
	\subsection{Helmholtz equation and quasi-periodic conditions}
	The source of the scattering problem is a downward-propagating plane wave:
	\begin{equation}\label{incident}
		\bE^\inc:=(0,0,u^\inc),\qquad
		u^\inc(\bx):=u^\inc(x_1,x_2)=\exp\{\ri\kappa^+(x_1 \cos \theta + x_2 \sin \theta)\},
	\end{equation} 
	where 
	\begin{align*}
		\theta \in \ & [-\pi, 0] && \text{is the wave propagation angle against the horizontal},\\
		\kappa^\pm:=\ &k\sqrt{\varepsilon^\pm} &&\text{is the wavenumber in the upper/lower region }\Omega_H^{\pm},\\
		k:=\ &\omega/c_0 &&\text{is the wavenumber in free space},\\
		\omega>\ &0 &&\text{is the angular frequency, and}\\
		c_0:=\ & 1/\sqrt{\varepsilon_0\mu_0} &&\text{is the light speed in free space.}
	\end{align*}
	We define the piecewise-constant wavenumber function $\kappa(\bx):=k\sqrt{\varepsilon(\bx)}\in L^\infty(\Omega_0)$. 
	
	We denote by $u$ the third component of the total electric field $\mathbf E=(0,0,u)$ generated by the scattering of $\bE^\inc$ on $D$ and on the heterogeneities of $\varepsilon$.
	From the Maxwell equations $\mathrm{curl}(\mu_0^{-1}\mathrm{curl}\,\bE)-\omega^2\varepsilon\varepsilon_0\bE=\mathbf0$, it is classical (e.g.\ \cite{BonnetBD1994,Civiletti2020}) that $u$ satisfies the Helmholtz equation 
	$$
	\Delta u+k^2\varepsilon u=0 \qquad \text{in }\Omega_0,
	$$
	which has to be complemented with appropriate boundary, radiation and periodicity condition.
	
	We assume that the obstacle $D$ is a perfect electric conductor (PEC), so we impose the boundary condition $u=0$ on $\partial D$.
	
	While the domain and the material parameters are periodic in the $x_1$ direction with period $L$, the incoming wave $u^\inc$ is periodic in the same direction with period $2\pi/(\kappa^+\cos\theta)$, which is in general different from $L$.
	However, we observe that $u^\inc$ satisfies the relation $u^\inc(x_1+L,x_2)=\re^{\ri\kappa^+L\cos\theta}u^\inc(x_1,x_2)$.
	This suggests the following classical definition \cite{Pinto2020}.
	
	\begin{defin}[Quasi-periodic function]\label{def:QP}
		A function $u\in \mathcal{C}^0(\IR^2)$ is called quasi-periodic in $\IR^2$, of period $L$, with parameter $\alpha_0>0$, if
		\begin{equation*}
			u(x_1+L,x_2)=\re^{\ri\alpha_0 L}u(x_1,x_2) \qquad \forall \bx = (x_1,x_2) \in \IR^2.
		\end{equation*}
		A function $u\in \mathcal{C}^0(\Omega_0)$ is quasi-periodic in $\Omega_0$ if the same relation holds for all $\bx\in\Omega_0$.
	\end{defin} 
	For a quasi-periodic function $u$, it holds that 
	$u(x_1+nL,x_2)=\re^{\ri n\alpha_0 L}u(x_1,x_2)$ for all $n\in\IZ$ and $\bx$ in its domain.
	The quasi-periodicity of $u$ is equivalent to the periodicity with period $L$ of the function $x_1 \mapsto \re^{-\ri\alpha_0 x_1}u(x_1,x_2)$.
	
	Since $u^\inc$ is quasi-periodic with parameter $\alpha_0=\kappa^+\cos\theta$, also $u$ is quasi-periodic with the same parameter.
	To enforce this property, we restrict the Helmholtz equation to the bounded domain $\Omega$ in \eqref{eq:Omega} and impose the following conditions on the vertical sides $\Gleft$, $\Gright$:
	\begin{align*}
		\begin{aligned}
			u(L,x_2) =&\ \re^{\ri\alpha_0 L}u(0,x_2) \\ 
			\dxo(L,x_2) =&\ \re^{\ri\alpha_0 L}\dxo(0,x_2) 
		\end{aligned}
		\qquad |x_2|<H, \; (0,x_2)\notin\overline D.
	\end{align*}
	
	We are left to impose a radiation condition on $\GpmH$ to ensure that the scattered field propagates away from the scatterer $D$ and the inhomogeneity of $\varepsilon$.
	To this purpose, we need to introduce the space of quasi-periodic functions and the Dirichlet-to-Neumann operator.
	
	\subsection{Quasi-periodic function spaces}
	We introduce the quasi-periodic Sobolev spaces on $\Omega$ and on the horizontal boundaries $\GpmH$.
	
	Following \cite[\S3.1]{BonnetBD1994}, for $\alpha_0>0$, we introduce the following function spaces:
	\begin{enumerate}
		\item $\mathcal{C}_{\alpha_0}^\infty(\mathbb{R}^2)$ is the set of all functions that are $\mathcal{C}^\infty$ on $\mathbb{R}^2$, quasi-periodic with parameter $\alpha_0$, and vanish for large $|x_2|$;
		\item $\mathcal{C}_{\alpha_0}^\infty(\Omega)$ is the set of the restrictions to $\Omega$ of all functions of $\mathcal{C}_{\alpha_0}^\infty(\mathbb{R}^2)$;
		\item $H^1_{\alpha_0}(\Omega)$ is the smallest  closed subspace of $H^1(\Omega)$ that contains $\mathcal{C}_{\alpha_0}^\infty(\Omega)$;
		\item $H^1_{\alpha_0, 0}(\Omega)$ is the space of all functions $u \in H^1_{\alpha_0}(\Omega)$ 
		whose trace on $\GD$ vanishes.
	\end{enumerate}
	The spaces $H^1_{\alpha_0}(\Omega)$ and  $H^1_{\alpha_0,0}(\Omega)$ are Hilbert spaces with the usual $H^1(\Omega)$ norm and inner product.
	See \cite{Pinto2020} for more general quasi-periodic Sobolev spaces and their properties.
	
	In the following, we will omit explicit notation for the Dirichlet traces of $H^1_{\alpha_0}(\Omega)$ functions on parts of $\partial\Omega$.
	
	The elements of $\mathcal{C}_{\alpha_0}^\infty(\mathbb{R}^2)$ may formally be written as Fourier series; we refer to \cite[Prop.~2.6]{Pinto2020} for the proof of the following proposition.
	\begin{prop}[Fourier expansion] \label{fourier}
		Every $u \in \mathcal{C}_{\alpha_0}^\infty(\mathbb{R}^2)$ may be represented as a Fourier series, i.e.
		\begin{equation*}
			u(x_1, x_2) = \sum_{n \in \mathbb{Z}} u_n(x_2) \re^{\ri \alpha_n x_1}, \hspace{0.3cm} \text{where} \hspace{0.2cm} \alpha_n := \alpha_0 + \frac{2\pi n}{L} \hspace{0.2cm} \text{for} \hspace{0.2cm} n \in \mathbb{Z}.
		\end{equation*}
		The coefficients $u_n$ are defined as
		\begin{equation} \label{fourier_coeff}
			u_n(x_2) := \frac{1}{L} \int_0^{L} \re^{-\ri\alpha_n x_1} u(x_1,x_2) \, \rd x_1, \hspace{0.2cm} \text{for} \hspace{0.2cm} n \in \mathbb{Z}.
		\end{equation}
	\end{prop} 
	
	We introduce the quasi-periodic fractional Sobolev spaces on one-dimensional boundaries, which will be used to define the Dirichlet-to-Neumann operator:
	\begin{equation*}
		H^{1/2}_{\alpha_0}(\GpmH) := \left\{ v \in L^2(\GpmH), \hspace{0.15cm} v(x_1) = \sum_{n \in \mathbb{Z}} v_n \re^{\ri \alpha_n x_1} \; \biggl| \; \sum_{n \in \mathbb{Z}} (1+\alpha_n^2)^{1/2} \, |v_n|^2 < \infty \right\}.
	\end{equation*}
	$H^{1/2}_{\alpha_0}(\GpmH)$ is a closed subspace of the usual Sobolev space $H^{1/2}(\GpmH)$ and the norm
	\begin{equation*}
		\| v \|_{1/2, \alpha_0}^2 := L \sum_{n \in \mathbb{Z}} (1+\alpha_n^2)^{1/2} \, |v_n|^2,
	\end{equation*}
	is equivalent to the classical $H^{1/2}(\GpmH)$-norm. 
	Its dual space is
	\begin{equation*}
		H^{-1/2}_{\alpha_0}(\GpmH) := \left\{ v(x_1) = \sum_{n \in \mathbb{Z}} v_n \re^{\ri \alpha_n x_1} \hspace{0.15cm} \biggl| \hspace{0.15cm} \sum_{n \in \mathbb{Z}} (1+\alpha_n^2)^{-1/2} \, |v_n|^2 < \infty \right\},
	\end{equation*}
	and the associated norm is
	\begin{equation*}
		\| v \|_{-1/2, \alpha_0}^2 := L \sum_{n \in \mathbb{Z}} (1+\alpha_n^2)^{-1/2} \, |v_n|^2.
	\end{equation*}
	It can be proved that $H^{1/2}_{\alpha_0}(\GpmH)$ is the space of the traces on $\GpmH$ of all functions of $H^1_{\alpha_0}(\Omega)$; see \cite[Lemma~2.27]{Pinto2020}.
	The duality product between $H^{1/2}_{\alpha_0}(\GpmH)$ and $H^{-1/2}_{\alpha_0}(\GpmH)$ is given by 
	\begin{equation*}
		\langle u , v \rangle_{\alpha_0, \GpmH} = L \sum_{n \in \mathbb{Z}} u_n
		\overline{v_n} 
	,\qquad\text{for }
	u(x_1)= \sum_{n \in \mathbb{Z}} u_n \re^{\ri \alpha_n x_1}
	, \quad v(x_1)= \sum_{n \in \mathbb{Z}} v_n \re^{\ri \alpha_n x_1}
	.
\end{equation*}
More generally, arbitrary-order quasi-periodic Sobolev spaces can be defined:
\begin{align}\label{eq:HsNorm}
	H^s_{\alpha_0}(\GpmH) := 
	\bigg\{ v(x_1) = \sum_{n \in \mathbb{Z}} v_n \re^{\ri \alpha_n x_1} \; \biggl| \;
	\|v\|_{s,\alpha_0}^2:=L\sum_{n \in \mathbb{Z}} (1+\alpha_n^2)^s \, |v_n|^2 < \infty \bigg\}
	\qquad\forall s\in\IR.
\end{align}

\subsection{Dirichlet-to-Neumann operators}

We specify a radiation condition on $\GpmH$, as we want that the scattered field propagates upward on $\GH$ and the total field propagates downward on $\GmH$; this is equivalent to asking that these two boundaries are transparent to waves propagating away from the grating.
To specify this condition we make use of Dirichlet-to-Neumann operators.

We first focus on the boundary condition on $\GH$. 
Since $u$ is $\alpha_0$-quasi-periodic in $\Omega$, we can write
\begin{equation*}
u(\bx)=\sum_{n\in\mathbb{Z}} u_n(x_2) \, \re^{\ri\alpha_n x_1}, \hspace{0.2cm} \text{for} \hspace{0.2cm} 
\bx\in \GH,
\end{equation*}
with $\alpha_n$ and $u_n(x_2)$ as in Proposition \ref{fourier} and $\alpha_0=\kappa^+\cos\theta$. 
In $\Omega_H^+$ we write the total field $u$ as the sum of the incident field \eqref{incident} and a scattered field:
\begin{equation*}
u = u^\inc+u^\scat.
\end{equation*}
(In $\Omega$ and $\Omega_H^-$ we do not use this decomposition.)
Since both $u$ and $u^\inc$ are smooth, quasi-periodic solutions of the Helmholtz equation in $\Omega_H^+$, also their difference $u^\scat$ enjoys the same properties, thus
\begin{equation*}
u^\scat(\bx)=\sum_{n\in\mathbb{Z}} u^\scat_n(x_2) \, \re^{\ri\alpha_n x_1}, \hspace{0.2cm} \bx\in \Omega_H^+
\end{equation*}
for $u^\scat_n\in \mathcal C^\infty(H,\infty)$. From
$\Delta u^\scat + k^2 \varepsilon^+ u^\scat = 0$ for $x_2 > H$, the Fourier coefficients of $u^\scat$ must solve the differential equation
\begin{equation*}
{\partial^2_{x_2} u^\scat_n} + (k^2 \varepsilon^+ - \alpha^2_n) \, u^\scat_n = 0 \hspace{0.3cm} \text{for}
\hspace{0.2cm} x_2 > H.
\end{equation*}
Since each of these equations admits a 2-dimensional solution space, we select a particular solution as follows:
\begin{itemize}
\item[(i)] if $(k^2 \varepsilon^+ - \alpha^2_n)<0$, we select the exponentially decaying solution
\begin{equation*}
	u^\scat_n(x_2) = u_n^\scat(H) \re^{-\sqrt{\alpha^2_n-k^2 \varepsilon^+}(x_2-H)},
\end{equation*}
\item[(ii)] if $(k^2 \varepsilon^+ - \alpha^2_n)=0$, we choose the constant
\begin{equation*}
	u^\scat_n(x_2) = u_n^\scat(H),
\end{equation*}
\item[(iii)] if $(k^2 \varepsilon^+ - \alpha^2_n)>0$, we opt for the solution corresponding to an outgoing wave
\begin{equation*}
	u^\scat_n(x_2) = u_n^\scat(H) \re^{\ri\sqrt{k^2 \varepsilon^+-\alpha^2_n}(x_2-H)}.
\end{equation*}
\end{itemize}
The wave terms corresponding to case (iii) propagate upwards in $\Omega_H^+$ because of the $\re^{-\ri\omega t}$ time convention stipulated.
With these choices, the scattered field has the form
\begin{equation} \label{representation}
u^\scat(\bx)=\sum_{n\in\mathbb{Z}} u^\scat_n(H)\re^{\ri\beta_n^+ (x_2-H)} \re^{\ri\alpha_n x_1},    \qquad \bx\in\Omega_H^+ \cup \GH,
\end{equation}
where
\begin{equation} \label{beta_piu}
\beta_n^+ := \begin{cases} \sqrt{k^2\varepsilon^+-\alpha_n^2} & \alpha_n^2 \leq k^2\varepsilon^+, \\ 
	\ri\sqrt{\alpha_n^2-k^2\varepsilon^+} & \alpha_n^2 > k^2\varepsilon^+. \end{cases} 
	\end{equation}
	In particular, $\beta_0^+=-\kappa^+\sin\theta$ is real and positive.
	The normal derivative on $\GH$ of $u^\scat$ in \eqref{representation} is
	\begin{equation}\label{eq:dusdn}
\dxt u^\scat (x_1)
= \ri\sum_{n\in\mathbb{Z}} u^\scat_n(H) \beta_n^+ \re^{\ri\alpha_nx_1}.
\end{equation}
We want to enforce the relation between the value \eqref{representation} of $u^\scat$ and that of its normal derivative \eqref{eq:dusdn} in the formulation of the scattering problem,
so we define the Dirichlet-to-Neumann (DtN) operator $T^+$ as
\begin{align} \label{DtN}
& T^+:H^{1/2}_{\alpha_0}(\GH) \to H^{-1/2}_{\alpha_0}(\GH), \\ \nonumber
& (T^+\phi)(x_1) := \ri \sum_{n\in\mathbb{Z}} \phi_n \beta_n^+\re^{\ri\alpha_nx_1}, 
\hspace{0.5cm} \text{for} \hspace{0.2cm} 
\phi(x_1)=\sum_{n\in\IZ}\phi_n \re^{\ri\alpha_nx_1} \in H^{1/2}(\GH).
\end{align}

We say that a Helmholtz solution $u^\scat$ in $\Omega_H^+$ propagates upwards---equivalently, that it satisfies the radiation condition in $\Omega_H^+$---if 
$\gamma_H(\dxt u^\scat)=T^+(\gamma_H u^\scat)$, where $\gamma_H$ is the trace on $\GH$. 
In this case, $u^\scat$ admits the expansion \eqref{representation}.

The linear growth of the sequence $n\mapsto|\beta_n^+|$ in \eqref{beta_piu} (recall the definition of $\alpha_n$ in Proposition~\ref{fourier}) gives the following continuity result.
\begin{lem}[{\cite[Lemma 3.7]{Pinto2020}}]
The operator $T^+:H^{1/2}_{\alpha_0}(\GH) \to H^{-1/2}_{\alpha_0}(\GH)$ is continuous.
\end{lem}
More generally, recalling \eqref{eq:HsNorm}, $T^+:H^s_{\alpha_0}(\GH) \to H^{s-1}_{\alpha_0}(\GH)$ is continuous for all $s\in\IR$.

We derive the expression of the DtN operator $T^-$ on $\GmH$.
In this case, we require the total field $u$---without subtracting the incoming wave $u^\inc$---to propagate downwards. 
Using a Fourier expansion and imposing that the total field solves the Helmholtz equation in $\Omega_H^-$ with $\varepsilon=\varepsilon^-$, we look for $u$ in the form
\begin{equation*}
u(\bx)=\sum_{n\in\mathbb{Z}} u_n(-H)\re^{-\ri\beta_n^- (x_2+H)} \re^{\ri\alpha_n x_1},    \qquad \bx\in\Omega_H^-\cup\GmH.
\end{equation*}
Selecting the solutions similarly to (i)--(iii) above, if $\varepsilon^->0$ we obtain the coefficients
$\beta_n^-$ as
\begin{equation} \label{beta_meno}
\beta_n^- := \begin{cases} \sqrt{k^2\varepsilon^- - \alpha_n^2} & \alpha_n^2 \leq k^2\varepsilon^-, \\ 
	\ri\sqrt{\alpha_n^2-k^2\varepsilon^-} & \alpha_n^2 > k^2\varepsilon^-. \end{cases} 
	\end{equation}
	Instead, if $\varepsilon^-\notin\IR$, then the lower region $\Omega_H^-$ contains an absorbing medium, so all outgoing solutions decay exponentially for $x_2\to-\infty$.
	In this case we set 
	\begin{equation}\label{eq:betaComplEps}
\beta_n^-:=\sqrt{k^2\varepsilon^--\alpha_n^2} 
\end{equation}
choosing the complex root with positive imaginary part, i.e.\ $\Im\beta_n^->0$.
Since $k>0,\alpha_n\in\IR$ and $\Im\varepsilon^->0$, this is the standard branch cut of the square root, and we also have $\Re\beta_n^->0$.
Then, the DtN operator $T^-$ is
\begin{align} \label{DtN_meno}
& T^-:H^{1/2}_{\alpha_0}(\GmH) \to H^{-1/2}_{\alpha_0}(\GmH), \\ \nonumber
& (T^- \phi)(x_1) := \ri \sum_{n\in\mathbb{Z}} \phi_n \beta_n^- \re^{\ri\alpha_n x_1}, \hspace{0.5cm} \text{for} \hspace{0.2cm} 
\phi(x_1)=\sum_{n\in\IZ}\phi_n  \re^{\ri\alpha_nx_1} \in H^{1/2}_{\alpha_0}(\GmH),
\end{align}
and enjoys the same continuity property of $T^+$.
Note that the parameters $\alpha_n=k\sqrt{\varepsilon^+}\cos\theta+\frac{2\pi n}L$ enter the definition of the DtN operator $T^-$ on the lower boundary $\GmH$, but they depend on the value $\varepsilon^+$ of the material parameter $\varepsilon$ in the upper region $\Omega_H^+$, while $\varepsilon^-$ enters $T^-$ via the $\beta_n^-$ coefficients \eqref{beta_meno}--\eqref{eq:betaComplEps}.

In practical computations, we need to truncate the infinite series in the definitions of the DtN operators (\ref{DtN}) and (\ref{DtN_meno}), so we introduce the following operators

\begin{defin}[Truncated DtN operator]\label{def:TruncatedDtN}
For $M\in\IN$, we set
\begin{equation} \label{DtN_trunc}
	(T_M^\pm \phi)(x_1) := \ri \sum_{n=-M}^M \phi_n \beta_n^\pm \re^{\ri\alpha_nx_1},
	\qquad\forall \phi(x_1)=\sum_{n\in\IZ}\phi_n\re^{\ri\alpha_nx_1} \in H^{1/2}_{\alpha_0}(\GpmH).
\end{equation}
\end{defin}

\begin{lem} \label{lemma:l2prod}
Assuming $\varepsilon$ is real and positive, for every $w \in H^1_{\alpha_0} (\Omega)$,
\begin{align}\label{eq:ImT}
	\Im \int_\GpmH T^\pm w \, \cw  \rd s =
	L\sum_{\alpha^2_n < k^2\varepsilon^\pm} |w_n(\pm H)|^2 \sqrt{k^2\varepsilon^\pm-\alpha_n^2} \geq 0,\\ \label{eq:ReT}
	\Re \int_\GpmH T^\pm w \, \cw  \rd s =
	-L\sum_{\alpha^2_n > k^2\varepsilon^\pm} |w_n(\pm H)|^2 \sqrt{\alpha_n^2-k^2\varepsilon^\pm} \leq 0.
\end{align}
Moreover, for every $M \in \IN$,
\begin{align}
	\Im \int_\GpmH T^\pm w \, \cw  \rd s \geq \Im \int_\GpmH T_M^\pm w \, \cw  \rd s \geq 0,\\ 
	\Re \int_\GpmH T^\pm w \, \cw  \rd s \leq \Re \int_\GpmH T_M^\pm w \, \cw  \rd s \leq 0.
\end{align}
\end{lem}
\begin{proof}
We first consider the operator $T^+$. Fix some $w\in H^1_{\alpha_0} (\Omega)$.
We recall that on $\GH$ we can write
\begin{equation*}
	w(x_1,H) = \sum_{n \in \mathbb{Z}} w_n \re^{\ri \alpha_n x_1}, \hspace{0.7cm} 
	(T^+ w)(x_1,H) = \ri \sum_{n\in\mathbb{Z}} w_n \beta_n^+ \re^{\ri\alpha_n x_1}.
\end{equation*}
Recalling that $\alpha_n=\alpha_0+\frac{2\pi n}L$ from Proposition~\ref{fourier}, we have
\begin{align*}
	\int_\GH T^+ w \hspace{0.01cm} \cw \rd s & 
	= \int_{0}^{L} \ri \sum_{n\in\mathbb{Z}} w_n \beta_n^+ \re^{\ri\alpha_nx_1} \overline{\left( \sum_{m \in \mathbb{Z}} w_m \re^{\ri \alpha_m x_1} \right)} \rd x_1 \\
	& = \sum_{n\in\mathbb{Z}} \sum_{m \in \mathbb{Z}} \ri w_n \beta_n^+ \, \overline{w_m} \int_{0}^{L} \re^{\ri(\alpha_n-\alpha_m) x_1} \rd x_1 
	= \sum_{n\in\mathbb{Z}} \ri |w_n|^2 \beta_n^+ L.
\end{align*}
From the definition of $\beta_n^+$ in (\ref{beta_piu}), we have
\begin{align*}
	\Im \int_\GH T^+ w \, \cw\rd s 
	= L \sum_{n\in\mathbb{Z}}  |w_n|^2 \Re\beta_n^+
	= L\sum_{\alpha^2_n < k^2\varepsilon^+} |w_n|^2 \sqrt{k^2\varepsilon^+-\alpha_n^2} \geq 0, \\
	\Re \int_\GH T^+ w \, \cw\rd s 
	= -L \sum_{n\in\mathbb{Z}}  |w_n|^2 \Im\beta_n^+
	= -L\sum_{\alpha^2_n > k^2\varepsilon^+} |w_n|^2 \sqrt{\alpha_n^2-k^2\varepsilon^+} \leq 0,
\end{align*}

Similarly, with $w(x_1,-H) = \sum_{n \in \mathbb{Z}} w_n^- \re^{\ri \alpha_n x_1}$, recalling \eqref{beta_meno}--\eqref{eq:betaComplEps},
$$
\Im \int_\GmH T^- w \, \cw \rd s 
=L \sum_{n\in\mathbb{Z}}  |w_n^-|^2 \Re\beta_n^-
=\begin{cases}\displaystyle
	L\sum_{\alpha^2_n < k^2\varepsilon^-} |w_n^-|^2 \sqrt{k^2\varepsilon^- -\alpha_n^2} 
	& \varepsilon^-\in\IR,\\[2mm]
	\displaystyle
	L\sum_{n\in\IZ}|w_n^-|^2 \Re\sqrt{k^2\varepsilon^- -\alpha_n^2} 
	& \varepsilon^-\notin\IR.
\end{cases}
$$
In both cases, the quantity at the right-hand side is non-negative.
Note that for $\varepsilon^-\in\IR$ the sum is finite, while for $\varepsilon^-\notin\IR$ the series converges: $w\in H^1_{\alpha_0,0}(\Omega)$ implies that its trace on $\GmH$ belongs to $H^{1/2}_{\alpha_0}(\GmH)$ so $w_n^-=o(|n|^{-1})$, while $\lim_{|n|\to\infty} \Re\sqrt{k^2\varepsilon^- -\alpha_n^2} =0$.

The operators $T^\pm_M$ are defined in \eqref{DtN_trunc} as truncations of the Fourier diagonalization of $T^\pm$, so 
$\Im \int_\GpmH T^\pm_M w \,\cw\rd s$ and $-\Re\int_\GpmH T^\pm_M w \,\cw\rd s$ can be written as the sums (of non-negative terms) above, restricted to a subset of indices.
\end{proof}

\begin{rem} \label{imag_rem}
If $\varepsilon^-\in\IR$, the expansions of $\Im \int_\GpmH T^\pm w\, \cw\rd s$ in the proof of Lemma~\ref{lemma:l2prod} are finite, so they can be replicated by the truncated operators.
In particular, if 
\begin{align}\label{eq:MLarge}
	\varepsilon^-\in\IR, \qquad 
	M\ge{M_\star:=} \frac L{2\pi}\Big(\max\{\kappa^+,\kappa^-\} +|\alpha_0|\Big),
\end{align}
then 
$$
\Im \int_\GpmH T^\pm_M w\, \cw\rd s=\Im \int_\GpmH T^\pm w\, \cw\rd s
\qquad \forall w\in H^1_{\alpha_0,0}(\Omega).
$$
This identity is key to ensure the well-posedness of the TDG scheme in Proposition~\ref{prop:norma}.
\end{rem}

Since the Fourier expansion of $u^\inc$ has only one non-zero term, $\forall M\in\IN$ 
$$\big(T^+_M u^\inc(\cdot,H)\big)(x_1)
=\big(T^+ u^\inc(\cdot,H)\big)(x_1)
=\ri\beta_0^+u^\inc(x_1,H)
=-\ri\kappa^+\sin\theta\, \re^{\ri\kappa^+(x_1\cos\theta+H\sin\theta)}.$$

\begin{prop}\label{prop:TminusTM}
Let $\varepsilon^-\in\IR$, $M\ge M_\star$ (defined in \eqref{eq:MLarge}), $s\in\IR$, $t>0$, and
$\phi(x_1)=\sum_{n\in\IZ}\phi_n\re^{\ri\alpha_nx_1}\in H^{s+t}_{\alpha_0}(\GpmH)$ (recall~\eqref{eq:HsNorm}). 
Then the $H^{s-1}_{\alpha_0}(\GpmH)$ norm of the error committed approximating $T^\pm\phi$ by $T^\pm_M\phi$ decays algebraically in $M$:
$$
\|(T^\pm-T^\pm_M)\phi\|_{s-1,\alpha_0}
\le \Big(\frac{2\pi}LM-|\alpha_0|\Big)^{-t}\|\phi\|_{s+t,\alpha_0}.
$$
\end{prop}
\begin{proof}
We use the definition \eqref{eq:HsNorm} of the $H_{\alpha_0}^{s-1}(\GpmH)$ and the $H_{\alpha_0}^{s+t}(\GpmH)$ norms,
$|\beta_n^\pm|^2=\alpha_n^2-k^2\varepsilon^\pm$ for $|n|\ge M_\star$ by \eqref{beta_piu} and \eqref{beta_meno}, 
$\alpha_n=\alpha_0+\frac{2\pi}Ln$:
\begin{align*}
	\|(T^\pm-T^\pm_M)\phi\|_{s-1,\alpha_0}^2
	&=L\sum_{|n|>M}(1+\alpha_n^2)^{s-1} |\beta_n^\pm|^2 |\phi_n|^2
	\\
	&=L\sum_{|n|>M}\frac{\alpha_n^2-k^2\varepsilon^\pm}{(1+\alpha_n^2)^{1+t}} (1+\alpha_n^2)^{s+t}|\phi_n|^2
	\\
	&\le\max_{|n|>M}\frac1{(1+\alpha_n^2)^t} \|\phi\|_{s+t,\alpha_0}^2
	\le\frac1{\big(1+(\frac{2\pi}L M-|\alpha_0|)^2\big)^t} \|\phi\|_{s+t,\alpha_0}^2.
\end{align*}
\end{proof}

\subsection{Boundary value problem}

Using the DtN operators $T^\pm$, we impose the appropriate boundary conditions on $\GpmH$, obtaining the following boundary value problem: given the incident wave 
$u^\inc(\bx) = \re^{\ri \kappa^+(x_1\cos \theta + x_2\sin \theta)} 
= \re^{\ri\alpha_0 x_1 - \ri \beta_0^+ x_2}$,
with $\theta \in \left[-\pi, 0 \right]$, find $u \in H^1_{\alpha_0}(\Omega)$, with $\alpha_0 = \kappa^+ \cos \theta$, such that
\begin{equation} \label{nontrunc}
\begin{cases}
	\Delta u + k^2\varepsilon u = 0 & \hspace{0.5cm} \text{in} \; \Omega, \\
	u = 0 & \hspace{0.5cm} \text{on} \; \GD, \\
	\dn(u-u^\inc) - T^+(u-u^\inc) = 0 & \hspace{0.5cm} \text{on} \; \GH, \\
	\dn u - T^-u = 0 & \hspace{0.5cm} \text{on} \; \GmH.
\end{cases}
\end{equation}
By elliptic regularity, if the geometry singularities (i.e.\ $\partial D$ and the the discontinuities in $\varepsilon$) are away from $\GpmH$, then the solution $u$ to the Helmholtz problem \eqref{nontrunc} is smooth on $\GpmH$, and $T^\pm u-T^\pm_M u$ decays at least super-algebraically fast in $M$ by Proposition~\ref{prop:TminusTM}.

Numerically, we approximate a truncated boundary value problem, making use of the truncated DtN operators $T^\pm_M$ \eqref{DtN_trunc}: for $M\in\IN$, find $u^M \in H^1_{\alpha_0}(\Omega)$ such that
\begin{equation} \label{trunc}
\begin{cases}
	\Delta {u^M} + k^2\varepsilon {u^M} = 0 & \hspace{0.5cm} \text{in} \; \Omega, \\
	{u^M} = 0 & \hspace{0.5cm} \text{on} \; \GD, \\
	\dn ({u^M}-u^\inc) - T_M^+({u^M}-u^\inc) = 0 & \hspace{0.5cm} \text{on} \; \GH, \\
	\dn {u^M} - T_M^-{u^M} = 0 & \hspace{0.5cm} \text{on} \; \GmH.
\end{cases}
\end{equation}

\subsection{Variational formulation} 
Problem \eqref{nontrunc} can be written in variational form as in \cite[\S3.3]{BonnetBD1994}.
Let $u \in H^1_{\alpha_0, 0}(\Omega)$ be a distributional solution of the Helmholtz problem (\ref{nontrunc}). 
Multiplying both sides of the Helmholtz equation by a test function $v\in H^1_{\alpha_0, 0}(\Omega)$, and integrating by parts, the integrals over the left and right boundaries cancel by quasi-periodicity. Then, using the Dirichlet-to-Neumann operators to replace the normal derivative of the scattered field
$u^\scat=u-u^\inc$ on $\GpmH$, we obtain a weak formulation: find $u \in H^1_{\alpha_0, 0}(\Omega)$ such that
\begin{equation} \label{weak}
a_\varepsilon(u,v) = F(v) \hspace{0.5cm} \forall  v \in H^1_{\alpha_0, 0}(\Omega),
\end{equation}
where
\begin{align} \label{bilinear}
a_\varepsilon(u,v) & := 
\int_\Omega \left( \nabla u \cdot \nabla \cv - k^2\varepsilon u\cv \right) \rd\bx 
- \int_\GH  T^+u \; \cv \rd s 
- \int_\GmH T^-u \; \cv  \rd s, \\ \label{right-hand-side}
F(v) &  := -\int_\GH 2\ri\beta_0^+ u^\inc  \cv \rd s.
\end{align}
The value at the right-hand side \eqref{right-hand-side} comes from 
$-\dn u^\inc = T^+ u^\inc=\ri\beta_0^+ u^\inc$ on~$\GH$.

\section{Stability analysis} 
\label{s:Stability}
\subsection{Rellich identity}
\label{s:Uniqueness}

Following \cite{Civiletti2020, Lechleiter2010}, we prove a Rellich identity for the scattering problem (\ref{weak}) and 
we use it, together with the sign properties \eqref{eq:ImT}--\eqref{eq:ReT} of the DtN operators $T^\pm$,
to prove the well-posedness of the problem.

Suppose that $\Omega$ is divided into $P$ 
Lipschitz, connected subdomains $\Omega_j$, for $j=1,\ldots,P$, in which the relative permittivity $\varepsilon$ assumes a constant value.
Denote~by 
\begin{align*}
\Sigma:=\big\{&(j,j')\in \{1,\ldots,P\}^2: \; j<j',\\
&\Gamma_{j,j'}:=\partial\Omega_j\cap\partial\Omega_{j'} \text{ has positive 1-dimensional measure}\big\} 
\end{align*}
the index set of the interfaces between constant-$\varepsilon$ regions.
On $\Gamma_{j,j'}$ with $(j,j')\in\Sigma$, let $\bn=(n_1,n_2)$ denote the unit normal pointing from $\Omega_j$ into $\Omega_{j'}$.

\begin{lem}[Rellich identity] \label{RelLem}
Assume that $\varepsilon$ is real and positive. 
If $u\in H^1_{\alpha_0,0}(\Omega)$ is a solution to the variational problem (\ref{weak}) and the trace of $u$ on $\GpmH$ belongs to $H^1(\GpmH)$, then the following Rellich identity holds:
\begin{align} \nonumber
	& \int_\Omega 2 \left| \dxt u \right|^2 \rd\bx
	-  k^2\sum_{(j,j')\in\Sigma} ( \varepsilon_{j} -\varepsilon_{j'} ) \int_{\Gamma_{j,j'}} x_2 n_2 |u|^2\rd s
	- \int_{\GD} x_2n_2|\dn u|^2 \rd s\\
	&+ H \int_{\GH\cup\GmH} \biggl( \left|\dxo u \right|^2 - \left| \dn u \right|^2 -k^2\varepsilon |u|^2 \biggr)\rd s
	-\int_{\GH} T^+ u \cu\rd s -\int_\GmH T^- u \cu \rd s
	\nonumber \\
	& = -\int_{\GH} 2 \ri \beta_0^+ u^\inc \cu \rd s.
	\label{rellich}
\end{align}
\end{lem}
\begin{proof}
We first note that for quasi-periodic fields $v,w\in H^1_{\alpha_0}(\Omega)$, 
\begin{align*}
	v(L,x_2){\dn\cw}(L,x_2)
	& =v(L,x_2){\dxo\cw}(L,x_2)
	=\re^{\ri \alpha_0 L} v(0,x_2) \re^{-\ri \alpha_0 L} {\dxo\cw}(0,x_2)\\
	&= -v(0,x_2){\dn\cw}(0,x_2).
\end{align*}
Then all the integrals on $\Gleft$ and $\Gright$ that arise in the following integrations by parts cancel one another.

Testing the variational problem (\ref{weak}) with $v\in C^\infty_0(\Omega)$, we have $\Delta u\in L^2(\Omega)$.	By the Ne\v cas trace regularity result \cite[Theorem~4.24]{MCL00}, using that $u=0$ on $\Gamma_D$ and the assumption on the trace regularity of $u$ on $\GpmH$, we deduce $\dn u\in L^2(\Gamma_D)$ and $\dn u\in L^2(\GpmH)$.

Using the Rellich test function $x_2 \dxt u$, we have
\begin{align} \label{eqn:iniz}
	\int_\Omega x_2\dxt u \Delta \cu\rd\bx
	&=  - \int_{\Omega}\nabla \Big[x_2 \dxt u\Big] \cdot \nabla \cu\rd\bx
	+ \int_{\partial \Omega} x_2 \dxt u \dn \cu\rd s \\ \nonumber
	&=  - \int_\Omega \biggl[ \, \left| \dxt u \right|^2
	+ x_2 \nabla \left( \dxt u \right) \cdot \nabla \cu \, \biggl]\rd \bx \\
	&\quad\nonumber
	+ H \int_{{\GH\cup\GmH}} \left| \dxt u \right|^2 \rd s
	+ \int_{\GD} x_2 \dxt u \dn \cu\rd s.
\end{align}
Taking twice the real part of (\ref{eqn:iniz}) {and integrating by parts again}, we get
\begin{align} \label{eqn:realpart}
	2 \Re \int_\Omega x_2 \dxt u \Delta \cu \rd\bx
	= &  - \int_\Omega 2 \bigg(\left| \dxt u \right|^2 + x_2 \, 2 \Re \left[ \nabla \left( \dxt u \right) \cdot \nabla \cu \right]\bigg) \rd\bx\\ \nonumber
	& +  H \int_{{\GH\cup\GmH}} 2 \left| \dxt u \right|^2 \rd s
	+2\Re\int_{\GD} x_2 \dxt u \dn \cu\rd s\\ \nonumber
	= & - \int_\Omega \bigg(2\left| \dxt u \right|^2 +x_2 \dxt |\nabla u|^2\bigg)\rd\bx\\ \nonumber
	& +  H \int_{{\GH\cup\GmH}} 2 \left| \dxt u \right|^2\rd s
	+ 2\Re\int_{\GD} x_2 \dxt u \dn \cu\rd s \\ \nonumber
	= & \int_\Omega \bigg(|\nabla u|^2 - 2 \left| \dxt u \right|^2\bigg){\rd\bx}
	+ H \int_{{\GH\cup\GmH}} \biggl( - |\nabla u|^2 + 2 \left| \dxt u \right|^2 \biggl){\rd s}\\ \nonumber
	& 
	+ \int_{\GD} \biggl(2 x_2 \Re\bigg[\dxt u \dn \cu\bigg]
	{-x_2n_2|\nabla u|^2}\biggl){\rd s}.
\end{align}
Using that $\Delta u + k^2 \varepsilon u = 0$ in $\Omega$ {and integrating by parts on each region $\Omega_j$}, we get
\begin{align} \label{eqn:subs}
	2 \Re \int_\Omega x_2\dxt u \Delta \cu\rd\bx
	= & -2k^2 \int_\Omega x_2 \varepsilon  \Re \left( \dxt u {\cu} \right)\rd\bx
	= -k^2 \int_\Omega x_2 \varepsilon \dxt|u|^2\rd\bx \\ \nonumber
	= &\ k^2{\sum_{j=1}^P \int_{\Omega_j}} \dxt[x_2 \varepsilon]  |u|^2\rd\bx
	- Hk^2 \int_{{\GH\cup\GmH}} \varepsilon |u|^2\rd s
	\\ \nonumber
	& -  k^2{\sum_{(j,j')\in\Sigma} ( \varepsilon_{j} -\varepsilon_{j'} )
		\int_{\Gamma_{j,j'}}} x_2 n_2 |u|^2\rd s
	-k^2\int_\GD x_2n_2\varepsilon|u|^2\rd s.
\end{align}
The term on $\GD$ vanishes since $u\in H^1_{\alpha_0,0}(\Omega)$.
From (\ref{eqn:realpart}) and (\ref{eqn:subs}), {using that $\dxt [x_2 \varepsilon]=\varepsilon$ in each $\Omega_j$,} we get:
\begin{align}\label{A} 
	\int_\Omega \big( |\nabla u|^2 - &k^2\varepsilon{|u|^2}\big)\rd\bx
	= \int_\Omega 2 \left| \dxt u \right|^2\rd\bx
	+ H \int_{{\GH\cup\GmH}} \biggl( |\nabla u|^2 - 2 \left| \dxt u \right|^2 -k^2\varepsilon |u|^2 \biggr){\rd s} \\ \nonumber
	&+\int_{\GD}\biggl(-2x_2\Re\dxt u\dn \cu
	+x_2n_2|\nabla u|^2\biggr) \rd s
	-  k^2\sum_{(j,j')\in\Sigma} ( \varepsilon_{j} -\varepsilon_{j'} ) \int_{\Gamma_{j,j'}} x_2 n_2 |u|^2 \rd s.
\end{align}
Since $u\in H^1_{\alpha_0,0}(\Omega)$, on $\GD$
$$
\nabla u = \bn \dn u  \quad\implies\quad
\dxt u =n_2 \dn u \quad\implies\quad
\dxt u\dn \cu =n_2 |\dn u |^2,
$$
so the integral on $\GD$ is equal to $-\int_\GD x_2n_2|\dn u|^2\rd s$ (this is \cite[eq.~(3.38)]{BonnetBD1994}).
From the variational problem (\ref{weak}),
\begin{equation} \label{B}
	\int_\Omega \big(|\nabla u|^2 - k^2\varepsilon {|u|^2}\big) \rd\bx
	= \int_{\GH} T^+ u \cu\rd s +\int_\GmH T^- u \cu \rd s
	- \int_{\GH} 2 \ri \beta_0^+ u^\inc \cu\rd s.
\end{equation}
Combining (\ref{A}) and (\ref{B}) yields the Rellich identity \eqref{rellich}.
\end{proof}

The solution $u^M$ of the truncated BVP \eqref{trunc} satisfies the Rellich identity \eqref{rellich} with $T^\pm$ replaced by $T^\pm_M$.

The proof of the Rellich identity \eqref{rellich} relies on the integration by parts of the Helmholtz equation tested against the ``Rellich multiplier'' $x_2\partial_{x_2}u$, as in, e.g.\ \cite[Theorem~3.5]{BonnetBD1994}, \cite[Lemma~3.2]{Lechleiter2010} and \cite[Lemma~1]{Civiletti2020}.
This takes the name from the analogous multiplier $\bx\cdot\nabla u$ used in the context of scattering by a bounded obstacle.
In this setting, a ``Morawetz multiplier'' in the form $\bx\cdot\nabla u+\alpha u$ for a suitable scalar field $\alpha$ is typically used; see e.g.\ \cite[Lemma~3.5]{CWM08} for an example, and \cite[Remark~4.2]{MS17} for a brief history of these multipliers.

\subsection{Solution existence and uniqueness}

The following existence results is obtained from a decomposition of the bilinear form $a_\varepsilon(\cdot, \cdot)$ \cite[Lemma 3.1]{BonnetBD1994} and Fredholm analysis.
\begin{thm}[{\cite[Theorem 3.2--3.4]{BonnetBD1994}}]
Problem (\ref{weak}) has at least one solution, and the set of solutions is at most a finite-dimensional affine space. Problem (\ref{weak}) is well-posed for every value of $k$ except possibly for an increasing sequence $(k_m)_{m \geq 1}$	that tends to infinity with~$m$. 
\end{thm}

A ``singular frequency'' is a value of $k$ such that the homogeneous problem $a_\varepsilon(u,v)=0$ $\forall v \in H^1_{\alpha_0, 0}(\Omega)$ has a non-trivial solution {\cite[\S3.4]{BonnetBD1994}}. 
For a given $\alpha_0$, the singular frequencies form at most a countable sequence without accumulation points.
We report a condition on the relative permittivity $\varepsilon$ and on the Dirichlet boundary $\GD$ that guarantees the uniqueness of the solution of (\ref{nontrunc}) for all values of $k${, ruling out the presence of singular frequencies}.
\begin{thm}[{\cite[Theorem~3.5]{BonnetBD1994}}] \label{wellposed} 
Assume that $\varepsilon$ is positive real and
\begin{equation}\label{eq:NonTrap}
	\begin{cases}
		n_2 \, x_2 \leq 0 \quad\text{on} \hspace{0.2cm} \GD, \\
		\forall x_1\in[0,L], \; \tau\mapsto\varepsilon(x_1,\tau)\text{ and }
		\tau\mapsto\varepsilon(x_1,-\tau)\\
		\hspace{20mm}\text{are monotonically non-decreasing for }\tau\in[0,\infty),
	\end{cases}
\end{equation}
where $\bn = (n_1, n_2)$ is the unit normal on $\GD$ pointing from $\Omega$ into $D$.
Assume that the scattering problem is non-trivial, i.e.\ $D\ne\emptyset$ or $\varepsilon$ is not constant (or both).
Then (\ref{nontrunc}) is well-posed for every value of $k$.
\end{thm}
\begin{proof}
Under assumption \eqref{eq:NonTrap}, Problem~\eqref{nontrunc} is a particular case of that described in \cite[Theorem~3.5]{BonnetBD1994}, so the proof therein applies.
However, in the proof \cite[Theorem~3.5]{BonnetBD1994}, 
the treatment of the terms on $\GpmH$ when $\beta_n^\pm=0$, i.e.\ in the case of Rayleigh--Wood anomalies, is not specified, so here we give the details following the ideas in \cite[Theorem~4.5]{Strycharz1998}.

Thanks to Fredholm theory \cite[Theorem~3.2]{BonnetBD1994}, we only have to show that the homogeneous problem is well-posed: 	if $u$ is solution of Problem~\eqref{weak} with $u^\inc=0$, then $u=0$.
We denote by $u_n^\pm$ the coefficients of the quasi-periodic Fourier expansion of $u$ on $\GpmH$:
$u(x_1,\pm H)=\sum_{n\in\IZ}u_n^\pm \re^{\ri\alpha_nx_1}$.
The imaginary part of the Rellich identity \eqref{rellich} with $u^\inc=0$, together with \eqref{eq:ImT}, gives that $u_n^\pm=0$ for all $n$ with $\alpha_n^2<k^2\varepsilon^\pm$.
The definitions \eqref{beta_piu} and \eqref{beta_meno} of $\beta_n^\pm$ give
\begin{equation}\label{Beta_nExpansion}
	\alpha_n^2-|\beta_n^\pm|^2-k^2\varepsilon^\pm =-(\beta_n^\pm)^2-|\beta_n^\pm|^2 =
	\begin{cases}
		-2|\beta_n^\pm|^2 &\alpha_n^2\le k^2\varepsilon^\pm,\\
		0&\alpha_n^2> k^2\varepsilon^\pm.
	\end{cases}
\end{equation}
This identity allows to expand the $\GpmH$ term in the Rellich identity \eqref{rellich} 
in terms of the Fourier coefficients and show that it vanishes:
\begin{align*}
	\int_\GpmH\bigg(
	\left| \dxo u \right|^2 - \left| \dn u \right|^2 -k^2\varepsilon^\pm |u|^2 
	\bigg)\rd s
	&=L\sum_{n\in\IZ}(\alpha_n^2-|\beta_n^\pm|^2-k^2\varepsilon^\pm)|u_n^\pm|^2\\
	&= - 2 L \sum_{\alpha^2_n < k^2\varepsilon^\pm} |\beta_n^\pm|^2 |u_n^\pm|^2 =0.
\end{align*}
Assumption \eqref{eq:NonTrap}, together with the convention on $\bn$ on $\Gamma_{j,j'}$ stipulated at the beginning of \S\ref{s:Uniqueness} and inequality \eqref{eq:ReT}, imply that the real part of all the remaining terms in the Rellich identity are non-negative.
Since $u^\inc=0$, they sum to zero, thus each term vanishes.

In particular, $\partial_{x_2}u=0$ 
in $\Omega$.
Since $u=0$ on $\GD$ and on the interfaces $\Gamma_{j,j'}$ with $x_2n_2\ne0$, then $u=0$ on the vertical strips that intersect the obstacles $D\cup\bigcup_{(j,j')\in \Sigma}\Gamma_{j,j'}$.
(The case with $D=\emptyset$ and $\varepsilon$ discontinuous only on $\{x_2=0\}$ is treated by translating vertically the Cartesian axes.)
Since $u$ is a distributional solution of the piecewise-constant coefficient Helmholtz equation in $\Omega$, as can be seen by taking any $v\in C^\infty_0(\Omega)$ in the variational formulation \eqref{weak}, the unique continuation principle \cite[Theorem~8.6]{ColtonKress2019} ensures that $u=0$ in $\Omega$.
\end{proof}

The trivial scattering problem with $D=\emptyset$, $\Omega_0=\IR^2$, and constant $\varepsilon>0$ on $\Omega$ is not well-posed in general: if $k,\theta, L$ are such that $\beta_n=0$ for some $n\in\IZ$, then any multiple of the horizontal plane wave $\re^{\pm\ri\kappa x_1}$ solves the homogeneous quasi-periodic boundary value problem.

The requirements \eqref{eq:NonTrap} on $\varepsilon$ and $\GD$, which specialize \cite[eq.~(3.34)]{BonnetBD1994} to piecewise-constant $\varepsilon$, constitute a non-trapping condition. 
Geometrically, they mean that a point moving along a vertical half line from any $(x_1,0)$, either upwards or downwards, does not enter any impenetrable obstacle, and that at every material interface the value of $\varepsilon$ increases.
In particular, a bounded Lipschitz impenetrable obstacle in $\Omega$ is allowed if it 
can be written as $D\cap\{0<x_1<L\}=\{(x_1,x_2):\;a<x_1<b,\, g_-(x_1)<x_2<g_+(x_1)\}$ for some $0<a<b<L$ and $g_\pm\in C^{0,1}(a,b)$ with $g_-(x_1)<0<g_+(x_1)$ (see e.g.\ the squares in Figures~\ref{fig:region} and \ref{fig:dir}).
Similar conditions are present in \cite[Theorem~1]{Civiletti2020}. If conditions \eqref{eq:NonTrap} are not satisfied, one can build examples of non-unique solutions, as it is done in \cite[\S5]{BonnetBD1994}.

{\begin{corol}\label{cor:TruncWP}
	Under the assumptions of Theorem~\ref{wellposed}, the truncated problem \eqref{trunc} with $M\ge M_\star$ as in \eqref{eq:MLarge} is wellposed.
\end{corol}
\begin{proof}
	Under condition \eqref{eq:MLarge}, the imaginary part of $T^\pm_M$ coincides with the imaginary part of $T^\pm$.
	The only requirement on the real part of $T^\pm_M$ in the proof of Theorem~\ref{wellposed} is its non-negativity, which is preserved by the truncation.
\end{proof}
}

\subsection{Explicit stability estimates away from Rayleigh--Wood anomalies} 

Given $\varepsilon^\pm$, $\theta$, and $L$, the Rayleigh--Wood anomalies are the values of $k$ 
such that there is a $n \in \IZ$ with $\alpha^2_n = k^2\varepsilon^\pm$, \cite{Pinto2020,Bruno2014}.
Equivalently, $(\kappa^{\pm})^2 = (\kappa^+\cos\theta+\frac{2\pi}{L}n)^2$, or $\beta_n^\pm=0$, i.e.\ some Fourier modes are in the kernel of one of the DtN operators $T^\pm$.
This means that one of the two plane waves $\re^{\ri \kappa^\pm x_1}$, propagating in $\Omega^\pm$ in direction $x_1$, are quasi-periodic with parameter $\alpha_0$ over the interval $[0,L]$. 
We denote by $\delta$ a measure of the distance from the closest Rayleigh--Wood anomaly:
\begin{equation} \label{distance}
\delta_\pm := \min_{n\in\IZ} |\beta_n^\pm|, \qquad
\delta := \min\{ \delta_+, \delta_- \},
\end{equation}
and in the following analysis we assume that its value is nonzero.
Analogous assumptions are made also in the analysis of DtN operators for acoustic waveguides, see \cite[eq.~(3.9)--(3.11a)]{Melenk2023} and \cite[\S2, $\beta_j\ne0$ assumption]{Monk2024}.

The next theorem gives an explicit quantitative bound on the (wavenumber-weighted) $H^1(\Omega)$ norm of the solution under this assumption.
The key tool in its proof is the Rellich identity \eqref{rellich}.
The non-trapping condition, together with Lemma~\ref{lemma:l2prod}, determines the signs of the integrals in \eqref{rellich}.
The non-anomaly assumption $\delta>0$ and the integrals over $\GpmH$ in \eqref{rellich} allow to control all the terms in the Fourier expansion of $u$ on $\GpmH$ (see \eqref{unmeno}--\eqref{unpiu}).
From this, we bound $\|u\|_{L^2(\GH\cup\GmH)}$.
The norm in $\Omega$ is controlled using a Poincaré-type inequality in the $x_2$ direction, taking advantage of the $\|\partial_{x_2}u\|_{L^2(\Omega)}$ term in the Rellich identity.

\begin{thm} \label{thm-stima}
Under the non-trapping assumptions \eqref{eq:NonTrap}, and assuming that $\delta$ in \eqref{distance} is nonzero,  the unique solution $u$ of \eqref{nontrunc} satisfies the stability bound
\begin{align}\label{stimaH1} 
	\|\kappa u\|_{L^2(\Omega)}^2 + \|\nabla u\|_{L^2(\Omega)}^2
	\le 10L\kappa_+|\sin\theta|\bigg(\frac H\delta\|\kappa\|_{L^\infty(\Omega)}^2
	\big(2+7H\max\{\kappa^+,\kappa^-\}\big)^3+1\bigg).
\end{align}
\end{thm}
\begin{proof}
Let $u$ be the solution of \eqref{nontrunc}.
Under the non-trapping assumption \eqref{eq:NonTrap}, we know by Theorem~\ref{wellposed} that $u$ exists and it is uniquely defined.
We expand the outward-propagating components, 
i.e.\ $u^\scat$ in $\Omega_H^+$ and $u$ in $\Omega_H^-$, in Fourier series:
\begin{align*}
	&u(\bx)=
	\begin{cases}\displaystyle
		u^\inc(\bx)+u^\scat(\bx)
		=\re^{\ri(\alpha_0x_1-\beta_0^+x_2)}+
		\sum_{n \in \IZ}u_n^\scat\re^{\ri(\alpha_nx_1+\beta_n^+ (x_2-H))} & \bx\in\GH\cup \Omega_H^+,
		\\\displaystyle
		\sum_{n \in \IZ}u_n^-\re^{\ri(\alpha_nx_1-\beta_n^- (x_2+H))} & \bx\in\GmH\cup \Omega_H^-.
	\end{cases}
\end{align*}
In particular,
$u(\bx)=(u_0^\scat+\re^{-\ri\beta_0^+H})\re^{\ri\alpha_0x_1}
+\sum_{0\ne n \in \IZ}u_n^\scat\re^{\ri\alpha_nx_1}$ on $\GH$.
Their partial derivatives are:
\begin{align*}
	&\dxo u(\bx)=
	\begin{cases}\displaystyle
		\ri\alpha_0\re^{\ri(\alpha_0x_1-\beta_0^+x_2)}+
		\ri\sum_{n \in \IZ}\alpha_n u_n^\scat\re^{\ri(\alpha_nx_1+\beta_n^+ (x_2-H))} & \bx\in\GH\cup \Omega_H^+,
		\\\displaystyle
		\ri\sum_{n \in \IZ}\alpha_n u_n^-\re^{\ri(\alpha_nx_1-\beta_n^- (x_2+H))} & \bx\in\GmH\cup \Omega_H^-,
	\end{cases}
	\\
	&\dxt u(\bx)=
	\begin{cases}\displaystyle
		-\ri\beta_0^+\re^{\ri(\alpha_0x_1-\beta_0^+x_2)}+
		\ri\sum_{n \in \IZ}\beta_n^+ u_n^\scat\re^{\ri(\alpha_nx_1+\beta_n^+ (x_2-H))} & \bx\in\GH\cup \Omega_H^+,
		\\\displaystyle
		-\ri\sum_{n \in \IZ}\beta_n^- u_n^-\re^{\ri(\alpha_nx_1-\beta_n^- (x_2+H))} & \bx\in\GmH\cup \Omega_H^-.
	\end{cases}
\end{align*}
Identity \eqref{Beta_nExpansion} allows to expand the $\GpmH$ term in the Rellich identity \eqref{rellich} in terms of the Fourier coefficients $u_n^\scat$ and $u_n^-$:
\begin{align*}
	&\int_{\GH\cup \GmH} \biggl(\left| \dxo u \right|^2 - \left| \dn u \right|^2 -k^2\varepsilon |u|^2 \biggr) \rd s \\
	&= L(\alpha_0^2-k^2\varepsilon^+)|\re^{-\ri \beta_0^+H}+u_0^\scat|^2
	-L|\beta_0^+|^2|\re^{-\ri \beta_0^+H}-u_0^\scat|^2\\
	&\qquad-2L \sum_{0\neq n \in \IZ, \alpha_n^2 \leq k^2\varepsilon^+} |\beta_n^+|^2 {|u_n^\scat|^2}
	- 2L \sum_{n \in \IZ, \alpha_n^2 \leq k^2\varepsilon^-} |\beta_n^-|^2|u^-_n|^2.
\end{align*}
Using the definitions \eqref{beta_piu}, \eqref{DtN}, \eqref{beta_meno}, \eqref{DtN_meno} of $T^\pm$ and $\beta_n^\pm$, we expand the DtN operators:
\begin{align*}
	&\int_{\GH} T^+ u\cu\rd s
	=-L\sum_{\alpha^2_n > k^2\varepsilon^+} |\beta_n^+||u_n^\scat|^2
	+\ri L |\beta_0^+||u_0^\scat+\re^{-\ri \beta_0^+H}|^2
	+\ri L\!\!\sum_{0\ne n\in\IZ,\;\alpha^2_n < k^2\varepsilon^-}\!\!\! |\beta_n^+||u_n^\scat|^2 ,
	\\
	&\int_\GmH T^- u\cu\rd s
	=-L\sum_{\alpha^2_n > k^2\varepsilon^-} |\beta_n^-||u_n^-|^2
	+\ri L\sum_{\alpha^2_n \leq k^2\varepsilon^-} |\beta_n^-||u_n^-|^2 .
\end{align*}

We separate the terms appearing in the Rellich identity \eqref{rellich} associated to $\GpmH$, to the other parts of the domain, and to the right-hand side:
\begin{align*} \nonumber
	I_\GpmH := & H\int_{\GH\cup\GmH} \biggl( \left|\dxo u \right|^2 - \left| \dn u \right|^2 -k^2\varepsilon |u|^2 \biggr)\rd s -\int_{\GH} T^+ u\cu\rd s-\int_\GmH T^- u\cu\rd s
	\\ \nonumber
	=& 
	-L\Big(H|\beta_0^+|^2+\ri|\beta_0^+|\Big) \big|\re^{-\ri\beta_0^+ H}+u^\scat_0\big|^2
	- L H|\beta_0^+|^2\big|\re^{-\ri\beta_0^+ H}-u^\scat_0\big|^2\\ \nonumber
	&-L\sum_{0\ne n\in\IZ,\; \alpha_n^2\le k^2\varepsilon^+}\Big(2H|\beta_n^+|^2+\ri|\beta_n^+|\Big) |u^{\scat}_n|^2
	-L\sum_{ \alpha_n^2\le k^2\varepsilon^-}\Big(2H|\beta_n^-|^2+\ri|\beta_n^-|\Big)|u^-_n|^2\\ 
	&+L\sum_{\alpha^2_n > k^2\varepsilon^+} |\beta_n^+||u_n^{\scat}|^2
	+L\sum_{\alpha^2_n > k^2\varepsilon^-}|\beta_n^-||u_n^-|^2,
	\\
	I_{\Omega, \Sigma, \GD} :=& \int_\Omega 2|\dxt u|^2\rd\bx
	-k^2\sum_{(j,j')\in\Sigma} ( \varepsilon_{j} -\varepsilon_{j'} ) \int_{\Gamma_{j,j'}} x_2 n_2 |u|^2\rd s
	- \int_{\GD} x_2n_2|\dn u|^2 \rd s,
	\\
	I_{\mathrm{RHS}} :=& -\int_{\GH} 2 \ri \beta_0^+ {u^\inc} \cu \rd s
	=-2\ri L\beta_0^+ {(1+\overline{u_0^\scat} \re^{-\ri\beta_0^+H})}.
\end{align*}
The Rellich identity writes as 
\begin{equation}\label{eq:RellichIII}
	I_\GpmH+I_{\Omega, \Sigma, \GD}=I_{\mathrm{RHS}}, \qquad
	\text{with} \quad I_{\Omega, \Sigma, \GD}\in \IR, \quad
	I_{\Omega, \Sigma, \GD}\ge 0, \quad\Im I_\GpmH\le0, 
\end{equation}
thanks to the non-trapping assumption \eqref{eq:NonTrap}, which implies that $(\varepsilon_j-\varepsilon_{j'})x_2n_2\le0$ on $\Gamma_{j,j'}$, $(j,j')\in\Sigma$.
For simplicity we use the coefficients of the total field $u$ on $\GH$:
$$
u_n^+:=\begin{cases} u_0^\scat+\re^{-\ri\beta_0^+H}  & n=0,\\ u_n^\scat & n\ne0.\end{cases}
$$
It follows that $I_{\mathrm{RHS}}=-2\ri L\beta_0^+ \overline{u_0^+}\re^{-\ri\beta_0^+H}$.
Thanks to \eqref{eq:RellichIII} and the weighted Young inequality,
\begin{align*}
	L\sum_{\alpha_n^2\le k^2\varepsilon^+}|\beta_n^+||u^+_n|^2
	+L\sum_{\alpha_n^2\le k^2\varepsilon^-}|\beta_n^-||u^-_n|^2
	& =|\Im I_\GpmH|
	=|\Im I_{\mathrm{RHS}}|
	\le 2L |\beta_0^+||u_0^+|\\
	&\le 2L|\beta_0^+|+\frac L2|\beta_0^+||u_0^+|^2,
\end{align*}
so, bringing the last term of the inequality to the left and dividing by $\frac12$, we can control the propagative-mode coefficients:
\begin{align} \label{unmeno}
	\sum_{\alpha_n^2\le k^2\varepsilon^+}|\beta_n^+||u^+_n|^2
	+\sum_{\alpha_n^2\le k^2\varepsilon^-}|\beta_n^-||u^-_n|^2
	&\le 4|\beta_0^+|,
	\\ \label{unpiu}
	\sum_{\alpha_n^2\le k^2\varepsilon^+}|u^+_n|^2
	+\sum_{\alpha_n^2\le k^2\varepsilon^-}|u^-_n|^2
	&\le \frac{4}{\delta}|\beta_0^+|.
\end{align}
In particular, \eqref{unmeno} gives $|u_0^+|\le2$ and $|I_{\mathrm{RHS}}|\le 4L|\beta_0^+|$.
Moreover,
\begin{align}\nonumber
	\Re I_\GpmH
	=&-LH|\beta_0^+|^2 |u^+_0|^2
	-2LH\sum_{0\ne n\in\IZ,\; \alpha_n^2\le k^2\varepsilon^+}|\beta_n^+|^2|u^+_n|^2 
	-2LH\sum_{n\in\IZ,\; \alpha_n^2\le k^2\varepsilon^-}|\beta_n^-|^2|u^-_n|^2  \\
	& - L H|\beta_0^+|^2\big|2\re^{-\ri\beta_0^+ H}-u^{+}_0\big|^2
	+L\sum_{\alpha^2_n > k^2\varepsilon^+} |\beta_n^+||u_n^+|^2
	+L\sum_{\alpha^2_n > k^2\varepsilon^-} |\beta_n^-||u_n^-|^2,
	\label{eq:ReIGamma}
\end{align}
so\, from $\Re I_\GpmH\le |I_{\mathrm{RHS}}| {\le  4L|\beta_0^+|}$,
\begin{align} \nonumber
	&\sum_{\alpha^2_n > k^2\varepsilon^+} |\beta_n^+||u_n^+|^2
	+\sum_{\alpha^2_n > k^2\varepsilon^-} |\beta_n^-||u_n^-|^2
	\\
	\nonumber
	\le	&  \, {4|\beta_0^+|} 
	+20H|\beta_0^+|^2 
	+2H\sum_{0\ne n\in\IZ,\; \alpha_n^2\le k^2\varepsilon^+}|\beta_n^+|^2|u^+_n|^2 
	+2H\sum_{ \alpha_n^2\le k^2\varepsilon^-}|\beta_n^-|^2|u^-_n|^2 
	\\ \label{unresto}
	\le&
	\big(4+28H\max\{\kappa^+,\kappa^-\}\big)|\beta_0^+|,
\end{align}
where we used 
inequality~\eqref{unmeno} and
$|\beta_n^\pm|\le\kappa^\pm$ 
for $n$ with $\alpha_n^2\le k^2\varepsilon^\pm$.
From \eqref{unpiu} and \eqref{unresto}, and recalling the definition~\eqref{distance} of $\delta$, we get an estimate for the norm of the solution on $\GH \cup \GmH$
\begin{align}
	\|u\|_{L^2(\GH\cup\GmH)}^2= L
	\sum_{n\in\IZ}(|u^+_n|^2+ |u^-_n|^2)
	&\le
	{\frac{4L}\delta\big(2+7 H\max\{\kappa^+,\kappa^-\}\big)|\beta_0^+|. }
	\label{stimaorizzontale}
\end{align}
We define 
\begin{equation*}
	\Omega_+:=\Omega\cap\{x_2>0\}, \quad
	\Omega_-:=\Omega\cap\{x_2<0\}, \quad
	m_+(x_1):=\inf\{x_2\ge 0: (x_1,x_2)\in\Omega\}, \; 0<x_1<L,
\end{equation*}
and use 
a Poincaré-type inequality 
in the $x_2$ variable 
to bound the $L^2$ norm on $\Omega_\pm$
\begin{align*}
	\|u\|_{L^2(\Omega_+)}^2 & =
	\int_0^L\int_{m_+(x_1)}^H \bigg|u(x_1,H)-\int_{x_2}^H \dxt u(x_1,x')\rd x'\bigg|^2\rd x_2\rd x_1\\&
	\le 2H\|u\|_{L^2(\GH)}^2 + H^2\|\dxt u\|_{L^2(\Omega_+)}^2.
\end{align*}
The same holds for $\Omega_-$ and $\GmH$, so 
\begin{equation} \label{stimal2}
	\|u\|_{L^2(\Omega)}^2\le 2H\|u\|_{L^2(\GH\cup\GmH)}^2 + H^2\|\dxt u\|_{L^2(\Omega)}^2.
\end{equation}
Recalling \eqref{eq:RellichIII}, $|I_{\mathrm{RHS}}|\le 4L|\beta_0^+|$, and collecting the negative terms in \eqref{eq:ReIGamma},
\begin{align} \nonumber
	2\|\partial_{x_2}u\|^2_{L^2(\Omega)}
	\le&I_{\Omega, \Sigma, \GD}
	=\Re I_{\Omega, \Sigma, \GD}
	=\Re I_{\mathrm{RHS}}-\Re I_\GpmH
	\\ \nonumber
	\le &  
	4L|\beta_0^+| + 20 LH|\beta_0^+|^2 
	+2LH\!\!\sum_{0\ne n\in\IZ,\; \alpha_n^2\le k^2\varepsilon^+}|\beta_n^+|^2|u^+_n|^2 
	+2LH\sum_{\alpha_n^2\le k^2\varepsilon^-}|\beta_n^-|^2|u^-_n|^2 
	\\ \label{stimaverticale}
	\le& 4L|\beta_0^+| + 20 LH|\beta_0^+|^2 +
	2 k^2 H \Big(\varepsilon^+\|u\|_{L^2(\GH)}^2+\varepsilon^-\|u\|_{L^2(\GmH)}^2\Big)
	,\end{align}
where we used that $|\beta_n^\pm|^2\le k^2\varepsilon^\pm$ for these $n$.
From \eqref{stimal2}, \eqref{stimaorizzontale}, \eqref{stimaverticale}, and again $\beta_0^+\le\kappa^+$, we control $u$ over the whole domain:
\begin{align*} \nonumber
	&\|u\|_{L^2(\Omega)}^2\\ \nonumber
	&\le   2H\|u\|_{L^2(\GH\cup\GmH)}^2 
	+2LH^2|\beta_0^+| + 10 LH^3|\beta_0^+|^2 +
	k^2 H^3 \Big(\varepsilon^+\|u\|_{L^2(\GH)}^2+\varepsilon^-\|u\|_{L^2(\GmH)}^2\Big)
	\\ \nonumber
	&\le \frac{4LH}\delta\big(2+H^2\max\{\kappa^+,\kappa^-\}^2\big)\big(2+7H\max\{\kappa^+,\kappa^-\}\big)|\beta_0^+|
	+2LH^2(1+5H\kappa^+)|\beta_0^+|
	\\
	&\le \frac{5LH}\delta\big(2+7H\max\{\kappa^+,\kappa^-\}\big)^3|\beta_0^+|.
\end{align*}
In the last inequality we have have bounded the last term with a quarter of the first one, using $2H\le\frac1\delta+H^2\delta\le\frac1\delta(1+H^2(\kappa^+)^2)$, which follows from $\delta\le\kappa^+$.
To conclude, we bound the gradient of $u$ using the weak formulation \eqref{weak}, the sign \eqref{eq:ReT} of $\Re T^\pm$, $|I_{\mathrm{RHS}}|\le 4L\beta_0^+$, and $\beta_0^+=-\kappa^+\sin\theta>0$:
\begin{align*} 
	\|u\|_{\kappa, H^1(\Omega)}^2 = 
	&\|\kappa u\|_{L^2(\Omega)}^2 + \|\nabla u\|_{L^2(\Omega)}^2
	\\ 
	=&\|\kappa u\|_{L^2(\Omega)}^2 +
	\int_\Omega \kappa^2 |u|^2\rd\bx+
	\Re\Big\{\int_\GH T^+u\cu\rd s+\int_\GmH T^-u\cu\rd s\Big\}+\Re I_{\mathrm{RHS}}
	\\ 
	\le& 2 \|\kappa u\|_{L^2(\Omega)}^2 +4L|\beta_0^+|
	\\ 
	\le& \frac{10LH}\delta\|\kappa\|_{L^\infty(\Omega)}^2 \big(2+7H\max\{\kappa^+,\kappa^-\}\big)^3|\beta_0^+|
	+4L|\beta_0^+|
	\\ 
	\le& 10L\kappa_+|\sin\theta|\bigg(\frac H\delta\|\kappa\|_{L^\infty(\Omega)}^2
	\big(2+7H\max\{\kappa^+,\kappa^-\}\big)^3+1\bigg).
\end{align*}
\end{proof}

\begin{rem}[Stability estimates related to \eqref{stimaH1} available in the literature]
\label{rem:StabilityCompare}
Estimate \eqref{stimaH1} resembles the one in \cite[Theorem.~3.1]{Zhu2024}, since they {are both proportional to the sine of the incidence} 
angle $\theta$. 
In \cite{Zhu2024}, the diffraction grating profile is described by a periodic Lipschitz function and a Dirichlet condition is imposed, while the wavenumber is constant, so they can use a Poincaré-type inequality from the grating profile to the artificial DtN boundary without having to deal with Rayleigh--Wood anomalies. 

A different way to derive explicit stability estimates from the Rellich identity \eqref{rellich} is to use a Poincaré inequality in the vertical direction, similarly to \eqref{stimal2}, starting from a discontinuity in the wavenumber $\kappa$ instead of starting from $\GpmH$.
This has been done both in the case of two regions, as in \cite[\S4, in particular pp.~380--381]{Lechleiter2010} and \cite[\S5]{Zhu2024}, and in the case of domains composed of multiple materials, as in \cite[Lemma~2]{Civiletti2020}, under non-trapping assumptions similar to \eqref{eq:NonTrap}. 
The advantage of such estimates is that they are insensitive to Rayleigh--Wood anomalies.
On the other hand, \emph{(i)} they require a material interface crossing the whole domain $\Omega$ horizontally (thus they exclude configurations such as that in Figure~\ref{fig:dir} 
below), and \emph{(ii)} the bounds obtained blow up for vanishingly small material jumps.
\end{rem}

\section{The DtN-TDG method}
\label{s:TDG}
The finite element method for the Helmholtz equation and the diffraction grating problem with DtN boundary conditions has been studied in various works, such as \cite{Bao2000,BaoLi2022,Bao1995}{, while finite volume methods have been used in} \cite{Wang2022}.
Here we present the DtN-TDG method, adapting the TDG of \cite{Hiptmair2011,Hiptmair2016,Kapita2018,Howarth2014} to the model problem introduced in \S\ref{secmod}.

From now on, we write $T_M^\pm$ for $M\in \IN\cup\{\infty\}$, with $T_\infty^\pm=T^\pm$, and denote by $u^\infty=u$ the solution of \eqref{nontrunc}.

\subsection{Formulation of the numerical method} \label{sec_tdg}

{We incorporate the (truncated) DtN operators $T^\pm_M$ in the TDG formulation following the strategy of }
\cite{Kapita2018}. 
The main difference is that here we consider a periodic scatterer with a DtN boundary condition on {two} segments, while in \cite{Kapita2018} the scatterer is bounded and the DtN {condition} is posed on a {surrounding} circle. 
Moreover, in our formulation we solve the problem for the total field and in \cite{Kapita2018} the problem is solved for the scattered one.

The definitions of the numerical fluxes on the interior faces and on the Dirichlet faces on $\GD$ are taken as those of standard TDG methods in \cite[\S2.2.1]{Hiptmair2016}, \cite{Hiptmair2011}, and in particular \cite[Ch.~3]{Howarth2014} as we allow a discontinuous wavenumber $\kappa$.
On the artificial boundaries $\GpmH$, we propose new numerical fluxes, adapting the idea in \cite{Kapita2018}.

Let $\Mh = \{ K \}$ be a convex-polygonal finite element partition of $\Omega$, 
such that the permittivity $\varepsilon$ is constant in each element.
We write $\bn_K$ for the outward-pointing unit normal vector on $\partial K$, and 
$h = \max_{K\in\Mh}h_K$ for the mesh width of $\Mh$,  
where $h_K$ is the diameter of $K$.
We denote by $\nabla_h$ the elementwise application of $\nabla$ and write $\dn = \bn \cdot \nabla_h$ and $\dnK = \bn_K \cdot \nabla_h$ for the normal derivatives on $\partial \Omega$ and $\partial K$ respectively.

We require the mesh to be quasi-periodic: for each element $K\in\Mh$ with a face $F\subset\Gleft$, denoting its endpoints $(0,x_2^-)$ and $(0,x_2^+)$, there is a $K'\in\Mh$ with a face $F'\subset\Gright$ with endpoints $(L,x_2^-)$ and $(L,x_2^+)$. 
Then $F$ and $F'$ are {identified and} treated as a single internal face. 
In particular, the union of the mesh and its copy translated by $\pm L$ in the direction $x_1$ is a conforming mesh; see Figure~\ref{fig:periodicity}.

\begin{figure}[htb]
\centering
\includegraphics[width=.75\textwidth]{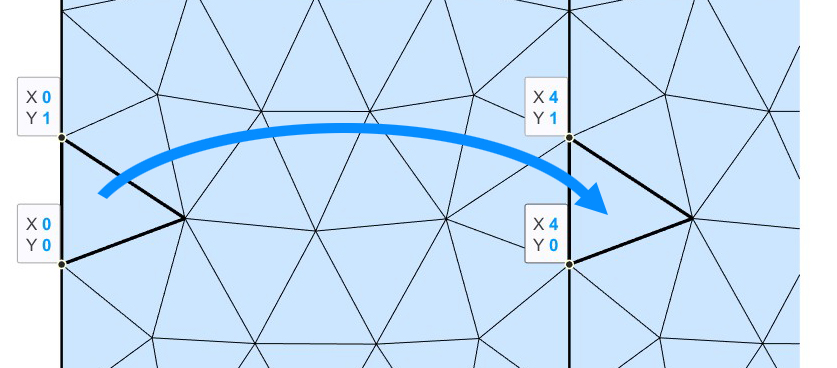}
\caption{Periodicity of the mesh. }
\label{fig:periodicity}
\end{figure}

We denote by $\Fh = \bigcup_{K \in \Mh } \partial K $ and $\FhI = \Fh \setminus(\GD\cup\GH\cup\GmH)$ the skeleton of the mesh and its inner part.
Note that the faces in $\Gleft\cup\Gright$ are part of $\FhI$.

Given two adjacent elements $K_1, K_2\in \Mh$, an elementwise-smooth function $v$ and vector field $\pmb\tau$ on $\Mh$, we introduce on $\partial K_1 \cap \partial K_2$ the averages and the normal jumps:
\begin{align} \label{jumps}\begin{aligned}
	\dlgraffa v\drgraffa & := \frac{1}{2}\left( v_{|K_1} + v_{|K_2} \right), & 
	\dlgraffa\pmb\tau\drgraffa & := \frac{1}{2}\left( \pmb\tau_{|K_1} + \pmb\tau_{|K_2} \right), \\ 
	\llbracket v \rrbracket_N & := v_{|K_1}\bn_{K_1} + v_{|K_2}\bn_{K_2},\qquad & 
	\llbracket \pmb\tau \rrbracket_N & := 
	\pmb\tau_{|K_1}\cdot \bn_{K_1} +  \pmb\tau_{|K_2}\cdot \bn_{K_2}. 
	\end{aligned}\end{align}
	These face quantities must take into account quasi-periodicity:
	if $K_1$ has two vertices with coordinates $(0,x_2^-)$ and $(0,x_2^+)$, and
	$K_2$ has two vertices with coordinates $(L,x_2^-)$ and $(L,x_2^+)$, when writing the average and jump formulae \eqref{jumps} on their common face $v_{|K_1}$ and $\pmb\tau_{|K_1}$ have to be replaced by $v_{|K_1}\re^{\ri\alpha_0L}$ and $\pmb\tau_{|K_1}\re^{\ri\alpha_0L}$.
	
	We then define the two main spaces where Trefftz functions live, the broken Sobolev space $H^s(\Mh)$ and the Trefftz space $T(\Mh)$:
	\begin{align} \label{broken_sob}
H^s(\Mh) :=&\ \{ v \in L^2(\Omega) : v_{| K}\in H^s(K) \hspace{0.2cm} \forall K \in \Mh \}, \hspace{0.3cm} \text{for} \hspace{0.2cm} s>0, \\ 
T(\Mh) :=&\ \{ v \in H^1(\Mh) : \Delta v + \kappa^2v = 0 \; \text{in} \; K \; \text{and} \; \dnK v \in L^2(\partial K) \hspace{0.2cm}\forall K \in \Mh \}. \nonumber
\end{align}
The discrete Trefftz space $V_p(\Mh)$ is a finite-dimensional subspace of $T(\Mh)$ and can be represented as $V_p(\Mh) = \bigoplus_{K \in \Mh}V_{p_K}(K)$, where $V_{p_K}(K)$ is a $p_K$-dimensional subspace of $T(\Mh)$ of functions supported in $K$.

We derive the TDG formulation following \cite{Hiptmair2011}.
We multiply the Helmholtz equation by a test function $v$ and integrate by parts twice on each $K\in\Mh$. We then replace $u$ and $v$ by discrete functions $u_h, v_h \in V_p(\Mh)$; on $\partial K$ we replace the trace of $u_h$ and $\nabla u_h$ by the numerical fluxes $\hat{u}_h$ and $-\widehat{\ri\kappa\sigma}_h$, that are single-valued approximations of $u_h$ and $\nabla u_h$, respectively, on each face, obtaining 
\begin{equation} \label{TDG}
\int_{\partial K} \hat{u}_h \, \overline{\partial_\bn v_h}  \rd s 
+ \int_{\partial K} \widehat{\ri\kappa\sigma}_h \cdot \bn \, \overline{v_h}  \rd s = 0,
\end{equation}
where no volume terms appears because the test field $v_h\in V_p(\Mh)$ is Helmholtz solution in $K$.

We define the DtN-TDG numerical fluxes
$\hat u_h$ and $\widehat{\ri\kappa\sigma}_h$. 
Note that the penalty terms have opposite signs to those in e.g.\ \cite{Hiptmair2011}, because of the opposite time-harmonic sign convention, as mentioned in \S\ref{secmod}.
On the interior faces in $\FhI$ we choose
\begin{equation*}
\begin{cases}
	\hat{u}_h = \dlgraffa u_h\drgraffa  -\ri \xi^{-1}\pb
	\llbracket \nabla_h u_h  \rrbracket_N, & \\
	\widehat{\ri\kappa\sigma}_h = - \dlgraffa \nabla_h u_h\drgraffa  - \ri \xi \pa  \llbracket u_h  \rrbracket_N, & 
\end{cases}
\end{equation*}
where  $\pa\in L^\infty(\FhI\cup \Gamma_D)$, $\pb\in L^\infty(\FhI)$ are positive flux coefficients, and $\xi$ is defined,  following \cite[\S3.3]{Howarth2014}, on an internal face $F=\partial K_1\cap\partial K_2$ with $K_1,K_2\in\Mh$ as
\begin{equation} \label{xi_2}
\xi = {\frac{\Re (\kappa_{|K_1}) + \Re (\kappa_{|K_2})}2.}
\end{equation}

On the boundary faces, we define the numerical fluxes as
\begin{align*}
& \begin{cases}
	\hat{u}_h = 0, & \\
	\widehat{\ri\kappa \sigma}_h = - \nabla_h u_h - \ri \kappa \pa u_h \bn &  
\end{cases} & \text{on }\GD, \\
& \begin{cases}
	\hat{u}_h = u_h + -\ri\kappa^{-1}\pd\left(\nabla_h  u_h \cdot \bn - T_M^+ u_h + 2\ri\beta_0^+ u^\inc \right), & \\
	\widehat{\ri \kappa \sigma}_h = -  T_M^+ u_h \bn + 2\ri\beta_0^+ u^\inc \bn 
	+\ri\kappa^{-1}\pd T^{+,*}_M \left( \nabla_h  u_h -T_M^+ u_h \bn + 2\ri\beta_0^+ u^\inc \bn \right) &
\end{cases} & \text{on }\GH, \\
& \begin{cases}
	\hat{u}_h = u_h -\ri\kappa^{-1}\pd\left( \nabla_h  u_h \cdot \bn - T_M^- u_h \right), & \\
	\widehat{\ri \kappa \sigma}_h = - T_M^- u_h \bn +\ri\kappa^{-1}\pd T^{-,*}_M \left( \nabla_h  u_h -T_M^- u_h \bn \right) &
\end{cases} &\text{on }\GmH,
\end{align*}
where $\pd \in L^\infty(\GH\cup\GmH)$ is a positive flux coefficient, and $T^{\pm,*}_M$ is the $L^2(\GH)$-adjoint of $T^\pm_M$:
\begin{equation*}
\int_\GpmH T^{\pm,*}_M v \, \cw \rd s
= \int_\GpmH v \, \overline{T^\pm_M w} \rd s
\qquad \forall v,w\in H^{1/2}_{\alpha_0}(\GpmH), \qquad M \in \IN \cup \{\infty\}.
\end{equation*}

Substituting the numerical fluxes in (\ref{TDG}), summing over all the mesh elements, and separating the terms depending on $u^\inc$, we get the following TDG scheme: Find $u_h^M \in V_p(\Mh)$ such that
\begin{equation}\label{eq:TruncDG}
\cA_h^M(u_h^M, v_h) =\ell_h^M(v_h) 
\quad \forall v_h \in V_p(\Mh), \qquad M\in \IN\cup\{\infty\},
\end{equation}
where
\begin{align}\nonumber
\cA_h^{ M}(u, v) :=
& \int_{\FhI} \left( \dlgraffa u\drgraffa \llbracket \overline{\nabla_h v} \rrbracket_N  
-  \dlgraffa\nabla_h u\drgraffa \cdot \llbracket \cv \rrbracket_N 
- \ri\xi \,  {\pa} \, \llbracket  u\rrbracket_N \cdot \llbracket \cv \rrbracket_N    
- \ri\xi^{-1} {\pb} \, \llbracket \nabla_h u\rrbracket_N  \, \llbracket \overline{\nabla_h v} \rrbracket_N \right) \rd s \\ 
\label{DtN_bilinear_exact}
& + \int_\GH \biggl( u \, \overline{\dn v }  - T_M^+u \, \overline{v}
-{\pd} \, \ri\kappa^{-1} \left(\dn u -T_M^+ u\right) \, \overline{\left(\dn v -T_M^+ v\right)} \biggl) \rd s \\ \nonumber
& + \int_\GmH \biggl(  u \, \overline{\dn v}  - T_M^-u \, \overline{v}
-{\pd} \, \ri\kappa^{-1} \left(\dn u -T_M^- u\right) \, \overline{\left(\dn v -T_M^- v\right)} \biggl) \rd s \\ \nonumber
& + \int_\GD \left( -\dn u \, \cv - \ri\kappa \,{\pa} \, u\, \cv \right) \rd s,
\end{align}
and 
\begin{equation} \label{operator}
\ell_h^M(v) :=  \int_\GH - 2\ri\beta_0^+ \, u^\inc \left( \cv - {\pd} \, \ri\kappa^{-1} \, \overline{ (\partial_\bn v -T_M^+ v )} \, \right) \rd s.
\end{equation}  
In practice, the sesquilinear form $\cA_h^M$ and the linear functional $\ell_h^M$ can be computed numerically only for $M<\infty$.

\subsection{DtN-TDG error analysis}

We develop our analysis under the non-trapping conditions \eqref{eq:NonTrap}. Under this assumptions, we can guarantee the existence and the uniqueness of the solution to the scattering problems \eqref{nontrunc} and \eqref{trunc}.
The consistency of the numerical fluxes gives the consistency of the method, \cite{Arnold2002}.
\begin{prop} \label{prop:consistent}
The DtN-TDG method is consistent, i.e.\ $\cA_h^M(u^M,v_h)=\ell_h^M(v_h)$ for $u^M$ solution of \eqref{trunc} and any $v_h\in V_p(\Mh)$.
\end{prop}

We rewrite the sesquilinear form \eqref{DtN_bilinear_exact} by integrating by parts elementwise:
\begin{align} \label{eqn:bilinear_exact_new}
\cA^{ M}_h(u, v) =&\ \int_\Omega \left( \nabla_h u \cdot \nabla_h \cv -\kappa^2 u \cv \right) \rd\bx  \\ \nonumber
&+ \int_{\FhI} \bigg( - \dlgraffa \nabla_h u \drgraffa \cdot \overline{\llbracket v \rrbracket}_N - \llbracket u \rrbracket_N \cdot \overline{\dlgraffa \nabla_h v \drgraffa} \\ 
\nonumber 
&\hspace{20mm}
-\ri\xi \,\pa \, \llbracket u \rrbracket_N \cdot \overline{\llbracket v \rrbracket}_N 
-\ri\xi^{-1}\pb \, \llbracket \nabla_h u \rrbracket_N \, \overline{\llbracket \nabla_h v \rrbracket}_N  \bigg) \rd s \\
\nonumber
& + \int_\GH \biggl( - T_M^+u \, \overline{v} -\pd \, \ri\kappa^{-1} \, \left(\partial_\bn u -T^+_{ M} u\right) \, \overline{\left(\partial_\bn v -T_{ M}^+ v\right)} \biggl) \rd s \\
\nonumber
& + \int_\GmH \biggl( - T_M^-u \, \overline{v} -\pd \, \ri\kappa^{-1} \, \left(\partial_\bn u -T_{ M}^- u\right) \, \overline{\left(\partial_\bn v -T^-_{ M} v\right)} \biggl) \rd s \\
& + \int_\GD \left(- u \, \overline{\partial_\bn v}  -\partial_\bn u \, \cv -\ri\kappa \,\pa \, u\, \cv \right) \rd s
\qquad\qquad{\forall u,v\in T(\Mh)}
. \nonumber
\end{align}
Integrating by parts elementwise again as in \cite[eq.~(39)]{Kapita2018}, and using the Trefftz property of $u$, we obtain
\begin{align} \label{eqn:bilinear_magic}
\cA^{ M}_h(u, v) =&\ \int_{\FhI} \bigg( \llbracket \nabla_h u \rrbracket_N \overline{\dlgraffa v \drgraffa} - \llbracket u \rrbracket_N \cdot \overline{\dlgraffa \nabla_h v \drgraffa} \\ 
\nonumber 
&\hspace{20mm}
-\ri\xi \,\pa \, \llbracket u \rrbracket_N \cdot \overline{\llbracket v \rrbracket}_N 
-\ri\xi^{-1}\pb \, \llbracket \nabla_h u \rrbracket_N \, \overline{\llbracket \nabla_h v \rrbracket}_N  \bigg) \rd s \\
\nonumber
& + \int_\GH \biggl( \left(\partial_\bn u -T^+_{ M} u\right) \overline{v} -\pd \, \ri\kappa^{-1} \, \left(\partial_\bn u -T^+_{ M} u\right) \, \overline{\left(\partial_\bn v -T_{\ M}^+ v\right)} \biggl) \rd s \\
\nonumber
& + \int_\GmH \biggl( \left(\partial_\bn u -T_{ M}^- u\right) \overline{v} -\pd \, \ri\kappa^{-1} \, \left(\partial_\bn u -T_{ M}^- u\right) \, \overline{\left(\partial_\bn v -T^-_{ M} v\right)} \biggl) \rd s \\
& + \int_\GD \left(- u \, \overline{\partial_\bn v} -\ri\kappa \,\pa \, u\, \cv \right) \rd s
\qquad\qquad{\forall u,v\in T(\Mh)}
. \nonumber
\end{align}
We define two mesh-dependent seminorms on the Trefftz space $T(\Mh)$, for any $M \in \IN$:
\begin{align} \nonumber
\vertiii{w}_{\mathrm{TDG}_M}^2 := & \left\| \xi^{\frac{1}{2}} \pa^{\frac{1}{2}} \llbracket w \rrbracket_N \right\|^2_{L^2(\FhI)} + \left\| \xi^{-\frac{1}{2}} \pb^{\frac{1}{2}} \llbracket \nabla_h w \rrbracket_N \right\|^2_{L^2(\FhI)}  \\ \nonumber
& + \left\| \kappa^{-\frac{1}{2}} \pd^{\frac{1}{2}} ( \partial_\bn w - T_M^+ w ) \right\|^2_{L^2(\GH)} + \left\| \kappa^{-\frac{1}{2}} \pd^{\frac{1}{2}} ( \partial_\bn w - T_M^- w ) \right\|^2_{L^2(\GmH)} \\ \label{normaTDG}
&+ \left\| \kappa^{\frac{1}{2}} \pa^{\frac{1}{2}} w \right\|^2_{L^2(\GD)},
\\ \nonumber
\vertiii{w}_{\mathrm{TDG}_M^+}^2 := & \vertiii{w}_{\mathrm{TDG}_M}^2  +\left\| \xi^{-\frac{1}{2}} \pa^{-\frac{1}{2}} \dlgraffa \nabla_h w \drgraffa \right\|^2_{L^2(\FhI)} + \left\| \xi^{\frac{1}{2}} \pb^{-\frac{1}{2}} \dlgraffa w \drgraffa \right\|^2_{L^2(\FhI)} \\ \label{normaTDG+}
& +  \left\| \kappa^{\frac{1}{2}} \pd^{-\frac{1}{2}} w \right\|^2_{L^2(\GH)} +  \left\| \kappa^{\frac{1}{2}} \pd^{-\frac{1}{2}} w \right\|^2_{L^2(\GmH)} + \left\| \kappa^{-\frac{1}{2}} \pa^{-\frac{1}{2}} \partial_\bn w \, \right\|^2_{L^2(\GD)}.
\end{align}

\begin{prop} \label{prop:norma}
Assume the non-trapping assumption \eqref{eq:NonTrap}, that $\varepsilon$ is real, and $M\ge M_\star$ as in \eqref{eq:MLarge}, including the case $M=\infty$.
Then the seminorm $\vertiii{\cdot}_{\mathrm{TDG}_M}$ is actually a norm on the Trefftz space $T(\Mh)$, and we have the coercivity inequality
\begin{equation}\label{eq:Coercivity}
	-\Im \cA^M_h(w, w) \geq \vertiii{w}_{\mathrm{TDG}_M}^2
	\qquad {\forall w\in T(\Mh)}.
\end{equation}
Moreover,
\begin{align}\label{eq:ContinuityA}
	\begin{aligned}
		\cA^M_h(v, w) \leq 2 \vertiii{v}_{\mathrm{TDG}_M} \, \vertiii{w}_{\mathrm{TDG}_M^+}
		\\ 
		\ell_h^M(v)\le 2|\sin\theta|\sqrt{\|\pd\|_{L^\infty(\GH)}\kappa^+ L}\, \vertiii{v}_{\mathrm{TDG}_M^+}
	\end{aligned}\qquad{\forall v,w\in T(\Mh)}.
\end{align}
We also have that the discrete problem \eqref{eq:TruncDG} has a unique solution $u_h^M \in V_p(\Mh)$. 
Let $u^M$ be the unique solution of the truncated boundary value problem \eqref{trunc}, and $u_h^M \in V_p(\Mh)$ the unique solution of the discrete DtN-TDG problem \eqref{eq:TruncDG}. 
Then:
\begin{equation} \label{quasiopt}
	\vertiii{u^M - u_h^M}_{\mathrm{TDG}_M} \leq 2 \inf_{v_h \in V_p(\Mh)} \vertiii{u^M - v_h}_{\mathrm{TDG}_M^+}.
\end{equation}
\end{prop}
\begin{proof}
Let $w \in T(\Mh)$ be such that $\vertiii{w}_{\mathrm{TDG}_M}= 0$.
The jumps of $w$ on mesh faces vanish, implying that $w\in H^1_{\alpha_0}(\Omega)$ and that $w$ satisfies the Helmholtz equation in $\Omega$.
Moreover, $w = 0$ on $\GD$ and $\nabla_h w \cdot \bn - T_M^\pm w = 0$ on $\GpmH$, so $w$ is solution of \eqref{trunc} with $u^{\mathrm{inc}}=0$.
The well-posedness stated in Theorem~\ref{wellposed} for $M=\infty$, and in Corollary~\ref{cor:TruncWP} for $M<\infty$, implies that $w=0$, so $\vertiii{\cdot}_{\mathrm{TDG}_M}$ is a norm.

From the expression of $\cA_h^M$ in \eqref{eqn:bilinear_exact_new},
we have, for all $w\in T(\Mh)$,
\begin{align*}
	-\Im \cA_h^M(w, w) = &\ 
	\left\| \xi^{-\frac{1}{2}} \pb^{\frac{1}{2}} \llbracket \nabla_h w \rrbracket_N \right\|^2_{L^2(\FhI)} 
	+ \left\| \xi^{\frac{1}{2}}\pa^{\frac{1}{2}} \llbracket w \rrbracket_N \right\|^2_{L^2(\FhI)}
	+\left\| \kappa^{\frac{1}{2}} \pa^{\frac{1}{2}} w \right\|^2_{L^2(\GD)}\\
	& + \Im \int_\GH T_M^+ w \, \cw \rd s 
	+ \left\| \kappa^{-\frac{1}{2}} \pd^{\frac{1}{2}} ( \partial_\bn w - T_M^+ w ) \right\|^2_{L^2(\GH)}\\
	& + \Im \int_\GmH T_M^- w \, \cw \rd s 
	+ \left\| \kappa^{-\frac{1}{2}} \pd^{\frac{1}{2}} ( \partial_\bn w - T_M^- w ) \right\|^2_{L^2(\GmH)}\\ \ge &\ 
	\vertiii{w}^2_{\mathrm{TDG}_M},
\end{align*}
where the inequality follows from Lemma \ref{lemma:l2prod} and Remark~\ref{imag_rem}.

The continuity \eqref{eq:ContinuityA} of the sesquilinear form follows from the application of Cauchy--Schwarz inequality to all terms in \eqref{eqn:bilinear_magic}. 
The coefficient 2 arises from the terms on $\GpmH$. 
The continuity of $\ell_h^M$ uses that $\|u^\inc\|_{\GH}=\sqrt L$ and $\beta_0^+=-\kappa^+\sin\theta$.

The coercivity \eqref{eq:Coercivity} implies the well-posedness of the discrete Galerkin problem \eqref{eq:TruncDG}, for any finite-dimensional $V_p(\Mh)\subset T(\Mh)$.

Finally, by the first part of the proposition, the quasi-optimality \eqref{quasiopt} follows from
\begin{align*}
	\vertiii{u^M - u_h^M}_{\mathrm{TDG}_M}^2 & \leq -\Im \, \cA_h^M (u^M - u_h^M, u^M - u_h^M) \\
	& \leq | \cA_h^M (u^M - u_h^M, u^M - u_h^M) | \\
	& = | \cA_h^M (u^M - u_h^M, u^M - v_h) | \\
	& \leq 2 \vertiii{u^M - u_h^M}_{\mathrm{TDG}_M} \, \vertiii{u^M - v_h}_{\mathrm{TDG}_M^+}
	\qquad \forall v_h\in V_p(\Mh).
\end{align*}
\end{proof}

See Remark \ref{rem:PWapprox} for a brief discussion of the approximation bounds that can be deduced from \eqref{quasiopt} for a discrete space $V_p(\Mh)$ made of plane waves.

\begin{rem}[Combined truncation and discretization error]
To control the error $e_h^M:=u^\infty-u_h^M$ taking into account both the TDG discretization and the DtN truncation, it seems natural to use a Strang-type argument.
Assume that $u^\infty$, solution of \eqref{weak}, belongs to $T(\Mh)$, which follows from elliptic regularity under suitable assumptions (e.g.~\cite[Lemma~3.1]{Lechleiter2010}).
Then, it satisfies $\cA_h^\infty(u^\infty,v)=\ell_h^\infty(v)$ for all $v\in T(\Mh)$, and 
\begin{align}\nonumber
	\|e_h^M\|_{\mathrm{TDG}_M}^2
	&\le |\cA_h^M(u^\infty-u_h^M,e_h^M)|
	=|\cA_h^M(u^\infty,e_h^M)-\ell_h^M(e_h^M)|\\
	&\le |\cA_h^M(u^\infty,e_h^M)-\cA_h^\infty(u^\infty,e_h^M)|+|\ell_h^\infty(e_h^M)-\ell_h^M(e_h^M)|.
	\label{eq:ehM}
\end{align}
The bound on $T^\pm-T^\pm_M$ in Proposition~\ref{prop:TminusTM} allows to control both these terms, exploiting the regularity of $u^\inc$ and $u^\infty$ on $\GpmH$.
For instance, the right-hand side difference is bounded by an $\mathcal O(M^{-t})$ term 
for any $t>0$ as
\begin{align*}
	|\ell_h^\infty(e_h^M)-\ell_h^M(e_h^M)|
	&=\Big|\int_\GH -2\ri\beta_0^+ u^\inc \pd\ri\kappa^{-1} \overline{(T^+_\infty e_h^M-T^+_M e_h^M)}\rd s\Big|
	\\
	&\le 2\sin\theta\|\pd\|_{L^\infty(\GH)}\|u^\inc\|_{1+t,\alpha_0} \|T^+_\infty e_h^M-T^+_M e_h^M\|_{-1-t,\alpha_0}
	\\
	&\le 2\sin\theta\|\pd\|_{L^\infty(\GH)}\|u^\inc\|_{1+t,\alpha_0} \Big(\frac{2\pi}L M-|\alpha_0|\Big)^{-t} 
	\|e_h^M\|_{0,\alpha_0}.
\end{align*}
However, $\|e_h^M\|_{0,\alpha_0}=\|e_h^M\|_{L^2(\GD)}$ can be bounded by $\|e_h^M\|_{\mathrm{TDG}_M^+}$ \eqref{normaTDG+} but not by the weaker $\|e_h^M\|_{\mathrm{TDG}_M}$ norm \eqref{normaTDG} (which we would like to obtain at the right-hand side, to cancel the second power in \eqref{eq:ehM}).
So this kind of argument does not easily allow to control the combined truncation/discretization error.
\end{rem}

\subsection{Plane-wave basis and linear system assembly}
We define a discrete plane-wave space.
For a mesh element $K \in \Mh$, we denote by $V_p(K)$ the plane wave space on $K$ spanned by $p$ plane waves, $p \in \mathbb{N}$:
\begin{equation} \label{pwspace}
V_p(K) = \bigg\{ \; v \in L^2(K) :  v(\bx) = \sum_{j=1}^p \eta_j 
\exp\{\ri\kappa\bd_j \cdot \bx \}, \hspace{0.3cm} \eta_j \in \mathbb{C} \; \bigg\},
\end{equation}
where $\{ \bd_j \}_{j=1}^p \subset \mathbb{R}^2$, with $|\bd_j|=1$, are different propagation directions.
To obtain isotropic approximations, uniformly-spaced directions can be chosen as 
$\bd_j = ( \cos\frac{2\pi j}p, \sin\frac{2\pi j}p)$, $j=1,\ldots,p$. 
For simplicity, we choose the same number $p$ of directions in every element $K\in\Mh$.
The value of $\kappa=k\sqrt{\varepsilon}$ depends on the region where the element $K$ is located;
recall that we consider meshes such that $\varepsilon$ and $\kappa$ are constant inside each element.
We define the global discrete space $V_p(\Mh)$ as
\begin{equation} \label{pwglobal}
V_p(\Mh)=\bigoplus_{K\in\Mh}V_p(K)
=\big\{ \; v \in L^2(\Omega) : v_{|K} \in V_p(K), \;\forall K \in \Mh \; \big\}.
\end{equation}
We denote by $\varphi_j$, for $j\in\{1,\ldots,N\}$ and $N=p\cdot\#\Mh$, the $j$-th element of the basis of $V_p(\Mh)$, defined as one of the exponentials in \eqref{pwspace} in one of the mesh elements and zero otherwise.
The matrix and the right-hand side of the linear system $A\mathbf u=\mathbf L$ corresponding to the DtN-TDG scheme \eqref{eq:TruncDG} have entries 
$A_{j,l}=\cA_h^M(\varphi_l,\varphi_j)$ and $L_j=\ell_h^M(\varphi_j)$, for $j,l=1,\ldots,N$.

An important advantage of the use of plane waves is that all matrix and vector entries can easily be computed analytically, as detailed below,
reducing the errors caused by numerical quadrature and the computational effort.

Recall that the mesh $\Mh$ is quasi-periodic in $x_1$, as in Figure~\ref{fig:periodicity}, and that $\Gleft$ is identified to $\Gright$.
Mesh faces contained in $\Gleft$ are treated as internal faces.
To assemble the system matrix, we have to include the interaction terms between basis function associated to elements respectively on the right and left side of the domain, which are adjacent to the same face in $\Gleft$.
In this interaction, we have to consider the quasi-periodicity of the basis functions, in the sense that we multiply the function associated to the left side by the quasi-periodicity constant $\re^{\ri\alpha_0 L}$. 
Multiplying the other function by $\re^{-\ri\alpha_0 L}$ gives the same result, since the test (but not the trial) functions are conjugated in all terms in $\cA_h^M$.

To compute the matrix entries, we first consider the case of two basis functions $\varphi_j,\varphi_l$ supported in the same element $K$ without faces on $\GpmH$.
Using that $\dnK \re^{\ri\kappa\bx\cdot\bd}=\ri\kappa\bn_K\cdot\bd\,\re^{\ri\kappa\bx\cdot\bd}$ and \eqref{DtN_bilinear_exact},
\begin{align}\nonumber
A_{j,l} = & \sum_{F\in \mathcal{F}_K, \; F\not\subset \GD} 
\left[ -\frac12\ri\kappa(\bd_j\cdot\bn_K+\bd_l\cdot\bn_K)
- \ri\pb\, \xi^{-1} \kappa^2 \bd_l\cdot\bn_K\,\bd_j\cdot\bn_K -\ri\pa  \xi \, \right] \int_{F} 
\re^{\ri\kappa\bx\cdot(\bd_l-\bd_j)} \rd s\\
& + \sum_{F\in \mathcal{F}_K, \; F\subset \GD}  \ri\kappa\left[ \,-\pa-\bd_l \cdot \bn \,\right] \int_{F} \re^{\ri\kappa\bx\cdot(\bd_l-\bd_j)} \rd s, 
\label{eq:AjlInner}
\end{align}
where $\mathcal{F}_K$ is the set of the faces of $K$.
When $\varphi_l$ and $\varphi_j$ are supported in the elements $K_1$ and $K_2\in\Mh$, respectively, and these share the face $F=\partial K_1\cap\partial K_2$, 
\begin{align}\label{eq:AjlInterface}
A_{j,l} & = \left[\frac12(\ri\kappa_2\bd_j\cdot\bn+\ri \kappa_1\bd_l\cdot\bn) 
+\ri\pb\xi^{-1} \kappa_1 \kappa_2\bd_l\cdot\bn\,\bd_j\cdot\bn 
+\ri\pa\xi\,\right] \int_F \re^{ \ri\kappa_1\bx\cdot\bd_l - \ri\kappa_2\bx\cdot\bd_j}\rd s,
\end{align}
with $\bn=\bn_{K_1}$, $\kappa_j=\kappa|_{K_j}$ and $\xi$ as in (\ref{xi_2}).
All terms in \eqref{eq:AjlInner} and \eqref{eq:AjlInterface} are easily computed analytically from the local wavenumber $\kappa$, the plane wave directions $\bd_j$, the numerical flux parameters $\pa,\pb,\xi$, and the mesh.

We now focus on the faces located on $\GH$. 
For basis functions $\varphi_j,\varphi_l$ supported on elements with a face lying on $\GH$, we include in the $A_{j,l}$ matrix entry also the contribution corresponding to the integral over $\GH$ in \eqref{DtN_bilinear_exact}:
\begin{align} \nonumber
& \int_\GH \biggl(  \varphi_l \, \overline{\dn \varphi_j}  
- T_M^+ \varphi_l \, \overline{\varphi_j}
- \pd \, \ri\kappa^{-1} \, \left(\dn \varphi_l -T_M^+ \varphi_l\right) \, \overline{\left(\dn \varphi_j -T_M^+ \varphi_j\right)} \biggl) \rd s 
\\ \label{local}
&=  \int_\GH  \left( \varphi_l -\pd \, \ri\kappa^{-1} \, \dn \varphi_l \right) \, \overline{\dn \varphi_j} \rd s \\ \label{global}
&\quad - \int_\GH \biggl( T_M^+ \varphi_l \, \left( \overline{\varphi_j} - \pd \, \ri\kappa^{-1} \, \overline{\dn \varphi_j} \right) + \pd \, \ri\kappa^{-1} \, \left( T_M^+ \varphi_l \, \overline{T_M^+ \varphi_j} - \dn \varphi_l \, \overline{T_M^+ \varphi_j} \right) \biggl) \rd s.
\end{align}
The local term (\ref{local}) is non-zero only if the two basis functions $\varphi_l$ and $\varphi_j$ are associated to the same upper boundary element $K$. 
Then, 
\eqref{local} is
\begin{align*}
-\ri\kappa \, \bd_j\cdot\bn \, (1 + \pd \, \bd_l\cdot\bn)\,
\int_{\partial K\cap\GH} \re^{\ri\kappa\bx\cdot(\bd_l-\bd_j)}\rd s,
\end{align*}
with $\bn=(0,1)$.
Since $T_M^+\varphi_l$ is supported on all $\GH$, the global term \eqref{global} is non-zero for every basis function pair such that both the supports of $\varphi_j$ and $\varphi_l$ include a portion of $\GH$.
We first explicitly compute the $2M+1$ Fourier coefficients \eqref{fourier_coeff} of 
$\varphi_l(x_1,H)=\sum_{n\in\IZ}\varphi_l^n\re^{\ri\alpha_nx_1}$ 
corresponding to $|n|\le M$ as
\begin{align*}
\varphi_l^n  
= \frac1L\re^{\ri\kappa(\bd_l)_{_2}H} \int_{p_1}^{p_2}\re^{\ri(\kappa (\bd_l)_{_1} -\alpha_n) x_1}\rd x_1
=\begin{cases}
	\frac{\varphi_l(p_2,H)\re^{-\ri\alpha_n p_2}-\varphi_l(p_1,H)\re^{-\ri\alpha_n p_1}}{\ri L(\kappa (\bd_l)_{_1} -\alpha_n)}
	& \kappa (\bd_l)_{_1}\ne\alpha_n,\\
	\frac{p_2-p_1}L\re^{\ri\kappa(\bd_l)_{_2}H}
	& \kappa (\bd_l)_{_1}=\alpha_n,
\end{cases}
\end{align*}
where $p_1$, $p_2$ are the $x_1$-coordinates of the endpoints of the face of $K$ (support of $\varphi_l$) that lies on $\GH$, i.e.\ $\partial K\cap\GH=[p_1,p_2]\times \{H\}$.
The $|\varphi_l^n|=\mathcal O_{n\to\infty}(n^{-1})$ behaviour, due to the $\alpha_n$ at the denominator, is consistent with the regularity $\varphi_l(\cdot,H)\in H^{1/2-\epsilon}(\GH)$ for all $\epsilon>0$, due to its discontinuity at $p_1,p_2$ (recall \eqref{eq:HsNorm}).
The computation of all the Fourier coefficient has cost $\mathcal O(Mp/h)$, for a quasi-uniform mesh, but it is in practice negligible as the $\varphi_l^n$ formula is explicit.
Then, using the definition \eqref{DtN_trunc} of the truncated DtN, the term \eqref{global} can be written as
\begin{align*}
&-( 1 - \pd \, \bd_j \cdot \bn) \int_\GH T_M^+ \varphi_l \;  \overline{\varphi_j} \rd s 
- \pd \, \bd_l \cdot \bn \int_\GH \varphi_l \, \overline{T_M^+ \varphi_j} \rd s
- \pd \, \ri\kappa^{-1} \int_\GH T_M^+ \varphi_l \, \overline{T_M^+ \varphi_j} \rd s
\\
&=-( 1 - \pd \, \bd_j \cdot \bn) 
\;\ri \sum_{n=-M}^M \varphi_l^n \beta_n^+ \re^{-\ri\kappa \, (\bd_j)_{_2} H} \int_{\partial K_2\cap\GH} 
\re^{\ri(\alpha_n-\kappa (\bd_j)_{_1}) x_1}  \rd x_1
\\
&\quad + \pd \, \bd_l \cdot \bn 
\,\ri \sum_{n=-M}^M \overline{\varphi_j^n \beta_n^+} \re^{\ri\kappa \, (\bd_l)_{_2}\hspace{0.03cm} H} 
\int_{\partial K_1\cap\GH} \re^{\ri(\kappa \,(\bd_l)_{_1}-\alpha_n) \, x_1} \, \rd x_1
- \pd \, \ri\kappa^{-1}L\sum_{n=-M}^M   | \beta_n^+ |^2\;\varphi_l^n\; \overline{\varphi_j^n},
\end{align*}
where $\bn=(0,1)$ and $K_1,K_2\in\Mh$ are the supports of $\varphi_l,\varphi_j$, respectively.

Analogous results can be derived for the integrals on $\GmH$ in \eqref{DtN_bilinear_exact}.

Lastly, we focus on the right-hand side vector of the linear system $A\mathbf u=\mathbf L$.
The component $L_j$ is nonzero only if the basis function $\varphi_j$ is supported in an element $K$ with a face on $\GH$. 
Using the definition \eqref{operator} of $\ell_h^M$ and 
$u^\inc (x_1,x_2)=\re^{\ri\alpha_0 x_1 - \ri \beta_0^+ x_2}$, so that 
$\int_\GH u^\inc \, \overline{T_M^+ \varphi_j}\rd s 
=-\ri L\overline{\beta_0^+\varphi_j^0}\re^{-\ri\beta_0^+H}$,
we compute
\begin{align*}
L_j &= \ell^M_h(\varphi_j) 
=  \int_\GH - 2\ri\beta_0^+ \, u^\inc \left( \overline{\varphi_j} -\pd \, \ri\kappa^{-1} \, \overline{\partial_\bn \varphi_j} + \pd \, \ri\kappa^{-1} \, \overline{T_M^+ \varphi_j} \right) \rd s \\ 
& =  -2\ri\beta_0^+ \, (1 - \pd \, \bd_j \cdot \bn ) \int_\GH u^\inc \, \overline{\varphi_j} \rd s
+2 \, \pd \,\beta_0^+ \, \kappa^{-1} \int_\GH u^\inc \, \overline{T_M^+ \varphi_j} \rd s
\\
&= -2\ri\beta_0^+ \ (1 - \pd \, \bd_j \cdot \bn )\, \re^{-\ri(\beta_0^++\kappa(\bd_j)_2)H} 
\int_{\partial K\cap\GH} \re^{\ri(\alpha_0-\kappa(\bd_j)_1)x_1}\rd s
- 2 \ri L\, \pd |\beta_0^+|^2 \, \kappa^{-1}\overline{\varphi_j^0}\,\re^{-\ri\beta_0^+H}.
\end{align*}

\begin{rem}[No quadrature is needed]
All the matrix and vector entries can be computed with a closed formula, integrating complex exponentials over segments.
Thus no quadrature error is incurred, and the computational cost of the assembly is independent of the wavenumber.
\end{rem}

\begin{rem}[Convergence rates, instability, and evanescent plane waves]\label{rem:PWapprox} \
Plane wave spaces admit best-approximation bounds in Sobolev norms that converge exponentially in the local dimension $p$, and algebraically in the mesh size $h$, see \cite[Thm.~5.2, Cor.~5.5]{Moiola2011}.
The convergence rates are asymptotically faster than those ensured by polynomial spaces: \cite[eq.~(45)]{Moiola2011} gives that $u\in H^{k+1}(K)$ can be approximated in $H^j(K)$ norm with rate $\mathcal O((q/\log q)^{k+1-j})$ on a convex element $K\subset\IR^2$, with $q$ proportional to the dimension $p$ of the local plane wave space; the analogous result with rate $\mathcal O(q^{k+1-j})$  for a $N$-dimensional polynomial space require $q\sim\sqrt N$ (i.e.\ $q$ represents the polynomial degree).
For sufficiently regular solutions, combined with the quasi-optimality \eqref{quasiopt}, these bounds allow to prove convergence rates for the TDG discretization error $u^M-u^M_h$.

The plane wave convergence analysis can be extended to lossy materials ($\varepsilon\notin\IR$), where all plane waves are evanescent, with some modifications using Remarks 2.3.6, 3.3.4, 3.4.12 and 3.5.9 in \cite{AndreaPhD}.

However, if either $\GD$ or the constant-$\varepsilon$ regions have corners, the Helmholtz solution presents singularities.
An $hp$ approach, with local $h$-refinement near singularities and $p$-refinement away from them, may lead to root-exponential convergence in the number of degrees of freedom, see \cite{Hiptmair2016FoCM}.
However, this may also lead to strong numerical instabilities: these are visible in some of the plots of \S\ref{s:Numerics} as a flattening of the convergence rates for large values of $p$.

A promising remedy to this instability is the use of evanescent plane waves (EPWs): exponential Helmholtz solutions in the form $\re^{\ri\kappa\bx\cdot\bd}$ with $\bd\in\IC^2$ and $d_1^2+d_2^2=1$.
EPW discrete spaces have been successfully used to approximate Helmholtz solutions in \cite{Parolin2023,Robert2024}, showing considerable improvements over classical plane waves for singular and near-singular solutions.
\end{rem}

\section{Numerical experiments}\label{s:Numerics}

We report some numerical results relative to the DtN-TDG approximation of different boundary value problems, with smooth (\S\ref{s:2regions}, \ref{s:3regions}) and singular (\S\ref{s:quad}, \ref{s:bao}, \ref{s:dir}) solutions, with (\S\ref{s:dir}) and without impenetrable obstacles, including problems admitting infinitely many analytical solutions (\S\ref{s:non_uniq}).
Some of the domain configurations are taken from  
\cite{Bao2000}.
In all the examples, the space period is $L=2\pi$.

The DtN-TDG scheme has been implemented in MATLAB; all linear systems are solved with the ``backslash'' direct solver.
To create the triangulation on the domain $\Omega$, we use the MATLAB PDE toolbox, and to evaluate the $L^2(\Omega)$ and $H^1(\Omega)$ norms of the DtN-TDG error over the triangles of the mesh we use a Duffy quadrature rule, which is presented in \cite{Duffy1982} and implemented in \cite{Kubatko2024}. 
In all experiments we use as numerical flux parameters those corresponding to the original ultra weak variational formulation (UWVF) of Cessenat and Després \cite{Cessenat1998}, which are $\pa=\pb=\pd=\frac{1}{2}$ (see \cite[\S2.2.2]{Hiptmair2016}).

\subsection{Flat interface between two homogeneous materials}
\label{s:2regions}

We first consider the rectangular domain $\Omega=(0,2\pi)\times(-3,3)$ split by the horizontal line $\{x_2=0\}$ into two homogeneous regions:
$\varepsilon=\varepsilon^+=1$ in $\{x_2>0\}$, and $\varepsilon=\varepsilon^-$ in $\{x_2<0\}$ with $\Re\varepsilon^->1$.
For this simple setup, the exact solution is
\begin{equation*} 
u(x_1,x_2) = \begin{cases}
	\re^{\ri k(x_1 \cos \theta + x_2 \sin \theta)} + R\re^{\ri k(x_1 \cos \theta - x_2 \sin \theta)} & x_2>0,\\
	T\re^{\ri k(x_1 \cos \theta - x_2 \sqrt{\varepsilon^- -\cos^2\theta})} & x_2<0, \end{cases}
	\end{equation*}
	where the reflection and the transmission coefficients are
	\begin{equation*}
T = \frac{2 \sin \theta}{\sin \theta - \sqrt{\varepsilon^- -\cos^2\theta}}, \hspace{1.0cm} R = \frac{\sin \theta + \sqrt{\varepsilon^- - \cos^2\theta}}{\sin \theta - \sqrt{\varepsilon^- -\cos^2\theta}}.
\end{equation*}
We consider two examples: 
\begin{enumerate}[(i)]
\item 
a lossless medium with $\varepsilon^-=1.5$ and incoming plane wave direction $\theta=-\pi/3$, and 
\item 
a lossy medium with $\varepsilon^-=(1.25 + 0.1\ri)^2$ together with $\theta=-\pi/4$.
\end{enumerate}
In both the experiments, we choose the wavenumber $k=5$ and the mesh parameter $h=1.5$, resulting in a mesh of 36 triangles. 
Figures \ref{fig:no_abs} and \ref{fig:abs} show the numerical solutions with $p=30$, the corresponding errors, and the $L^2(\Omega)$ and $H^1(\Omega)$ error norms as the number of local plane wave functions $p$ increases.
In both cases we observe exponential convergence in $p$.
The error concentrates near the element boundaries, as typical for TDG schemes, see e.g.\ \cite{Robert2024}.

\begin{figure}[htb]%
\includegraphics[height=.2\textwidth]{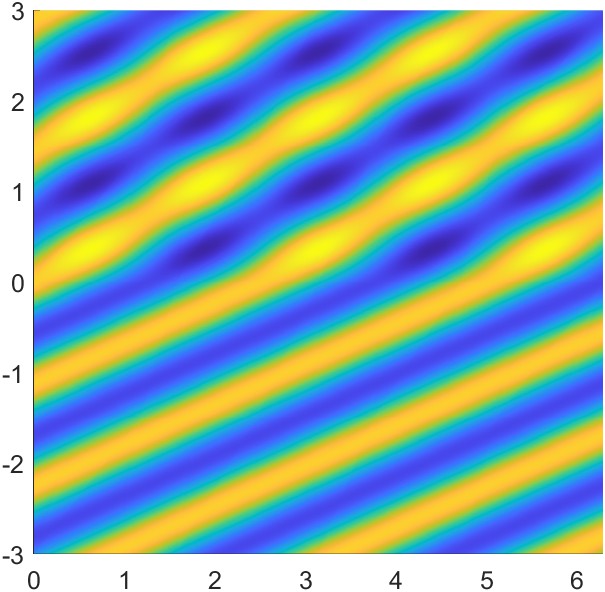}\hspace*{5pt}
\includegraphics[height=.2\textwidth]{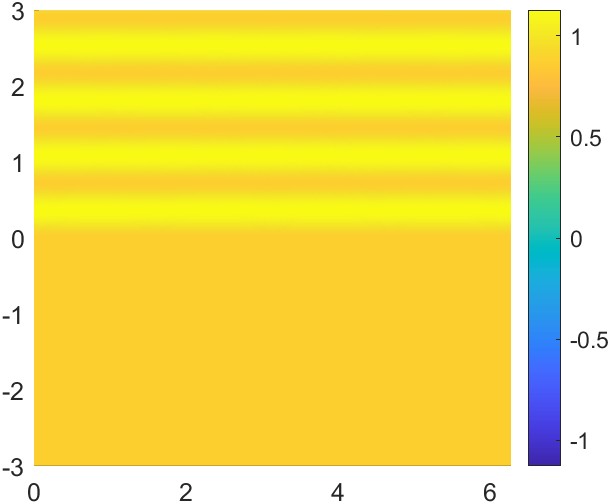}\hfill
\includegraphics[height=.2\textwidth]{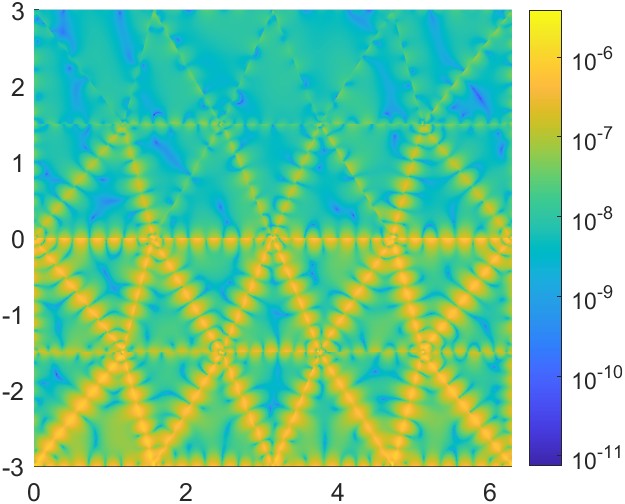}\hfill
\includegraphics[height=.21\textwidth,clip,trim=40 0 0 0]{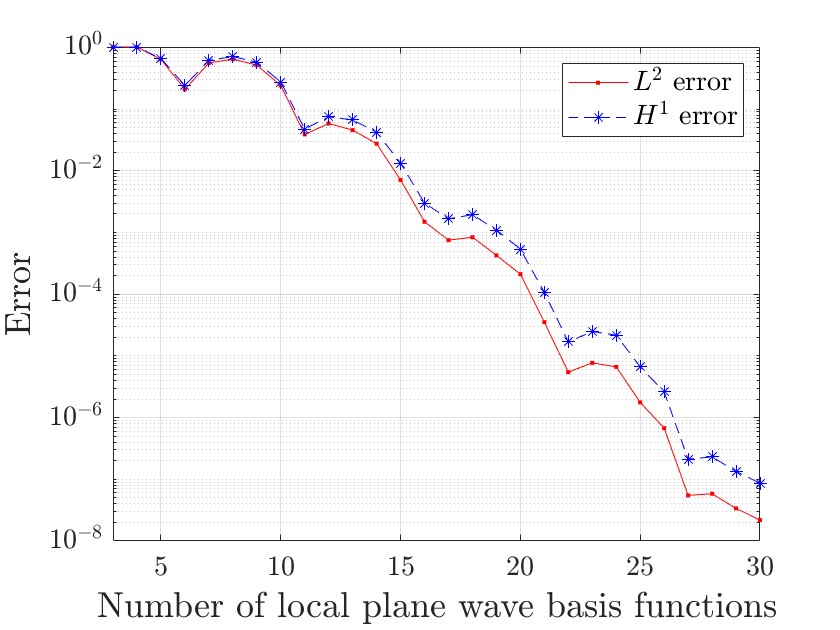}
\caption{
	Flat interface example of \S\ref{s:2regions}, lossless case (i) ($\varepsilon^-=1.5$, $\theta=-\pi/3$).
	Left to right: real part and absolute value of the numerical solution; 
	absolute value of the pointwise error (in logarithmic color scale) for $h=1.5$ and $p=30$; convergence of the $L^2(\Omega)$ and $H^1(\Omega)$ relative error norms for $p \in \{3,\ldots,30\}$ on the same mesh.
}
\label{fig:no_abs}
\end{figure}
\begin{figure}[htb]

\includegraphics[height=.2\textwidth]{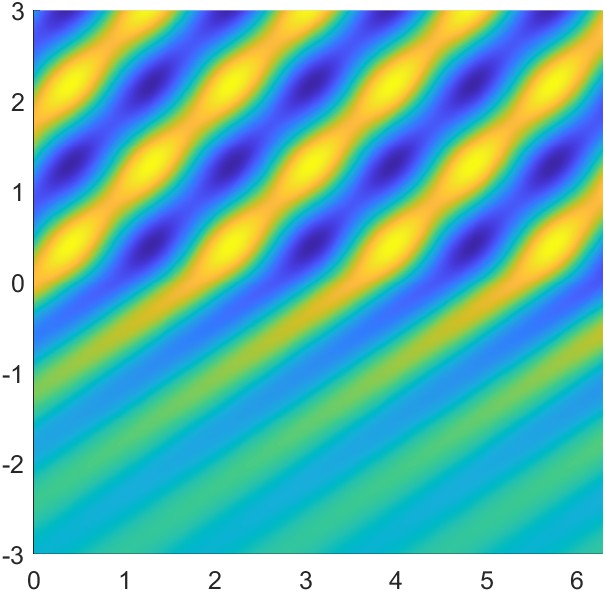}\hspace*{5pt}
\includegraphics[height=.2\textwidth]{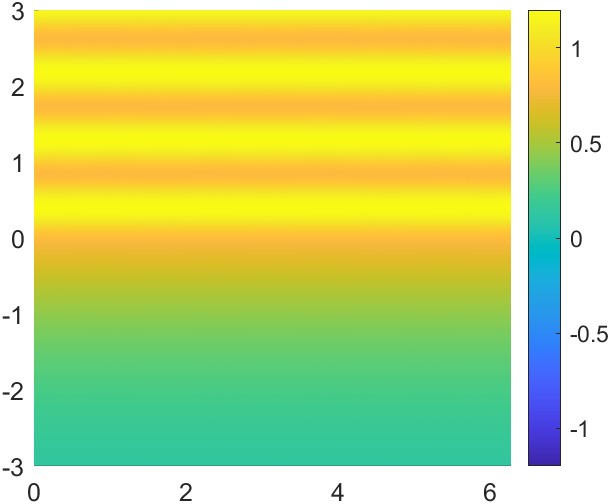}\hfill
\includegraphics[height=.2\textwidth]{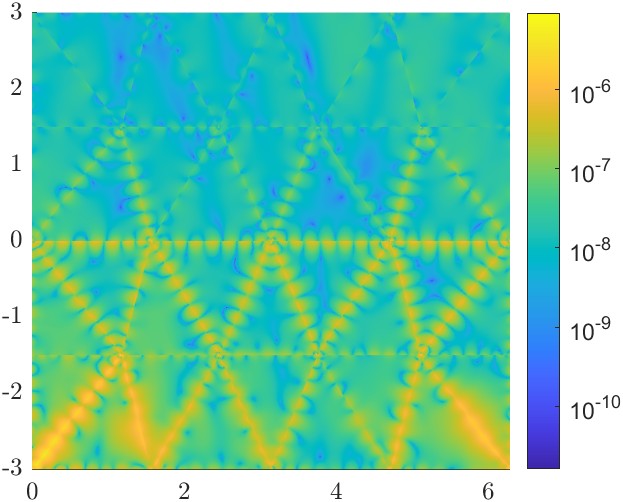}\hfill
\includegraphics[height=.21\textwidth,clip,trim=40 0 0 0]{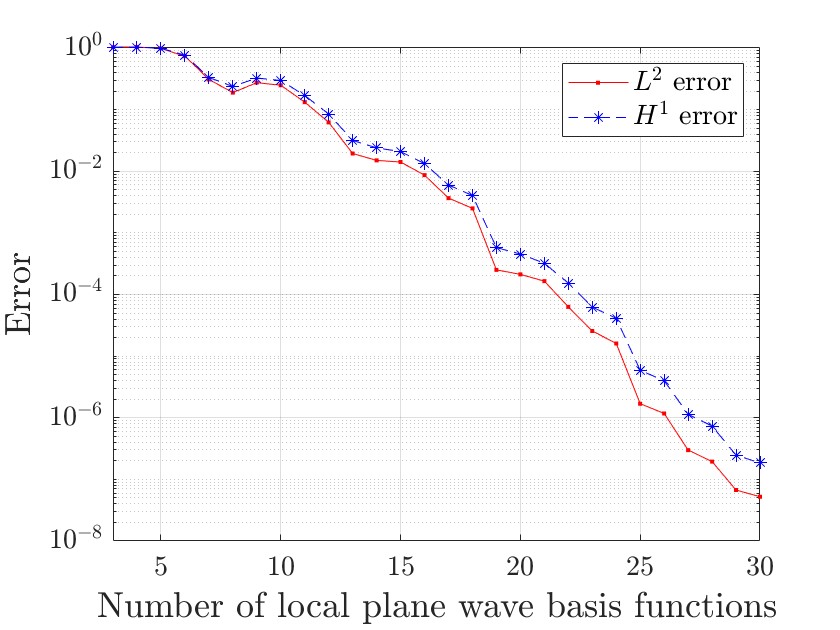}
\caption{Same plots as in Figure~\ref{fig:no_abs} for the lossy case (ii) of \S\ref{s:2regions} ($\varepsilon^-=(1.25 + 0.1\ri)^2$, $\theta=-\pi/4$).
}
\label{fig:abs}
\end{figure}

For this simple example (and the next one in \S\ref{s:3regions}), in the Fourier expansion of $u$ on any horizontal line, only the $u_0$ coefficient is non-zero.
Thus $u=u^M$ in \eqref{nontrunc}--\eqref{trunc} for all $M\in\IN$, and the truncation parameter $M$ in the DtN-TDG \eqref{eq:TruncDG} is irrelevant.

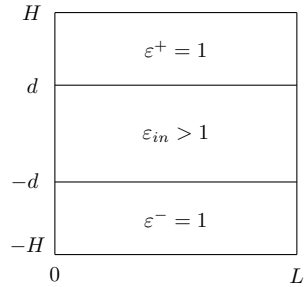
\begin{figure}[htb]
\centering
\begin{tikzpicture}[scale=0.8, every node/.style={scale=0.8}]
	\node[] at (-0.38,4) {$H$};
	\node[] at (-0.45,0.15) {$-H$};
	\node[] at (-0.5, 1.2) {$-d$};
	\node[] at (-0.37, 2.8) {$d$};
	\node[] at (0, -0.33) {$0$};
	\node[] at (4, -0.35) {$L$};
	\node[] at (2,3.4) {$\varepsilon^+ = 1$};
	\node[] at (2,2) {$\varepsilon_{in} > 1$};
	\node[] at (2,0.6) {$\varepsilon^- = 1$};
	\draw (0,0) -- (4,0); 
	\draw (0,0) -- (0,4); 
	\draw (0,1.2) -- (4,1.2);
	\draw (0,2.8) -- (4,2.8);
	\draw (0,4) -- (4,4);
	\draw (4,4) -- (4,0);
\end{tikzpicture}
\caption{Sketch of the domain for the examples in \S\ref{s:3regions} and \S\ref{s:non_uniq}.}
\label{fig:two_interface}
\end{figure}

\subsection{Two flat interfaces}
\label{s:3regions}

We separate the domain $\Omega=(0,2\pi)\times(-5,5)$ in three homogeneous regions:
$\varepsilon=\varepsilon^+=1$ for $x_2>d$ and $x_2<-d$, 
and 
$\varepsilon=\varepsilon_{in}>1$ for $-d<x_2<d$
(see Figure~\ref{fig:two_interface}).
Defining $\gamma := \sqrt{\varepsilon_{in} - \cos^2 \theta}$, the total field in $\Omega$ is:
\begin{equation} \label{sol_triple}
u(x_1,x_2)  = \begin{cases} 
	\re^{\ri k(x_1 \cos \theta + x_2 \sin \theta)} + R\re^{\ri k(x_1 \cos \theta - x_2\sin\theta)} & x_2>d,\\
	T_1\re^{\ri k(x_1 \cos \theta - x_2 \gamma)} + T_2\re^{\ri k(x_1 \cos \theta + x_2 \gamma)} & -d<x_2<d, \\
	T_3\re^{\ri k(x_1 \cos \theta + x_2 \sin \theta)} & x_2<-d. \end{cases}
	\end{equation}
	By enforcing the continuity of $u$ and $\partial_{x_2} u$ at $x_2 = d$ and $x_2 = -d$, we obtain the following 4-dimensional linear system for the reflection and transmission coefficients $R$, $T_1$, $T_2$, and $T_3$:
	\begin{equation*}
\begin{cases}
	\re^{-\ri kd\sin\theta} R - \re^{-\ri kd \gamma} \, T_1 - \re^{\ri kd \gamma} \, T_2
	= -\re^{\ri kd\sin \theta},\\
	-\sin \theta\, \re^{-\ri kd\sin\theta} R + \gamma \re^{-\ri kd \gamma} \, T_1 
	- \gamma \re^{\ri kd \gamma}\, T_2 = -\sin \theta \,\re^{\ri kd\sin\theta}, \\
	\re^{\ri kd \gamma} \, T_1 + \re^{-\ri kd \gamma} \, T_2 - \re^{-\ri kd\sin\theta} \, T_3 = 0, \\
	-\gamma \re^{\ri kd \gamma} \, T_1 + \gamma \re^{-\ri kd \gamma} \, T_2 - \sin \theta \, \re^{-\ri kd\sin\theta} \, T_3 = 0.
\end{cases}
\end{equation*}

We consider two instances:
\begin{enumerate}[(i)]
\item $k=5$, $\theta = - \pi / 3$, $\varepsilon_{in} = 2$, $d=2$, $h = 1.5$, resulting in a mesh of 68 triangles (Figure~\ref{fig:triple_eps2});
\item $k=5$, $\theta = - \pi / 4$, $\varepsilon_{in} = 10$, $d=2$, $h = 0.6$, resulting in a mesh of 392 triangles (Figure~\ref{fig:triple_eps10}).
\end{enumerate}
We observe again exponential $p$-convergence.

\begin{figure}[htb]
\centering

\includegraphics[height=.28\textwidth]{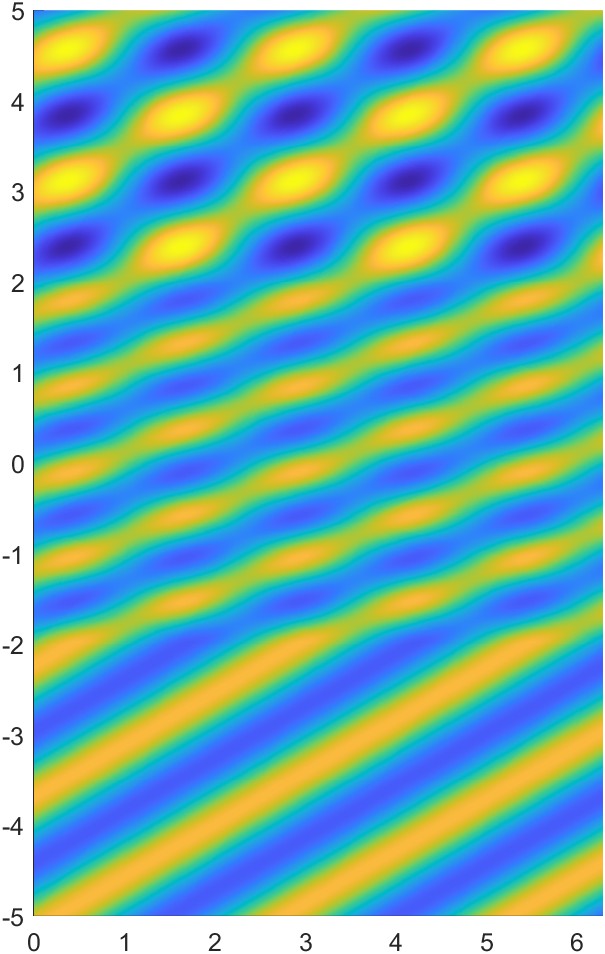}\hspace*{5pt}
\includegraphics[height=.28\textwidth]{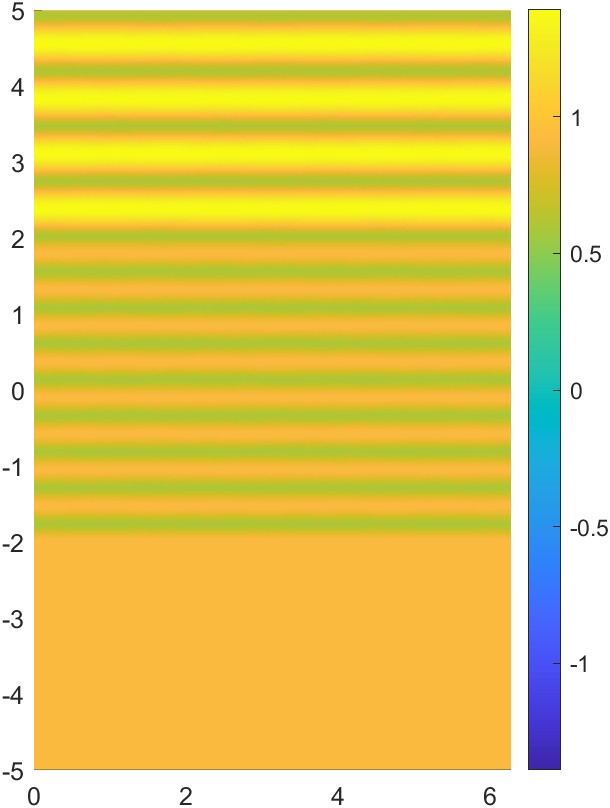}\hfill
\includegraphics[height=.28\textwidth]{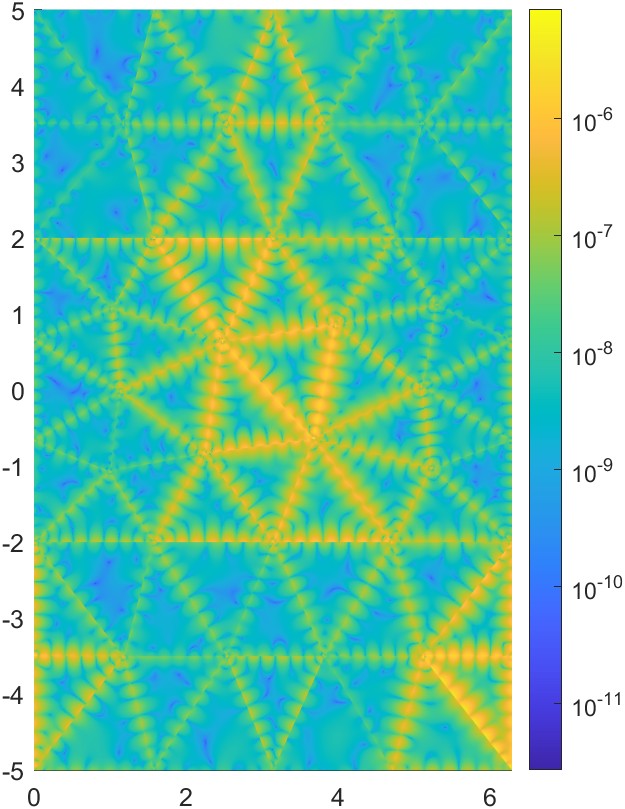}\hfill
\includegraphics[height=.29\textwidth,clip,trim=40 0 0 0]{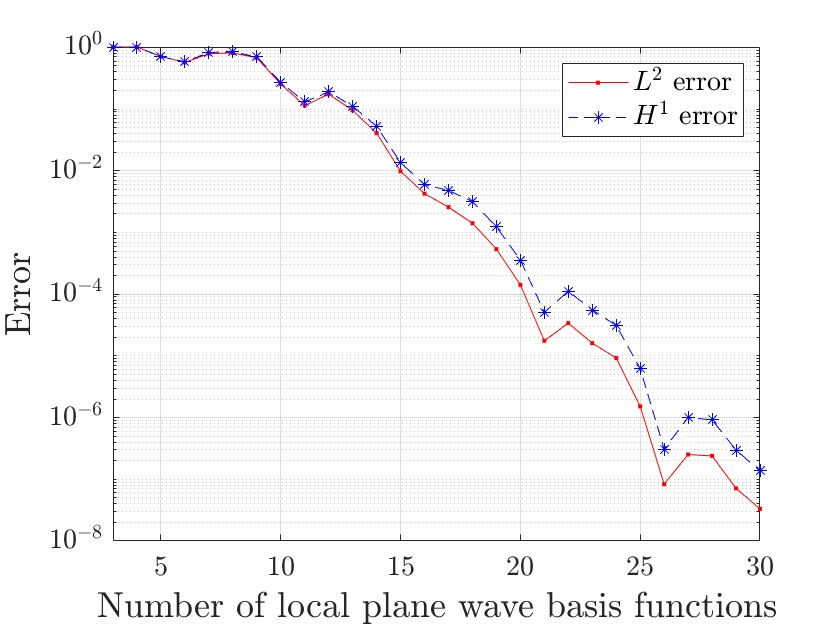}
\caption{Example with two interfaces of \S\ref{s:3regions}, low-contrast case (i) ($\varepsilon_{in}=2$, $\theta=-\pi/3$).
	Left to right: real part and absolute value of the numerical solution (with the same color bar); 
	absolute value of the pointwise error (in logarithmic color scale) for $h=1.5$ and $p=30$; convergence of the $L^2(\Omega)$ and $H^1(\Omega)$ relative error norms for $p \in \{3,\ldots,30\}$.
}
\label{fig:triple_eps2}
\end{figure}

\begin{figure}[htb]
\centering
\includegraphics[height=.28\textwidth]{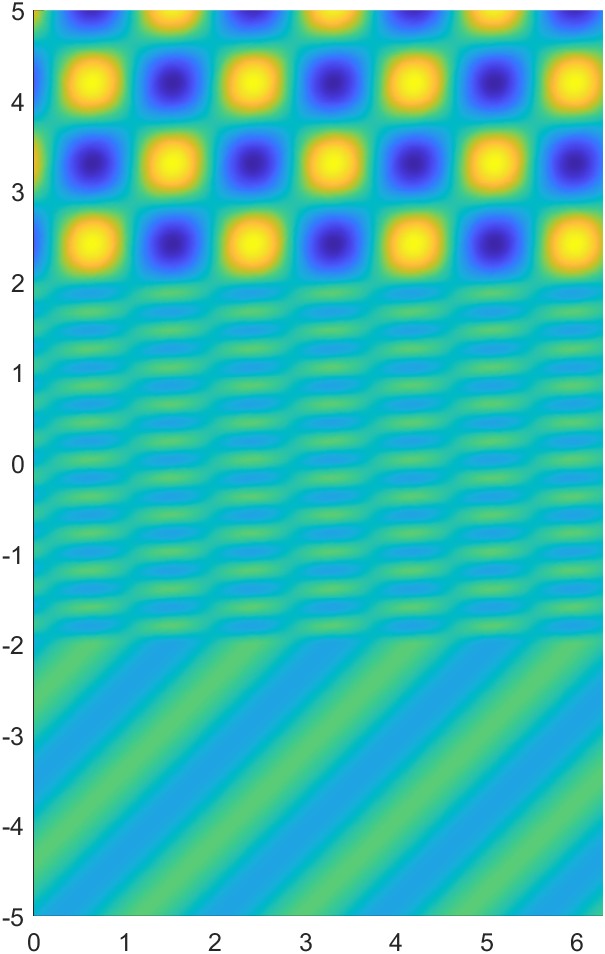}\hspace*{5pt}
\includegraphics[height=.28\textwidth]{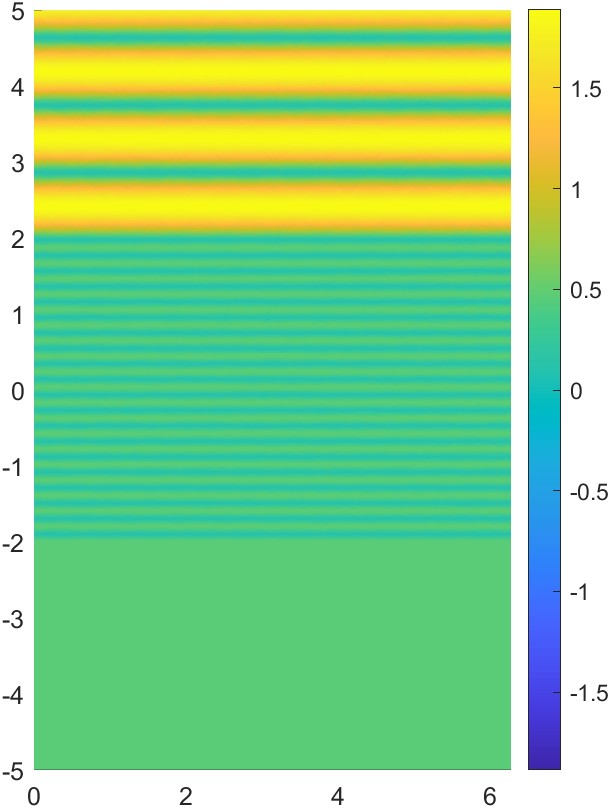}\hfill
\includegraphics[height=.28\textwidth]{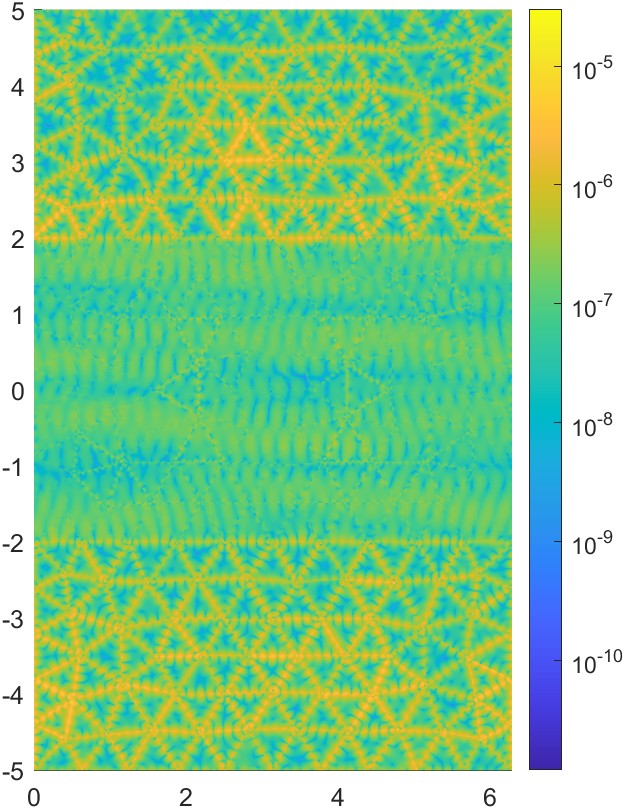}\hfill
\includegraphics[height=.29\textwidth,clip,trim=40 0 0 0]{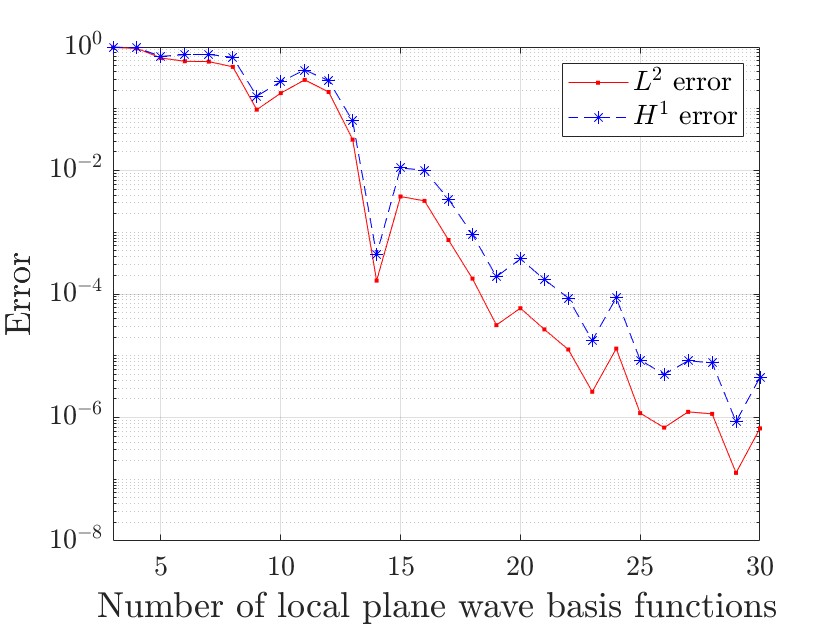}
\caption{Same plots as in Figure~\ref{fig:triple_eps2} for the high-contrast case (ii) ($\varepsilon_{in}=10$, $\theta=-\pi/4$, $h=0.6$).
}
\label{fig:triple_eps10}
\end{figure}

\begin{figure}[htb]
\centering
\includegraphics[height=.35\textwidth]{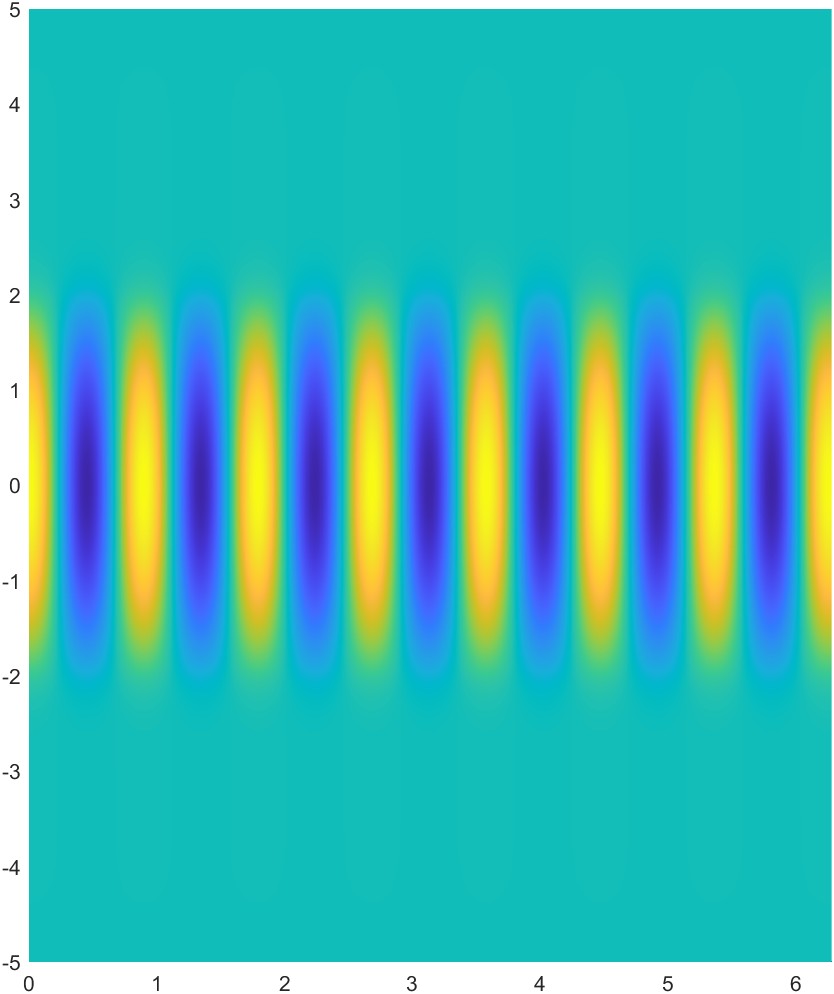}\qquad
\includegraphics[height=.35\textwidth]{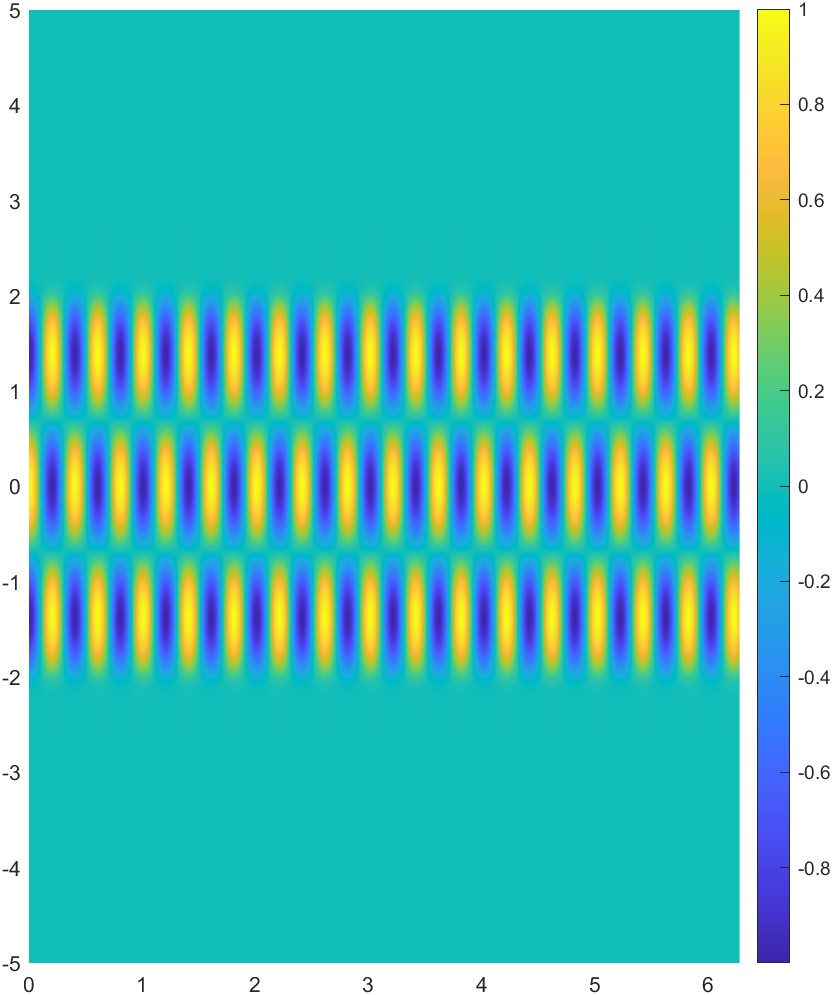}
\caption{Real part of the guided mode $u^{GM}$ in \eqref{trapped} for $\varepsilon_{in} = 2$ (left) and $\varepsilon_{in}=10$ (right).}
\label{fig:u_tr}
\end{figure}

\subsubsection{Problems with non-unique solutions} \label{s:non_uniq}

Given the values of $k,\varepsilon^+,\varepsilon_{in}$ and $d$, there are some values of $\theta$ for which problem \eqref{nontrunc} with the geometry as in \S\ref{s:3regions} ($\varepsilon$ constant in $\{|x_2|>d\}$ and $\{|x_2|<d\}$, $D=\emptyset$) is not well-posed.
Indeed, the non-trapping assumption \eqref{eq:NonTrap} is violated as $\varepsilon$ decreases away from $x_2=0$.
The fields in the problem kernel are the ``guided modes'', i.e.\ quasi-periodic solution of \eqref{nontrunc} that decay exponentially for $|x_2|\to\infty$, \cite[\S4.1]{BonnetBD1994}.
A guided mode can be written as 
\begin{equation} \label{trapped}
u^{GM}(x_1,x_2) = \begin{cases} C\re^{\ri k_1 x_1} \re^{-k_3 x_2} & x_2 > d, \\ 
\re^{\ri k_1 x_1} \cos(k_2 x_2) & -d < x_2 < d, \\ 
C\re^{\ri k_1 x_1} \re^{k_3 x_2} & x_2 < -d, \end{cases}
\end{equation}
where the parameters $C$, $k_1$, $k_2$, and $k_3$ are determined by 
imposing the continuity of $u^{GM}$ and $\partial_{x_2}u^{GM}$ at $x_2=\pm d$, and the Helmholtz equation in $|x_2|<d$ and $|x_2|>d$:
\begin{align*}
&\begin{cases} \; \cos(k_2 d) = C \re^{-k_3 d} \\ \; -k_2 \sin (k_2 d) = -k_3 C\re^{-k_3 d}\end{cases}
&&\implies \qquad k_3 = k_2 \tan(k_2 d),
\\
&\begin{cases} \; k_1^2 + k_2^2 = k^2 \varepsilon_{in} \\ \; k_1^2 - k_3^2 = k^2 \varepsilon^+\end{cases}
&&\implies \qquad k_3^2 =k^2(\varepsilon_{in}-\varepsilon^+)-k_2^2.
\end{align*}
By combining these equations, we obtain the following non-linear equation for $k_2$:
\begin{equation*} 
k_2^2 \big[ 1 + \tan^2(k_2 d) \big] = k^2 ( \varepsilon_{in} - \varepsilon^+ ).
\end{equation*}
This equation can be solved numerically for $k_2$, and the values of the other parameters can then be determined.
To ensure that $u^{GM}$ is quasi-periodic, the incoming wave propagation angle $\theta$ has to satisfy $k_1 = k \cos \theta + \frac{2 \pi}{L} n$ for some $n\in\IN$. 
Figure \ref{fig:u_tr} shows two plots of $u^{GM}$ for $k = 5$, $L = 2\pi$, $H = 5$, $d = 2$, and 
$\varepsilon_{in} = 2$ and $\varepsilon_{in} = 10$, corresponding to 
$k_2=0.713775889382297$, $\theta=-1.563806234657490$ and $k_2=2.279902057511183$, $\theta=-1.441203660687987$, respectively. 
The guided mode is very sensible to the wavenumber $k_2$ and the incident angle $\theta$, so these values cannot be approximated with only few decimal digits; this sensitivity has been observed for other (quasi-)resonances e.g.\ in \cite[Figure~2]{MS17}.

\begin{figure}[htb] \centering
\includegraphics[width=0.45\textwidth]{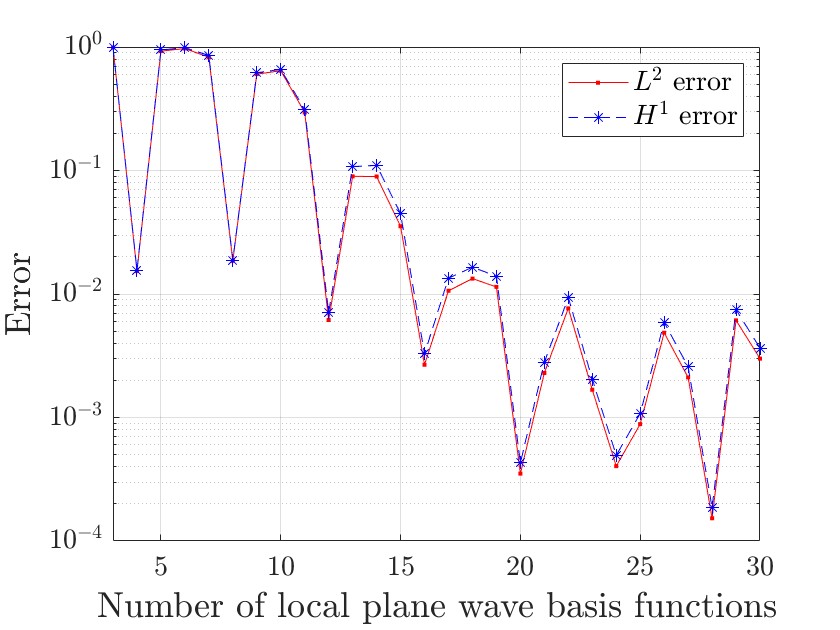}\hfil
\includegraphics[width=0.45\textwidth]{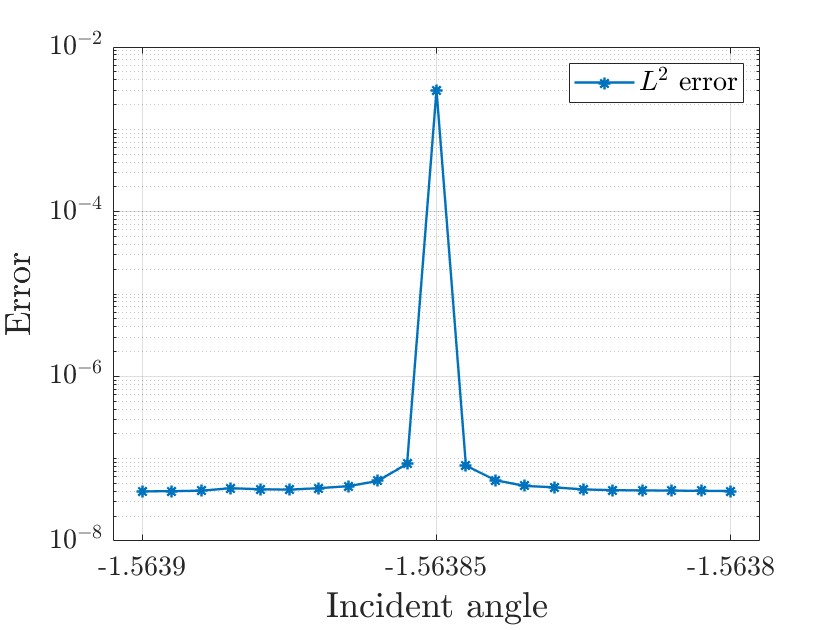}\\
\includegraphics[width=0.45\textwidth]{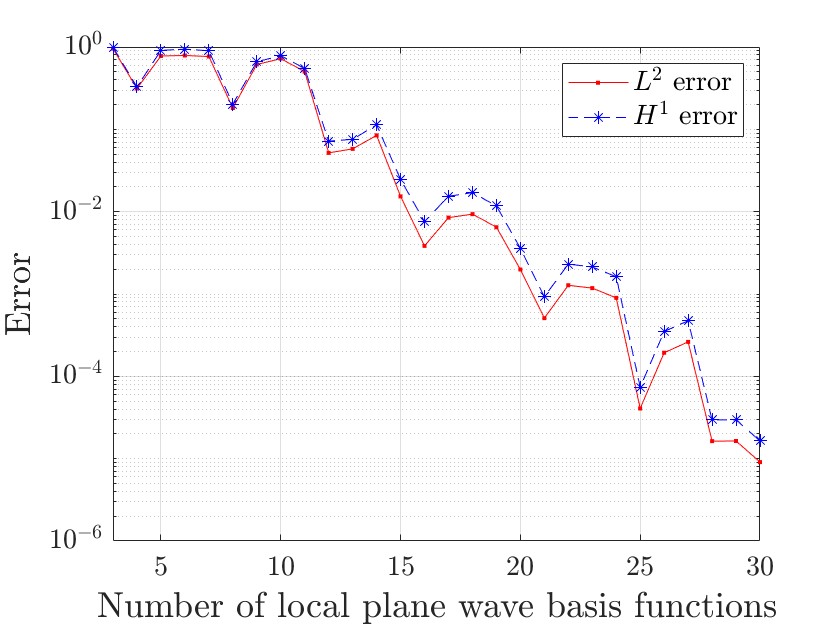}\hfil
\includegraphics[width=0.45\textwidth]{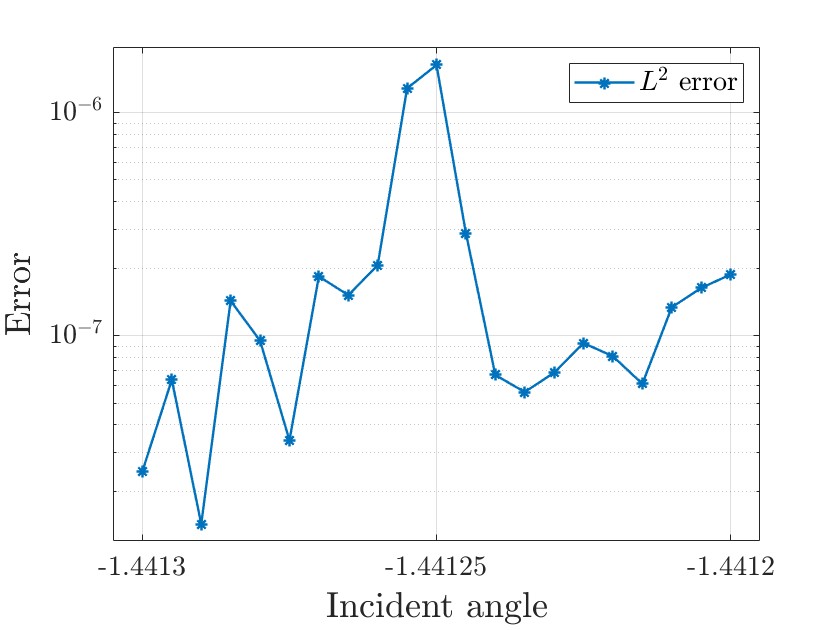}
\caption{
Numerical approximation of the two-interface problem admitting infinite solutions, described in \S\ref{s:non_uniq}.
The errors are computed against the ``unperturbed'' solution \eqref{sol_triple} without guided-mode components.
Left: $p$-convergence for the critical value of $\theta$, relative errors in $\Omega$ (compare against the right panels of Figures~\ref{fig:triple_eps2}--\ref{fig:triple_eps10}, for a non-singular choice of $\theta$).
Right: $L^2(\Omega)$ relative error in dependence of the angle $\theta$ of the incoming wave.
Top: $\varepsilon_{in}=2$; bottom: $\varepsilon_{in}=10$.
}
\label{fig:guided_conv}
\end{figure}

We test the convergence to \eqref{sol_triple}
for both cases $\varepsilon_{in}=2$ and $\varepsilon_{in}=10$, choosing the incident angle corresponding to the singular case.
In Figure \ref{fig:guided_conv} (left panels) we display the convergence plots: we observe that the method does not converge smoothly to the ``unperturbed'' solution.
Plotting the DtN-TDG error (not reported here), we observe that it is a multiple of the guided mode solution \eqref{trapped} in Figure \ref{fig:u_tr}. 
This means that the DtN-TDG method approximates one of the infinite solutions $u=u_0+\mu u^{GM}$, for $\mu \in \IC$ and $u_0$ as in \eqref{sol_triple}, but not necessarily to the solution with $\mu = 0$.
We also plot the DtN-TDG error for a fixed mesh and $p=30$, in dependence of the incident angle $\theta$, allowed to vary in a small interval of length $10^{-4}$ centered at the critical value.
We see that for $\varepsilon_{in}=2$ (top plots), the numerical error decreases by over 4 orders of magnitude when $\theta$ moves away from the critical value.
On the other hand, for $\varepsilon_{in}=10$ (bottom plots) the error oscillates more with respect to $\theta$, but the location of the singular value of $\theta$ is clearly detectable.

\subsection{Interfaces with corners and convergence in the DtN truncation parameter \texorpdfstring{$M$}M}
\label{s:quad}

In this and next sections, we apply the DtN-TDG scheme to problems involving corner singularities.
The exact solutions are not available in closed form, so we compute the errors by testing against DtN-TDG solutions computed with higher numbers of plane waves.

Following \cite{Bao2000}, we consider the domain $\Omega = (0, 2\pi) \times (-2, 2)$
with $\varepsilon=\varepsilon^+ = 1$ in $\{x_2>1\}\cup\{x_2>-1, |x_1-\pi|>\frac\pi2\}$ and $\varepsilon=\varepsilon^- = 1.6 + 0.25\ri$ otherwise, as shown in Figure \ref{fig:quad}.
We take
$k = 4$ and $\theta=-\pi/3$.
We choose the mesh parameter $h = 0.75$, resulting in a triangulation composed of 118 triangles.

Figure \ref{fig:quad} shows the approximate solution and the pointwise error, using as reference a DtN-TDG solution with $p=20$, showing how the error concentrates on the boundaries of the elements adjacent to the material interface corners.
The scatterer profile is also displayed in the solution plot.
The figure also shows the decay in $p$ of the relative $L^2(\Omega)$ and $H^1(\Omega)$ errors.

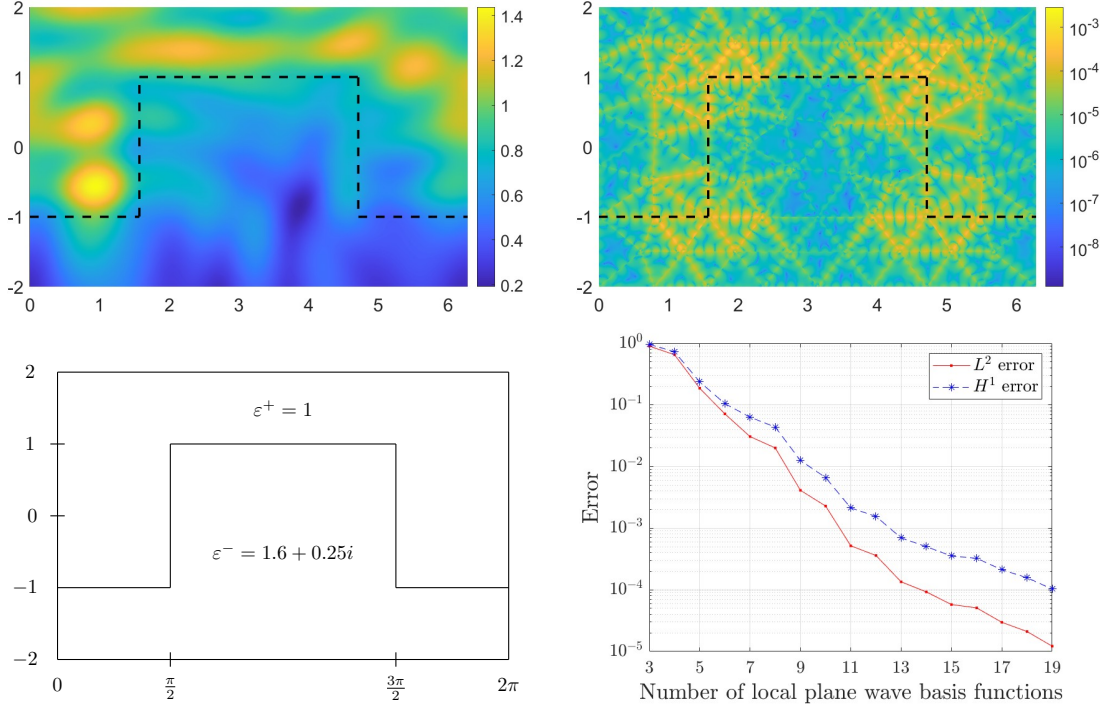
\begin{figure}[htb]
\centering
\includegraphics[height=.28\textwidth]{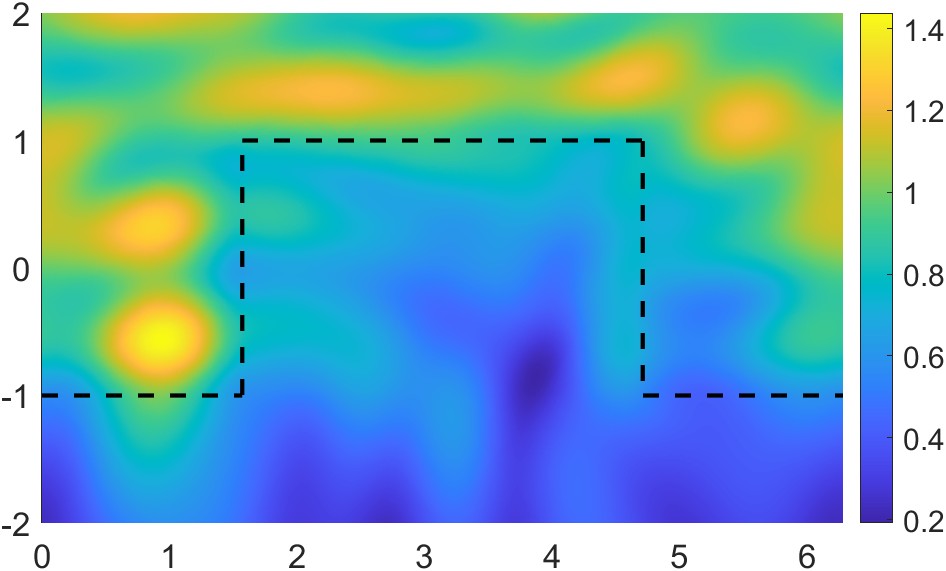}\qquad
\includegraphics[height=.28\textwidth]{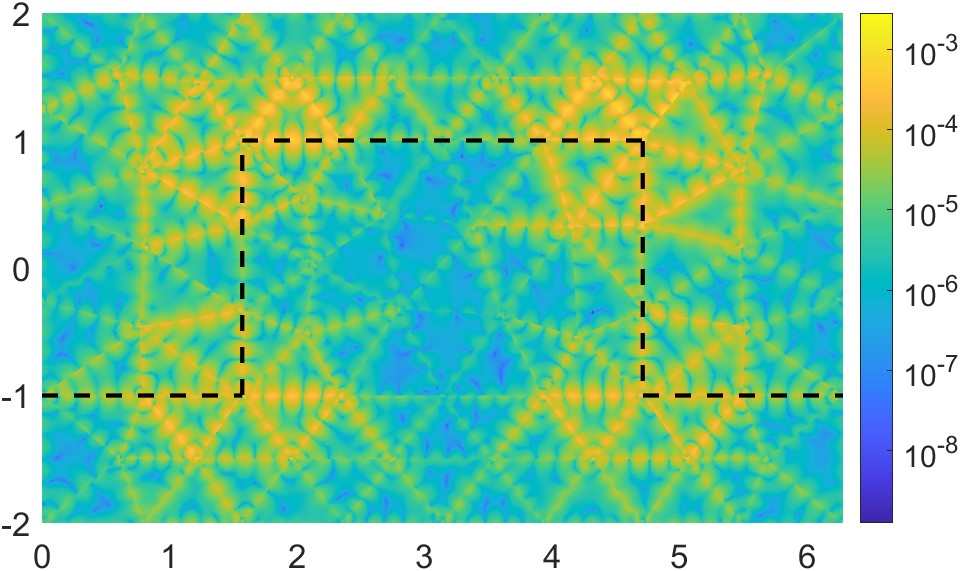}
\\
\begin{tikzpicture}[scale=.95, every node/.style={scale=0.75}]
\node[] at (-0.38,2) {$2$};
\node[] at (-0.38,1) {$1$};
\node[] at (-0.45, -2) {$-2$};
\node[] at (-0.45, -1) {$-1$};
\node[] at (-0.35, 0) {$0$};
\draw (-0.1,1) -- (0.1,1);
\draw (-0.1,0) -- (0.1,0);
\node[] at (6.28, -2.35) {$2\pi$};
\node[] at (0, -2.35) {$0$};
\draw (1.57,-1.9) -- (1.57,-2.1);
\draw (4.71,-1.9) -- (4.71,-2.1);
\node[] at (1.57, -2.4) {$\frac{\pi}{2}$};
\node[] at (4.71, -2.4) {$\frac{3\pi}{2}$};
\node[] at (3.14,1.5) {$\varepsilon^+ = 1$};
\node[] at (3.14,-0.5) {$\varepsilon^- = 1.6 + 0.25i$};
\draw (0,-2) -- (6.28,-2); 
\draw (0,-2) -- (0,2); 
\draw (0,2) -- (6.28,2);
\draw (6.28,-2) -- (6.28,2);
\draw (1.57,-1) -- (1.57,1);
\draw (0,-1) -- (1.57,-1);
\draw (4.71,-1) -- (6.28,-1);
\draw (4.71,-1) -- (4.71,1);
\draw (1.57,1) -- (4.71,1);
\end{tikzpicture}\qquad
\includegraphics[width=.47\textwidth]{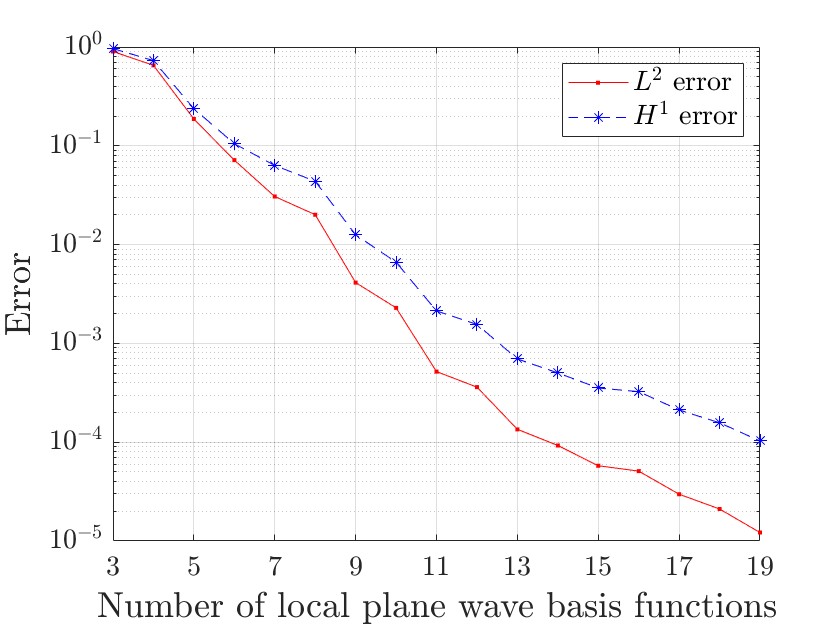}
\caption{The problem with corner singularities described in \S\ref{s:quad}.
Top left: the absolute value of the numerical solution.
Top right: the absolute value of the TDG error for $p=19$ and $M=20$ against a refined numerical solution, in logarithmic color scale.
Bottom left: the geometry of the problem.
Bottom right: the $p$-convergence of the relative $L^2(\Omega)$ and $H^1(\Omega)$ errors.}
\label{fig:quad}
\end{figure}

\subsubsection{Dependence of the error on the number of DtN Fourier modes} 
\label{s:Mconv}

The numerical experiments in \S\ref{s:2regions}--\ref{s:3regions} are independent of the choice of the DtN Fourier truncation parameter $M$, while the example in Figure~\ref{fig:quad} is not.
We thus study how the numerical error depends on $M$ in this case.
With the same configuration of \S\ref{s:quad}, we fix the mesh and the number of plane wave directions $p$, and increase the order of truncation $M$ in \eqref{DtN_trunc}, corresponding to choosing $2M+1$ Fourier modes on $\GpmH$. 
In Figure~\ref{fig:MConv}, we show the results for different values of $k$:
the error decreases quickly and then flattens, showing that the DtN-truncation error is dominated by the TDG discretization one.
As expected, the convergence is slower for larger wavenumbers $k$.
This is in complete agreement with the DtN-TDG results for the exterior scattering problem in \cite[Fig.~2]{Kapita2018}.

\begin{figure}[htb]
\centering
\includegraphics[width=.6\textwidth]{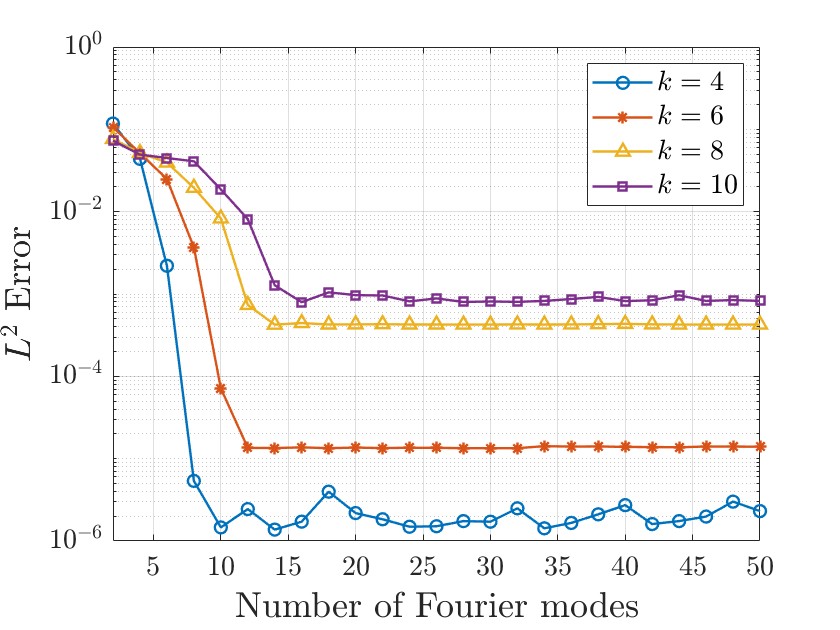}
\caption{Relative $L^2(\Omega)$ error against the truncation order $M$ for the DtN operator, for $k= 4, 6, 8, 10$, as described in \S\ref{s:Mconv}.}
\label{fig:MConv}
\end{figure}

\begin{figure}[htb]
\hbox to \textwidth{
\includegraphics[height=.4\textwidth]{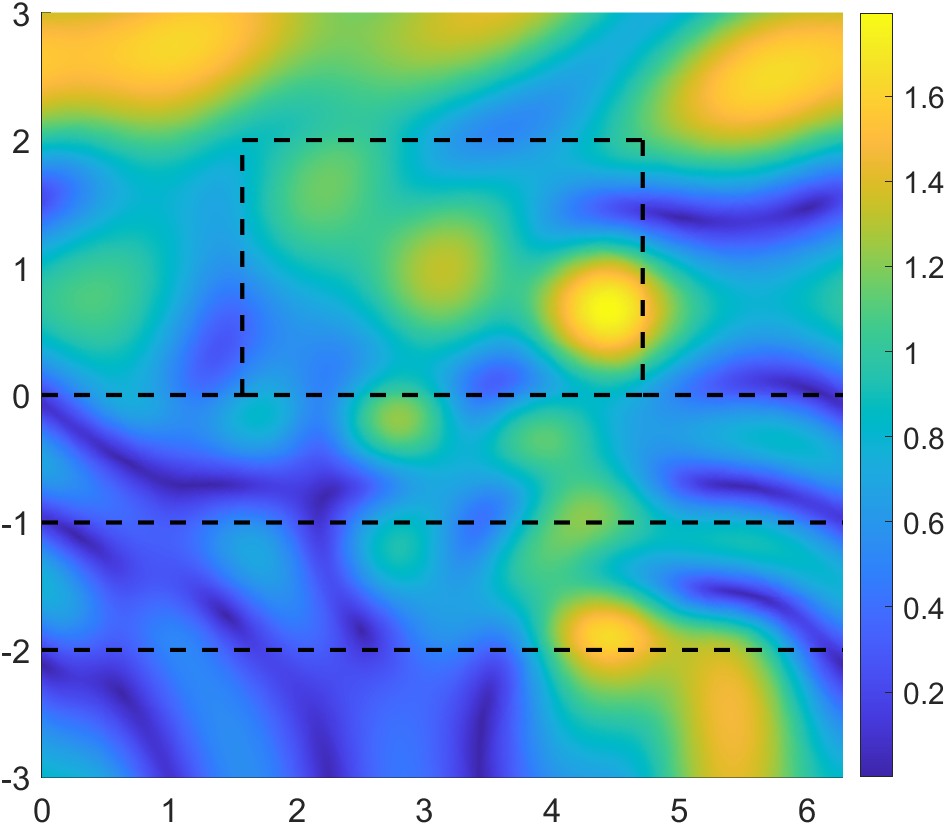}\hfil
\includegraphics[height=.4\textwidth]{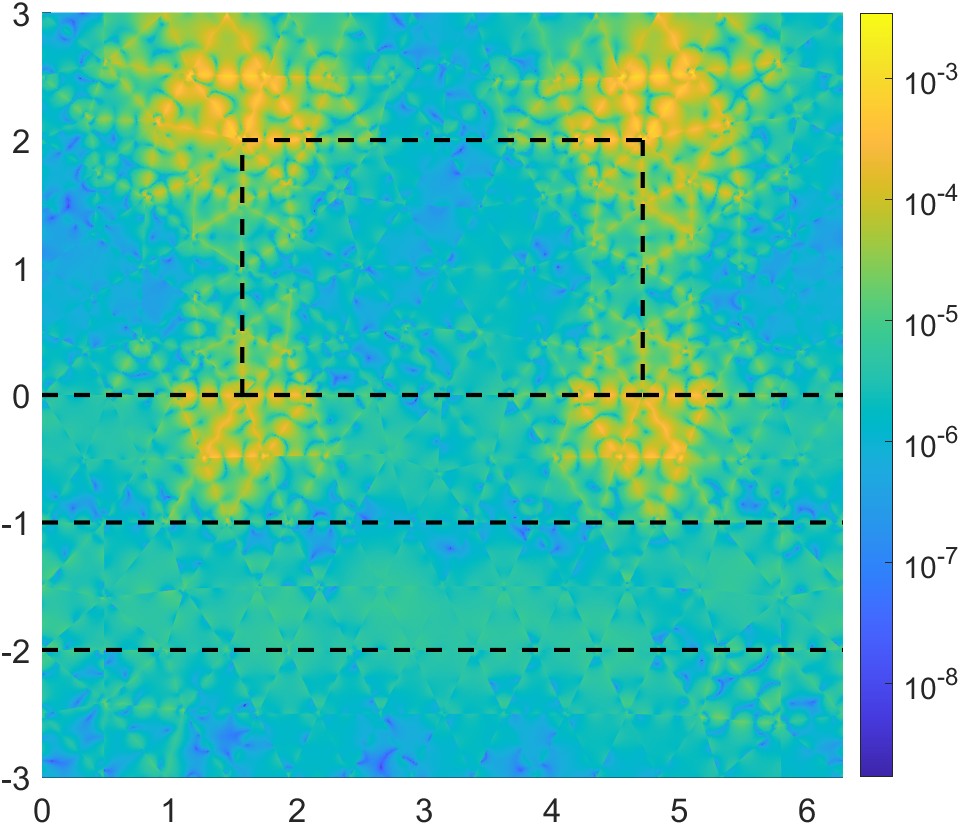}}
\caption{Absolute value of the numerical solution (left), and absolute value of the relative error (right, logarithmic color scale) for the problem with five materials in \S\ref{s:bao} and $p=14$.}
\label{fig:bao}
\end{figure}
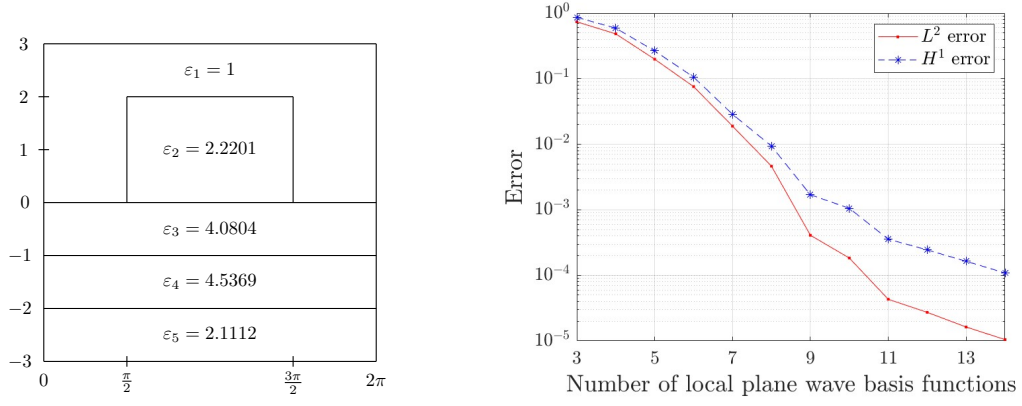
\begin{figure}[htb]
\centering
\noindent\hbox to .95\textwidth{\begin{tikzpicture}[scale=0.7, every node/.style={scale=0.7}]
	\node[] at (-0.38,3) {$3$};
	\node[] at (-0.38,2) {$2$};
	\node[] at (-0.38,1) {$1$};
	\node[] at (-0.45,-3) {$-3$};
	\node[] at (-0.45, -2) {$-2$};
	\node[] at (-0.45, -1) {$-1$};
	\node[] at (-0.35, 0) {$0$};
	\draw (-0.1,1) -- (0.1,1);
	\draw (-0.1,2) -- (0.1,2);
	\node[] at (6.28, -3.35) {$2\pi$};
	\node[] at (0, -3.35) {$0$};
	\draw (1.57,-2.9) -- (1.57,-3.1);
	\draw (4.71,-2.9) -- (4.71,-3.1);
	\node[] at (1.57, -3.4) {$\frac{\pi}{2}$};
	\node[] at (4.71, -3.4) {$\frac{3\pi}{2}$};
	\node[] at (3.14,2.5) {$\varepsilon_1 = 1$};
	\node[] at (3.14,1) {$\varepsilon_2 = 2.2201$};
	\node[] at (3.14,-0.5) {$\varepsilon_3 = 4.0804$};
	\node[] at (3.14,-1.5) {$\varepsilon_4 = 4.5369$};
	\node[] at (3.14,-2.5) {$\varepsilon_5 = 2.1112$};
	\draw (0,-3) -- (6.28,-3); 
	\draw (0,-3) -- (0,3); 
	\draw (0,3) -- (6.28,3);
	\draw (6.28,3) -- (6.28,-3);
	\draw (0,-2) -- (6.28,-2);
	\draw (0,-1) -- (6.28,-1);
	\draw (0,0) -- (6.28,0);
	\draw (1.57,0) -- (1.57,2);
	\draw (4.71,0) -- (4.71,2);
	\draw (1.57,2) -- (4.71,2);
\end{tikzpicture}\hfil\includegraphics[width=.5\textwidth]{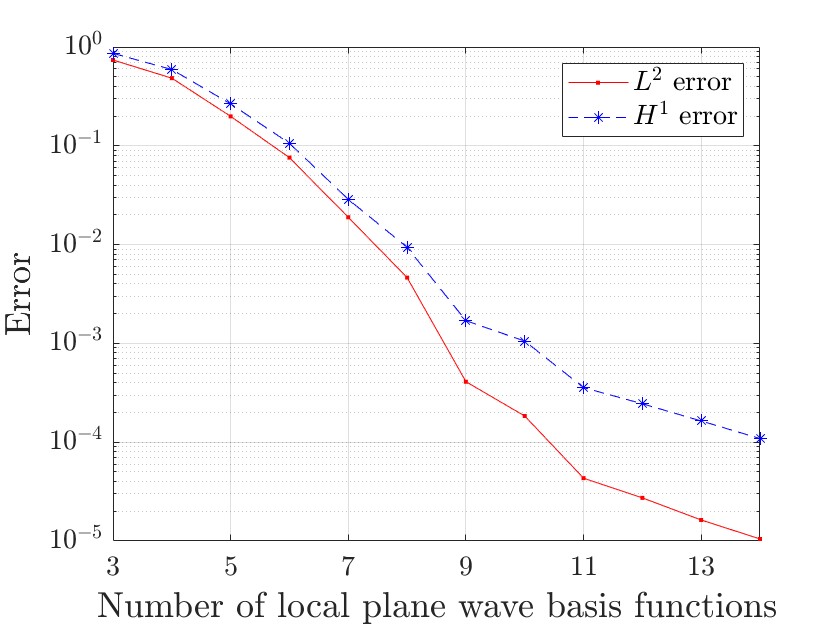}}
\caption{The domain (left) and the relative errors in $L^2(\Omega)$ and $H^1(\Omega)$ norms, for $h=0.5$ and $p \in \{3,\ldots,14\}$,	for the problem in \S\ref{s:bao}.}
\label{fig:bao_err}
\end{figure}

\subsection{Problem with more than two materials}
\label{s:bao}

We partition the domain $\Omega = (0, 2\pi) \times (-3, 3)$ in five polygonal regions, as depicted in Figure~\ref{fig:bao_err}.
This grating structure is used in optical filters and guided mode resonance devices, \cite[Ex.~3]{Bao2000}.
The parameters are: $\varepsilon_1 = \varepsilon^+ = 1$, $\varepsilon_2 = 1.49^2$, $\varepsilon_3 = 2.13^2$, $\varepsilon_4 = 2.02^2$, $\varepsilon_5 = \varepsilon^- = 1.453^2$, $\theta = -\pi/4$, and $k = 2$.
The mesh parameter is chosen as $h = 0.5$, resulting in a triangulation of 332 triangles.

Figure \ref{fig:bao} shows the approximate solution and the error against a refined solution obtained using $p = 15$. 
Figure \ref{fig:bao_err} presents the convergence of the $L^2(\Omega)$ and $H^1(\Omega)$ error norms, relative to the same refined solution.
For $p\gtrsim 9$ the convergence slows down, possibly because of the solution singularities near the polygon vertices.

To validate the method and its implementation, the problem was also solved on the extended domain $\Omega^* = (0, 4\pi) \times (-3, 3)$, and the corresponding solution compared with the solution computed on $\Omega$ and extended by quasi-periodicity.
Using the same number $p=15$ of plane wave functions per element, the relative $L^2(\Omega)$ and $H^1(\Omega)$ norm errors obtained are $3.373 \times 10^{-6}$ and $8.066 \times 10^{-5}$, respectively.

\subsection{Impenetrable obstacles} \label{s:dir}
So far we have only considered a domain composed of different materials transparent to waves.
Now we include a rectangular impenetrable obstacle $D$, equipped with 
Dirichlet boundary conditions.
We consider the domain $\Omega = (0,2\pi) \times (-5,5) \setminus [\frac{2}{3}\pi, \frac{4}{3}\pi] \times [-1,1]$, wavenumber $k=5$, and mesh width $h=0.75$, resulting in a triangulation of $236$ triangles.
We present two examples with constant and variable $\varepsilon$:
\begin{enumerate}[(i)]
\item $\varepsilon=1$ in $\Omega$ and incident angle $\theta=-\pi/4$ (Figure~\ref{fig:dir});
\item $\varepsilon$ takes three different values in $\{x_2>3\}$, $\{|x_2|<3\}$, $\{x_2<-3\}$, and $\theta=-\pi/3$ (Figure~\ref{fig:dir_triple}).
\end{enumerate}
Case (i) satisfies the non-trapping assumption \eqref{eq:NonTrap}, while in case (ii) the permittivity $\varepsilon$ takes a complex value in one of the regions.
In both cases, we compute the relative error with respect to a refined solution computed with $p=20$.
The results are comparable to those without impenetrable obstacles.
The error is largest near the corners of the scatterer.

\begin{figure}[htb]
\centering
\noindent\hbox to \textwidth{\begin{tikzpicture}[scale=0.88, every node/.style={scale=0.7}]
	\node[] at (-0.38,4) {$5$};
	\node[] at (-0.5,0) {$-5$};
	\draw (1,-0.1) -- (1,0.1);
	\draw (2,-0.1) -- (2,0.1);
	\draw (-0.1,1.5) -- (0.1,1.5);
	\draw (-0.1,2.5) -- (0.1,2.5);
	\node[] at (-0.4, 2.5) {$1$};
	\node[] at (-0.5, 1.5) {$-1$};
	\node[] at (0, -0.2) {$0$};
	\node[] at (3, -0.2) {$2 \pi$};
	\node[] at (1.5,3.35) {$\varepsilon = 1$};
	\node[] at (1, -0.3) {$\frac{2}{3}\pi$};
	\node[] at (2, -0.3) {$\frac{4}{3}\pi$};
	\fill[gray!40!white, draw=black] (1,1.5) rectangle (2,2.5);
	\draw (1.5,2)node{$D$};
	\draw (0,0) -- (3,0); 
	\draw (0,0) -- (0,4); 
	\draw (0,4) -- (3,4);
	\draw (3,4) -- (3,0);
\end{tikzpicture}\hfill
\includegraphics[height=.27\textwidth]{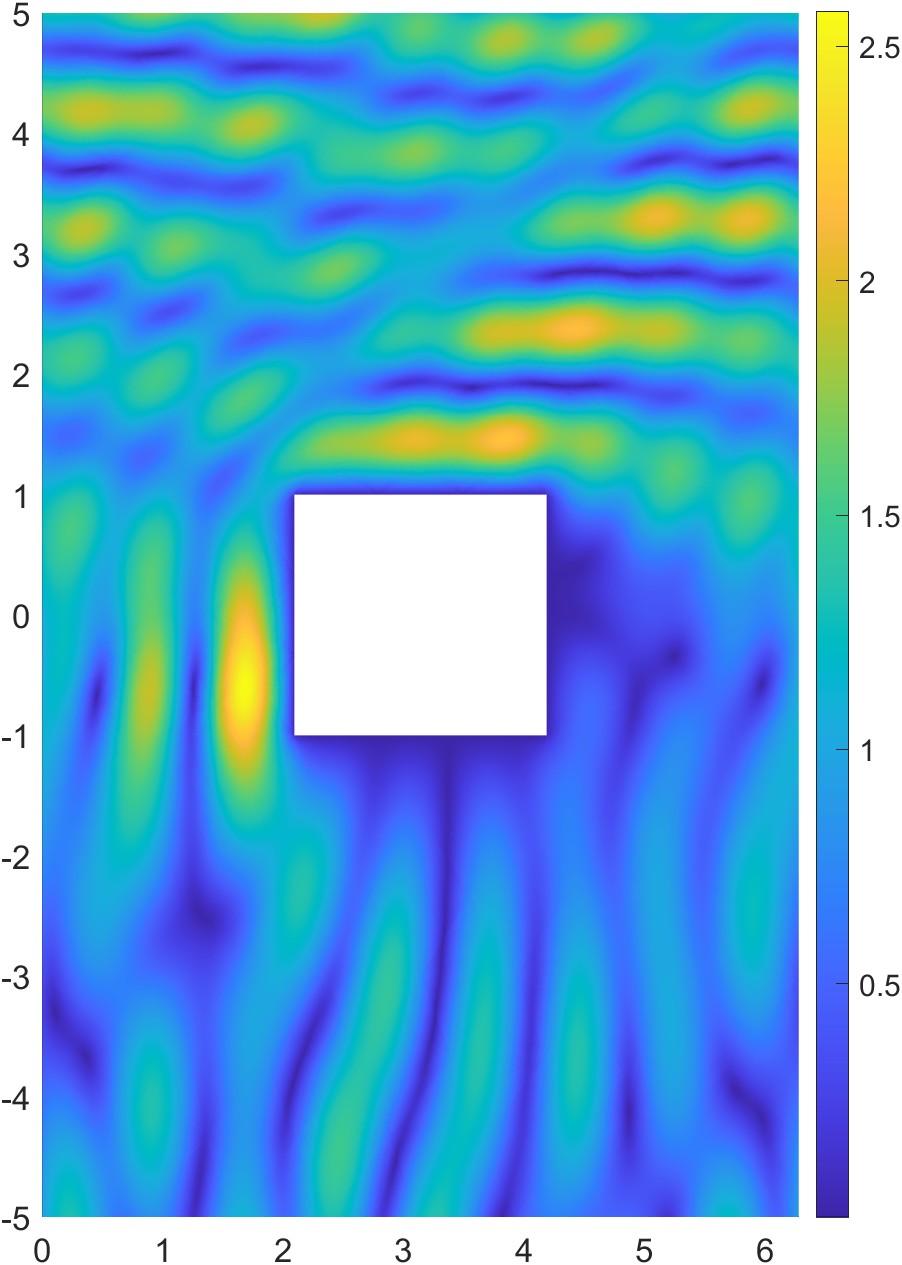}\hfill
\includegraphics[height=.27\textwidth]{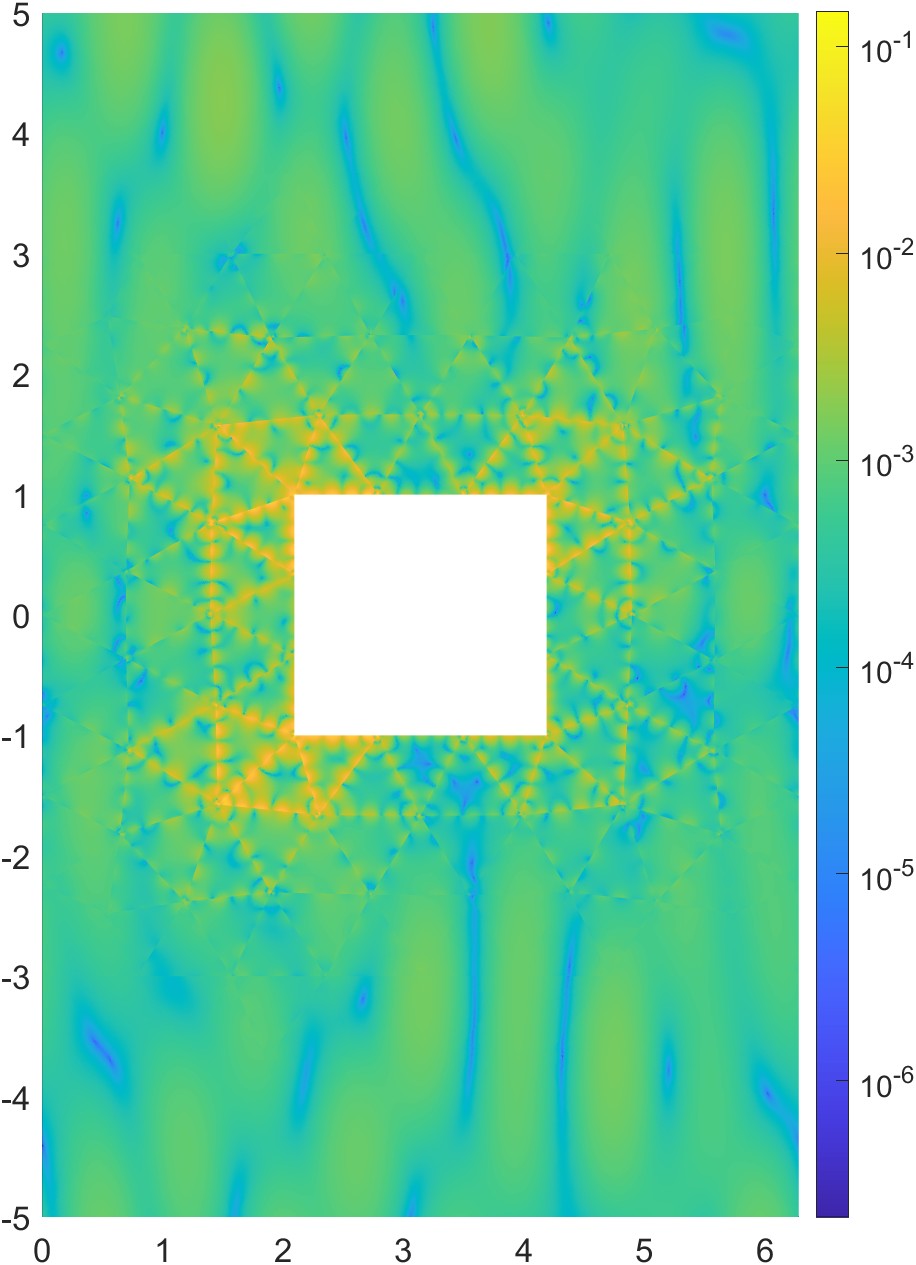}\hfill
\includegraphics[height=.28\textwidth,clip,trim=40 0 0 0]{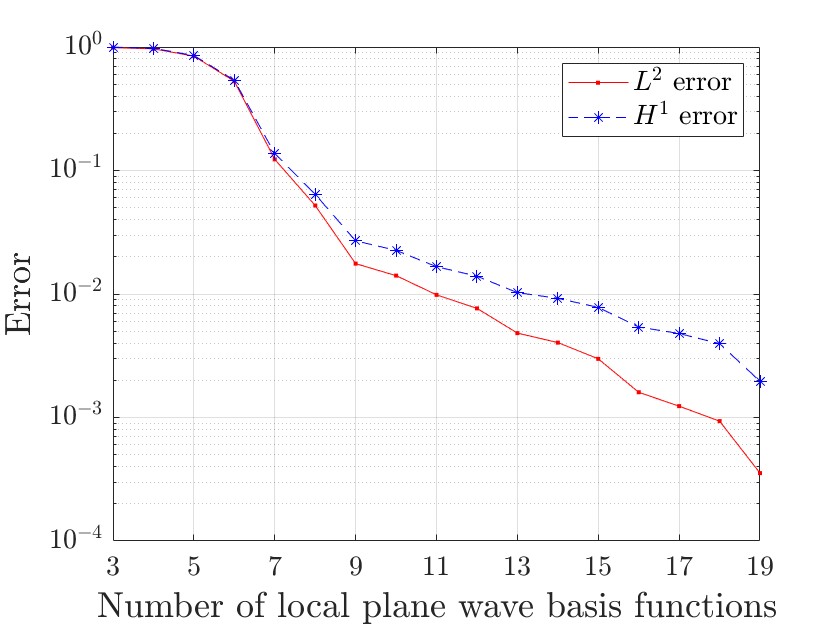}}
\caption{Left to right: domain sketch, absolute value of the solution and of the error (in logarithmic color scale), and relative errors in $L^2(\Omega)$ and $H^1(\Omega)$ norm for $h=0.75$ and $p \in \{3,\ldots,19\}$ 
for example (i) in \S\ref{s:dir}.}
\label{fig:dir}
\end{figure}
\begin{figure}[htb]
\centering
\noindent\hbox to \textwidth{\begin{tikzpicture}[scale=0.88, every node/.style={scale=0.7}]
	\node[] at (-0.38,4) {$5$};
	\node[] at (-0.5,0) {$-5$};
	\node[] at (-0.38,3.25) {$3$};
	\node[] at (-0.5,0.75) {$-3$};
	\draw (1,-0.1) -- (1,0.1);
	\draw (2,-0.1) -- (2,0.1);
	\draw (-0.1,1.5) -- (0.1,1.5);
	\draw (-0.1,2.5) -- (0.1,2.5);
	\node[] at (-0.4, 2.5) {$1$};
	\node[] at (-0.5, 1.5) {$-1$};
	\node[] at (0, -0.2) {$0$};
	\node[] at (3, -0.2) {$2 \pi$};
	\node[] at (1.5,3.6) {$\varepsilon^+ = 1$};
	\node[] at (1.5,2.83) {$\varepsilon_{in} = 2$};
	\node[] at (1.5,0.37) {$\varepsilon^- = 1.6 + 0.25i$};
	\node[] at (1, -0.3) {$\frac{2}{3}\pi$};
	\node[] at (2, -0.3) {$\frac{4}{3}\pi$};
	\fill[gray!40!white, draw=black] (1,1.5) rectangle (2,2.5);
	\draw (1.5,2)node{$D$};
	\draw (0,0) -- (3,0); 
	\draw (0,0) -- (0,4); 
	\draw (0,4) -- (3,4);
	\draw (3,4) -- (3,0);
	\draw (0,3.25) -- (3,3.25); 
	\draw (0,0.75) -- (3,0.75);
\end{tikzpicture}\hfill
\includegraphics[height=.27\textwidth]{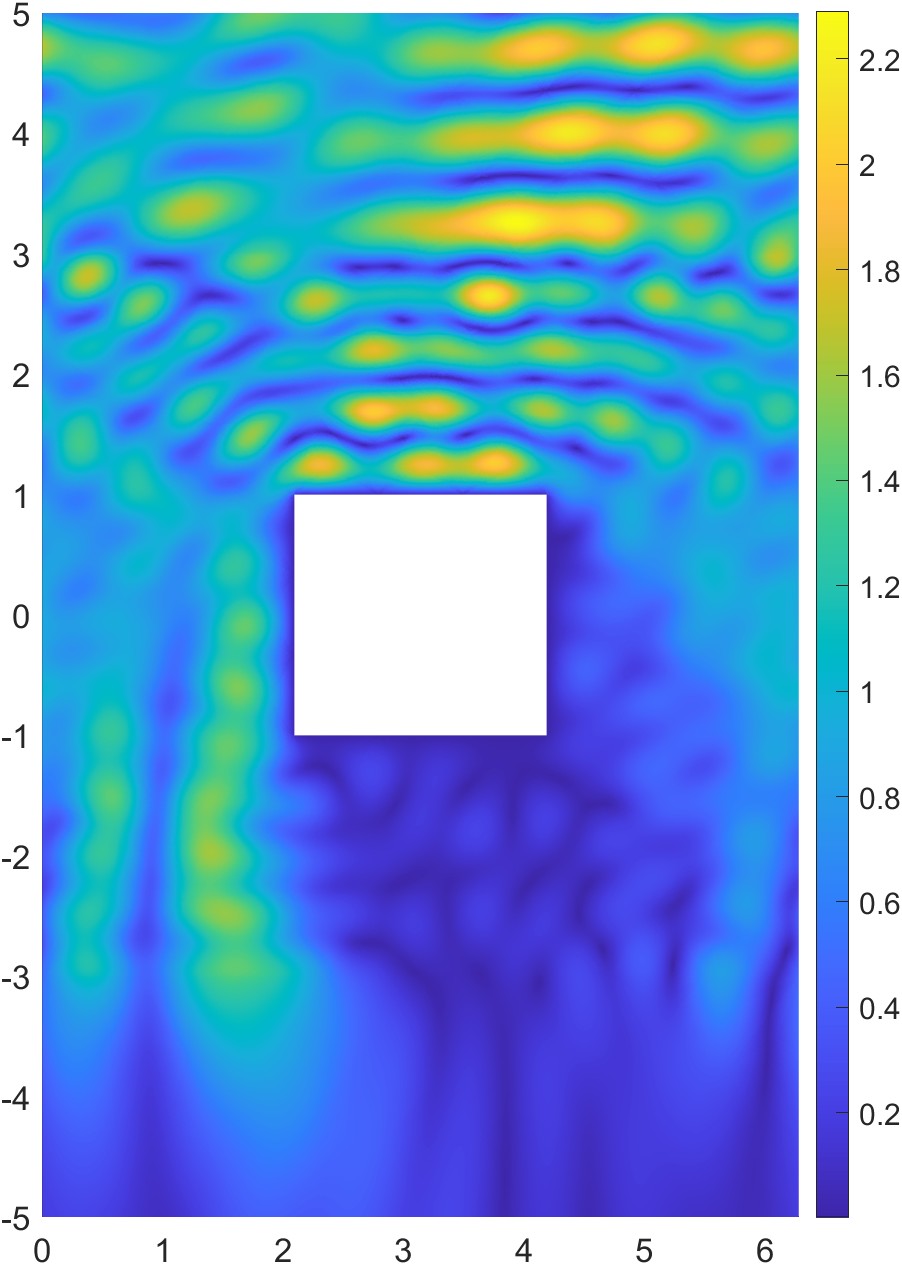}\hfill
\includegraphics[height=.27\textwidth]{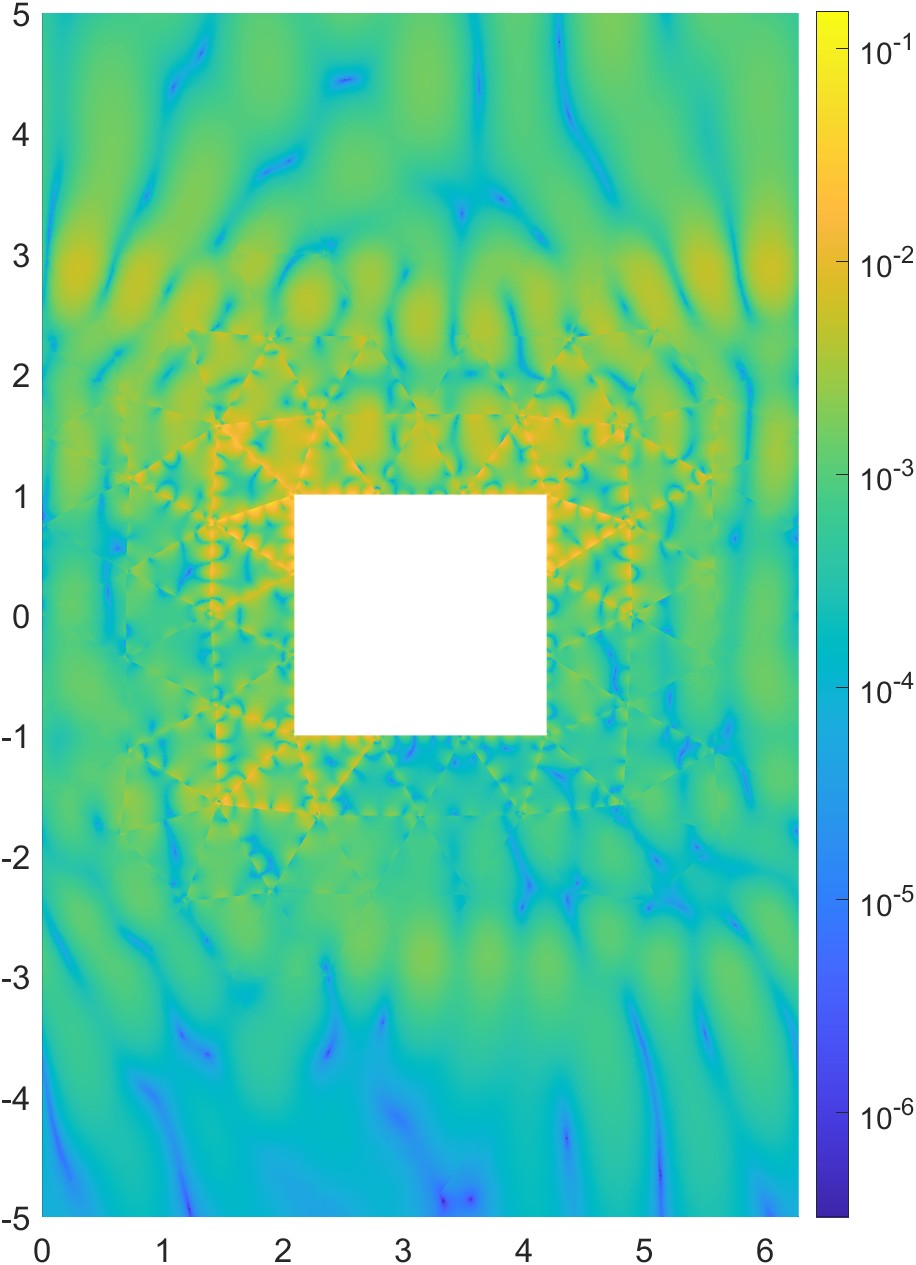}\hfill
\includegraphics[height=.28\textwidth,clip,trim=40 0 0 0]{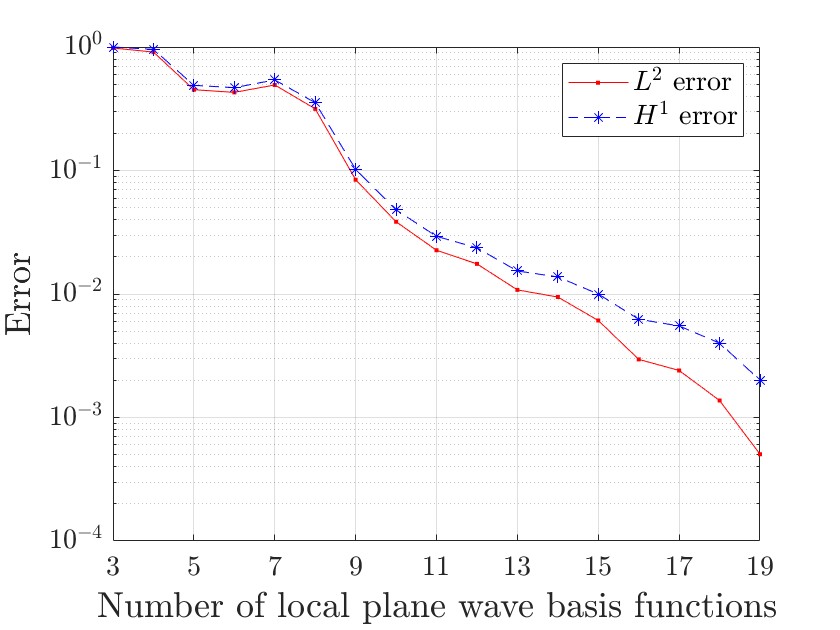}}
\caption{Same as Figure~\ref{fig:dir} for example (ii) in \S\ref{s:dir}.}
\label{fig:dir_triple}
\end{figure}
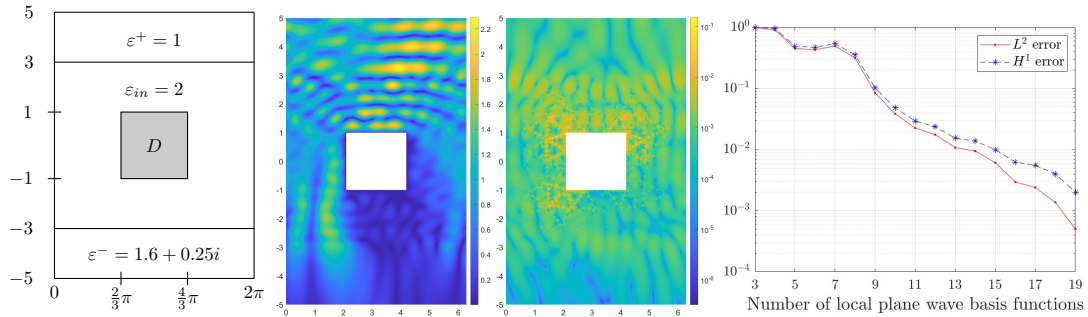


\section*{Declarations}
\subsection*{Acknowledgements} 
The authors are grateful to the anonymous referees for many constructive comments that improved the clarity of the paper.

\subsection*{Funding} This work was supported by GNCS--INDAM and by the PRIN project ``ASTICE'' (202292JW3F), funded by the European Union--NextGenerationEU.

\subsection*{Conflict of interest} The authors have no relevant financial or non-financial interests to disclose.

\subsection*{Ethics approval and consent to participate} Not applicable.

\subsection*{Consent for publication} All authors approved the version to be published and consent the publication.

\subsection*{Data availability} Not applicable.

\subsection*{Materials availability} Not applicable.

\subsection*{Code availability} The code implemented is available at \url{https://github.com/Arma99dillo/DtN-TDG}

\subsection*{Author contribution} All authors contributed to the paper conception and design. Computations and code development were performed by Armando Maria Monforte. The first draft of the manuscript was written by all authors and all authors commented on previous versions of the manuscript. All authors read and approved the final manuscript.

\addcontentsline{toc}{section}{References}
\printbibliography
\end{document}